\documentclass[11pt,a4paper]{amsart}
\usepackage[margin=27mm]{geometry}
\usepackage{amsmath,amssymb,amsthm,mathtools}
\usepackage[hidelinks]{hyperref}
\hypersetup{pdftitle={A Quadratic Bound for Automorphism Groups of Curves of Positive p-Rank in Odd Characteristic},pdfauthor={Saeed Tafazolian},pdfsubject={Automorphism groups and p-rank of algebraic curves}}
\usepackage{enumitem,booktabs,array,needspace}
\newtheorem{theorem}{Theorem}[section]
\newtheorem{proposition}[theorem]{Proposition}
\newtheorem{lemma}[theorem]{Lemma}
\newtheorem{corollary}[theorem]{Corollary}
\theoremstyle{remark}
\newtheorem{remark}[theorem]{Remark}

\newcommand{\Aut}{\operatorname{Aut}}
\newcommand{\Syl}{\operatorname{Syl}}
\newcommand{\Fix}{\operatorname{Fix}}
\newcommand{\PP}{\mathbb P}
\newcommand{\FF}{\mathbb F}

\newcommand{\Rad}{\operatorname{Rad}}
\newcommand{\core}{\operatorname{core}}
\newcommand{\PSL}{\operatorname{PSL}}
\newcommand{\PSU}{\operatorname{PSU}}
\newcommand{\PGL}{\operatorname{PGL}}

\newcommand{\soc}{\operatorname{soc}}
\newcommand{\meo}{\operatorname{meo}}
\newcommand{\kk}{\mathbb K}

\title[A quadratic bound for automorphism groups]{A Quadratic Bound for Automorphism Groups of Curves of Positive $p$-Rank in Odd Characteristic}
\author{Saeed Tafazolian}
\address{Department of Mathematics, IMECC, University of Campinas (UNICAMP), Campinas, SP, Brazil}
\email{saeed@unicamp.br}
\date{}

\begin{document}
\begin{abstract}
Let $X$ be a nonsingular projective curve of genus $g\ge2$ over an
algebraically closed field of odd characteristic $p$. Giulietti and
Korchm\'aros established a quadratic genus bound with constant $900$ forcing
zero $p$-rank for curves of even genus. We extend this bound to odd genus:
for every $g\ge2$, positive $p$-rank implies
$|\Aut(X)|<900g^2$. Starting from Montanucci's two-short-orbit reduction,
we combine ramification estimates and a minimum-genus argument to reduce a
counterexample to an almost-simple group acting on a cover of the projective
line with two branch points. The main ingredients are an intrinsic
description of the second ramification group via Hasse--Arf and genus
bounds for intermediate quotients associated with parabolic subgroups.
These, together with order and local-structure estimates, exclude every
possible simple socle. The proof uses the classification of finite simple
groups.
\end{abstract}
\maketitle

\section{Introduction}\label{reduction:sec:roadmap}
Let $\kk$ be an algebraically closed field of characteristic $p>2$, and let
$X/\kk$ be a nonsingular projective curve of genus $g\ge2$. We write
$\gamma(X)$ for the $p$-rank of $X$.

In positive characteristic, wild ramification allows automorphism groups
of curves to exceed the classical Hurwitz bound $84(g-1)$. Stichtenoth's
quartic bound and Henn's cubic bound describe successive restrictions on
this phenomenon; in particular, $|\Aut(X)|>8g^3$ forces
$\gamma(X)=0$; see \cite[Chapter~11, especially Theorem~11.127]{HKT}.
Nakajima's bounds for $p$-subgroups and ordinary curves further show how
the $p$-rank constrains the automorphism group \cite{Nakajima}.
These results lead to the problem of finding a uniform quadratic bound
in the genus for curves of positive $p$-rank.

Giulietti--Korchm\'aros established such a bound, with constant $900$, for
curves of even genus \cite{GK19}. They emphasized the remaining difficulty
in their abstract: ``The odd genus case appears to be much more difficult.''
The main result of this paper removes the even-genus restriction and
establishes the same quadratic threshold for all genera in odd
characteristic. In particular, it settles the odd-genus case of this
quadratic-bound problem.

Montanucci proved that a group of order greater than $336g^2$ has exactly
two short orbits, one wild and one tame, and obtained a zero-$p$-rank
criterion at the threshold $900g^2$ under additional hypotheses on the
wild-inertia quotient and two-point stabilizers \cite[Theorem~1.3]{Mon22}.
The theorem below removes these local hypotheses as well. A related
problem, bounding the group order in terms of the $p$-rank rather than
the genus, was studied by Giulietti--Korchm\'aros--Timpanella
\cite{GKT26}. Their bound assumes that the quotient by the Sylow
$p$-subgroup of a point stabilizer is rational and has an exception when
the orbit of that point has at most two elements. No such rationality
assumption is imposed here.

\begin{theorem}[Main theorem]\label{thm:main}
Let $X$ be a nonsingular projective curve of genus $g\ge2$ over an
algebraically closed field of odd characteristic. If $\gamma(X)>0$, then
for every subgroup $G\le\Aut(X)$,
\[
 |G|<900g^2.
\]
Equivalently, if $|\Aut(X)|\ge900g^2$, then $\gamma(X)=0$.
\end{theorem}

We retain the constant $900$ to match the threshold in \cite{GK19}.
The proof does not determine the optimal constant or establish that the
quadratic exponent is sharp.

Two features of the proof are particularly useful in treating odd genus.
For a counterexample, let $P$ lie on the wild short orbit and set
$S=G_P^{(1)}$, the wild inertia subgroup. After the rational quotient
configurations have been analyzed, every remaining case satisfies the
numerical constraints denoted by $(\Sigma)$ in Section~\ref{sec:sigma}.
The resulting bound on the different, together with Hasse--Arf, gives
\[
 S_P^{(2)}=S'=\Phi(S)\ne1,
\]
where $S_P^{(2)}$ is the second lower ramification group, $S'$ the
commutator subgroup, and $\Phi(S)$ the Frattini subgroup. Thus a geometric
ramification subgroup is determined by the abstract group $S$; see
Lemma~\ref{normalizer:lem:local}. The second ingredient is a genus
inequality for intermediate quotients, obtained from the action of inertia
on cosets. Applied to parabolic subgroups, it combines the geometric
ramification constraints with Levi orders and torus bounds and excludes
the remaining groups of Lie type; see Lemma~\ref{par:lem:genus} and
Corollary~\ref{par:cor:test}.

The group-theoretic argument uses the classification of finite simple
groups (CFSG). Its principal structural inputs include the results of
Parker--Stroth on strongly $p$-embedded subgroups \cite{PS11},
Flores--Foote on strongly closed subgroups \cite{FF}, and
Guralnick--Malmskog--Pries on TI Sylow subgroups
\cite[Theorem~3.16 and Remark~3.17]{GMP12}, together with the Schreier
theorem on outer automorphism groups and the final analysis of simple
socles. The normal-subgroup reductions also use the coprime-automorphism
theorem of Wang--Chen \cite{WC93} and the results of
Glauberman--Guralnick--Lynd--Navarro \cite{GGLN}, including the form of
Gross's theorem recalled there.

The following table records the main steps of the proof and their locations.
Throughout the argument, a counterexample is chosen with minimum genus
among all counterexamples to Theorem~\ref{thm:main}; this choice includes
no restriction on quotient genus, orbit structure, or socle.

\begin{center}
\small
\renewcommand{\arraystretch}{1.15}
\begin{tabular}{>{\raggedright\arraybackslash}p{0.20\textwidth}>{\raggedright\arraybackslash}p{0.52\textwidth}>{\raggedright\arraybackslash}p{0.18\textwidth}}
\toprule
Stage & Main conclusion & Location\\
\midrule
Large-group geometry
& The quotient $X/G$ is rational and $G$ has exactly two short orbits, one wild and one tame.
& Section~\ref{sec:prelim}\\
Quotient by wild inertia
& The rational and positive-genus quotient cases are analyzed; every remaining case satisfies the common numerical constraints $(\Sigma)$.
& Sections~\ref{sec:caseA}--\ref{sec:sigma}\\
Local ramification
& The different bound gives $S_P^{(2)}=S'=\Phi(S)\ne1$.
& Section~\ref{sec:sigma}\\
Normal subgroups
& The action on the wild orbit is faithful and the soluble radical is trivial. There is a unique minimal normal subgroup, transitive on the wild orbit.
& Sections~\ref{sec:faithfulness}--\ref{sec:imprimitive-normal}\\
Almost-simple reduction
& The socle $T$ is simple, $T\le G\le\Aut(T)$, and $X\to X/T\cong\PP^1$ has exactly two branch points.
& Section~\ref{sec:further-residual}\\
Simple-socle exclusion
& Order, local-structure, and intermediate-quotient arguments exclude all possible simple socles.
& Sections~\ref{scr:sec:screening}--\ref{par:sec:new}\\
\bottomrule
\end{tabular}
\end{center}

\section{Preliminaries and notation}\label{sec:prelim}

We collect here the geometric and elementary group-theoretic facts that are
used repeatedly.  Classification results that are needed only for one branch
of the argument are stated at the point where they are applied.

\subsection{Ramification and global formulas}

Let $H\le\Aut(X)$ be finite and let $P\in X$.  We write
$H_P^{(i)}$ for the lower ramification groups at $P$, so that
$H_P^{(0)}=H_P$ and $H_P^{(1)}$ is the wild inertia subgroup.  The point
stabilizer has the standard form
\[
 H_P=H_P^{(1)}\rtimes C,
 \qquad C\ \text{cyclic},\qquad (|C|,p)=1;
\]
see \cite[Theorem~11.49]{HKT}.  In particular $H_P^{(1)}$ is the unique
Sylow $p$-subgroup of $H_P$.

An $H$-orbit is \emph{short} if its point stabilizer is nontrivial.  A short
orbit is \emph{tame} if its point stabilizers have order prime to $p$, and
\emph{wild} otherwise.  The action of $H$ is called tame if every point
stabilizer has order prime to $p$.  The different exponent at $P$ is
\[
 d_P(H)=\sum_{i\ge0}\bigl(|H_P^{(i)}|-1\bigr).
\]
Thus $d_P(H)=|H_P|-1$ at a tame point.  If $Y=X/H$, the Hurwitz formula is
\begin{equation}\label{prelim:eq:hurwitz}
 2g(X)-2
 =|H|\bigl(2g(Y)-2\bigr)+\sum_{P\in X}d_P(H);
\end{equation}
see \cite[Theorem~11.72]{HKT}.

For a $p$-group $U\le\Aut(X)$ we use the Deuring--Shafarevich formula
\begin{equation}\label{prelim:eq:ds}
 \gamma(X)-1
 =|U|\bigl(\gamma(X/U)-1\bigr)
  +\sum_{\mathcal O}\bigl(|U|-|\mathcal O|\bigr),
\end{equation}
where the sum runs over the short $U$-orbits; see
\cite[Theorem~11.62]{HKT}.  We always use $0\le\gamma(X)\le g(X)$.

For the wild point used throughout the paper, put
\[
 B=G_P,
 \qquad S=G_P^{(1)},
 \qquad |S|=s.
\]
Then
\[
 B=S\rtimes C_h,
 \qquad |C_h|=h,
 \qquad (h,p)=1.
\]
We use the convention
\[
 d_P(G)=sh+v.
\]
Since the higher ramification groups of $B$ and $S$ agree from level one on,
\[
 d_P(S)=s+v.
\]
In particular
\begin{equation}\label{eq:v-lower}
 v\ge s-2,
\end{equation}
because $d_P(S)\ge2(s-1)$.

We shall use repeatedly the following local divisibility fact.

\begin{lemma}[Local divisibility]\label{lem:local-div}
With the notation above,
\[
 h\mid v+1.
\]
More generally, suppose that a $p$-subgroup $W$ of a point stabilizer is
constant from ramification level $1$ through level $r-1$ and drops at level
$r\ge2$.  If
\[
 q_1=[W:W^{(r)}],
\]
then the tame complement $C$ normalizing $W$ satisfies
\[
 |C|\mid (r-1)(q_1-1).
\]
Consequently
\[
 |C|\le (r-1)(|W|-1).
\]
\end{lemma}

\begin{proof}
Let $L=\kk((z))$ be the local field at the fixed point and
$F=L^{S\rtimes C_h}$.  In the tower $L/L^{C_h}/F$, the first extension is
tame of degree $h$, with different exponent $h-1$.  Transitivity of the
different gives
$d(L/F)\equiv h-1\pmod h$.  Since $d(L/F)=sh+v$, this proves
$h\mid v+1$; normality of $C_h$ is not required in this tower.

For the second assertion put $j=r-1$.  On the first nonzero ramification
quotient the conjugation character is the $j$th power of the faithful tangent
character.  Its image has order $|C|/\gcd(|C|,j)$ and acts freely on the
$q_1-1$ nonzero elements.  Thus
$|C|/\gcd(|C|,j)\mid q_1-1$, and hence
$|C|\mid j(q_1-1)$.  This is also the local relation in
\cite[Lemma~11.82]{HKT}.
\end{proof}

\subsection{Standard bounds}

The following estimates will be used without further comment.

\begin{proposition}[Standard automorphism bounds]\label{prelim:prop:bounds}
Let $X$ have genus $g\ge2$ over an algebraically closed field of odd
characteristic $p$.
\begin{enumerate}[label=\textup{(\roman*)}]
\item If $\gamma(X)>0$ and $U\le\Aut(X)$ is a $p$-group, then
\[
 |U|\le \frac{p}{p-2}(g-1)\le3(g-1).
\]
\item If $A\le\Aut(X)$ fixes a point and $p\nmid|A|$, then
\[
 |A|\le4g+2.
\]
\item If the action of $A\le\Aut(X)$ is tame, then
\[
 |A|\le84(g-1).
\]
In particular this holds for every $p'$-subgroup of $\Aut(X)$.
\item If $\gamma(X)>0$, then every $A\le\Aut(X)$ satisfies
\[
 |A|\le8g^3.
\]
\item If $E$ is an elliptic curve and $A\le\Aut(E)$ fixes a point, then
$|A|\le12$ in characteristic $3$ and $|A|\le6$ in characteristic $p>3$.
\item If $A\le\Aut(\PP^1)$ is a finite $p'$-group fixing a point, then
$A$ is cyclic.  Moreover every nonidentity element of such a group fixes
exactly two points of $\PP^1$.
\item If $\gamma(X)=0$ and $1\ne U\le\Aut(X)$ is a $p$-group, then $U$
fixes a unique point of $X$ and acts semiregularly on its complement.
\end{enumerate}
\end{proposition}

\begin{proof}
For (i), Nakajima's theorem gives $|U|\mid(g-1)$ when
$\gamma(X)=1$, and
$|U|\le p(\gamma(X)-1)/(p-2)$ when $\gamma(X)>1$; see
\cite{Nakajima} and \cite[Theorem~11.84]{HKT}.  Since
$\gamma(X)\le g$, the displayed uniform bound follows.  Assertion (ii) is
the tame fixed-point bound \cite[Theorem~11.60]{HKT}, and (iii) is the tame
Hurwitz bound \cite[Theorem~11.56]{HKT}.  Assertion (iv) is the consequence of
Henn's bound recorded in \cite[Theorem~11.127]{HKT}: an automorphism group of
order greater than $8g^3$ forces zero $p$-rank.  Assertion (v) follows from
the structure of point stabilizers on elliptic curves
\cite[Theorem~11.94]{HKT}.  Assertion (vi) follows from the standard
classification of finite subgroups of $\PGL_2(\kk)$; see
\cite[Theorem~11.91]{HKT}.  Finally, (vii) follows from
\cite[Lemma~11.129]{HKT}: an element of order $p$ on a zero-$p$-rank curve
has a unique fixed point.  Applying this to a central element of order $p$
in $U$ shows that $U$ fixes that point, and applying it to the order-$p$
subgroup of any nonidentity cyclic subgroup of $U$ gives semiregularity away
from it.
\end{proof}

\subsection{Finite-group notation and elementary facts}

For a finite group $H$ and a prime $r$, we denote by $O_r(H)$ and
$O_{r'}(H)$ the largest normal $r$-subgroup and the largest normal
$r'$-subgroup, respectively.  We write $\Rad(H)$ for the soluble radical,
$F(H)$ for the Fitting subgroup, $E(H)$ for the layer (the central product of
the components), and
\[
 F^*(H)=F(H)E(H)
\]
for the generalized Fitting subgroup.  The socle $\soc(H)$ is the product of
the minimal normal subgroups of $H$.  For a $p$-group $Q$, $\Phi(Q)$ denotes
its Frattini subgroup and $Z(Q)$ its center.  We also write
$\core_H(M)=\bigcap_{x\in H}M^x$ and use $N_H(A)$ and $C_H(A)$ for
normalizers and centralizers.

A finite group $H$ is \emph{almost simple} if
\[
 T\le H\le\Aut(T)
\]
for some nonabelian finite simple group $T$; then $T=\soc(H)$.  A component
is a subnormal quasisimple subgroup, where quasisimple means perfect with
simple quotient modulo its center.  A subgroup $Q\le H$ is a \emph{TI
subgroup} if $Q\cap Q^x=1$ for every $x\notin N_H(Q)$.  A $p$-subgroup
$R\le H$ is \emph{$p$-radical} if
\[
 R=O_p(N_H(R)).
\]
If $A\le P\in\Syl_p(H)$, then $A$ is \emph{strongly closed in $P$ with
respect to $H$} if $a^x\in P$ with $a\in A$ and $x\in H$ always implies
$a^x\in A$.  A proper subgroup $M<H$ is \emph{strongly $p$-embedded} if
$p\mid|M|$ and
\[
 p\nmid |M\cap M^x|
 \qquad(x\notin M).
\]
We use the term $K$-group in the standard finite-group-theoretic sense; by
CFSG every finite group occurring here is a $K$-group.  Standard tools used
only in isolated later arguments---such as Maschke's theorem,
Schur--Zassenhaus, Hasse--Arf, Herbrand's theorem, Lang--Steinberg, and finite
Borel--Tits---are recalled at the point where they are needed rather than
front-loaded here.

We shall also use the following elementary facts without repeated reference.

\begin{lemma}[Elementary finite-group facts]\label{prelim:lem:group-facts}
Let $H$ be finite.
\begin{enumerate}[label=\textup{(\roman*)}]
\item Every minimal normal subgroup of $H$ is characteristically simple;
thus it is either elementary abelian or a direct product $T^k$ of isomorphic
nonabelian simple groups.
\item If $N\triangleleft H$ and $P\in\Syl_p(N)$, then the Frattini argument
gives
\[
 H=N N_H(P).
\]
\item For a finite $p$-group $P$,
\[
 \Phi(P)=P^p[P,P],
\]
and a subgroup $L\le P$ equals $P$ as soon as its image in
$P/\Phi(P)$ is all of $P/\Phi(P)$.
\item If $T\triangleleft H$ is nonabelian simple and $C_H(T)=1$, then
conjugation embeds $H$ in $\Aut(T)$; hence $T\le H\le\Aut(T)$ and $H$ is
almost simple.
\item If $P\in\Syl_p(H)$ and $N_H(P)=C_H(P)$, then Burnside's normal
$p$-complement theorem gives a normal $p'$-subgroup $K\triangleleft H$ with
$H=PK$.
\end{enumerate}
\end{lemma}

\begin{proof}
These are standard.  For (i), a characteristic subgroup of a minimal normal
subgroup is normal in $H$, so minimal normal subgroups are characteristically
simple.  Assertion (ii) is the Frattini argument, (iii) is the standard
Frattini property of finite $p$-groups, (iv) follows from the conjugation
homomorphism $H\to\Aut(T)$, whose kernel is $C_H(T)$, and (v) is Burnside's
normal $p$-complement theorem.
\end{proof}

\subsection{A numerical lemma}

\begin{lemma}[The integer reduction]\label{reduction:lem:ell}
Let $s,v,\ell$ be integers with $s\ge3$ odd, $v\ge s-2$, and $\ell\ge2$.
If $u=\ell v-s>0$, then $(s+\ell)/u\le7$.
\end{lemma}
\begin{proof}
For $s\ge5$, replace $v$ by $s-2$.  The resulting quotient
$(s+\ell)/(\ell(s-2)-s)$ decreases with $\ell$, so it is at most
$(s+2)/(s-4)\le7$.  For $s=3$, if $\ell=2,3$, positivity forces $v\ge2$
and gives bounds $5,2$, respectively.  For $\ell\ge4$, use $v\ge1$ and
$(\ell+3)/(\ell-3)\le7$.
\end{proof}

\subsection{The large-group reduction}

We isolate the short-orbit reduction that will be used throughout.

\begin{proposition}[Two short orbits]\label{prop:two-orbits}
Assume
\[
 \gamma(X)>0,
 \qquad
 |G|\ge900g^2.
\]
Then $X/G$ is rational and $G$ has exactly two short orbits, one wild and one
tame.  If $\Delta=G(P)$ is the wild short orbit, then
\[
 |\Delta|>2.
\]
\end{proposition}

\begin{proof}
By \cite[Theorem~1.3(4)]{Mon22}, under the present odd-characteristic and
large-group hypotheses $G$ has exactly two short orbits, one tame and one
non-tame.  If $g(X/G)\ge1$, the tame orbit contributes at least $|G|/2$ to
Hurwitz, giving $|G|\le4(g-1)$, a contradiction.  Thus $X/G$ is rational.

Write $B=G_P=S\rtimes C_h$, with $|S|=s$.  If
$n:=|\Delta|\le2$, Proposition~\ref{prelim:prop:bounds}(i)--(ii) gives
\[
 s\le3(g-1),
 \qquad h\le4g+2.
\]
Therefore
\[
 |G|=nsh\le6(g-1)(4g+2)<24g^2,
\]
contrary to $|G|\ge900g^2$.
\end{proof}

Under the hypothesis $|G|\ge900g^2$ and $\gamma(X)>0$, we henceforth write
\[
 \Delta=G(P)
\]
for the wild short orbit and $\Omega$ for the tame short orbit.

The principal notation used throughout the proof is collected here; symbols
introduced only inside a single proof remain local to that proof.
\begin{center}
\small
\renewcommand{\arraystretch}{1.08}
\begin{tabular}{p{0.18\textwidth}p{0.72\textwidth}}
\toprule
Symbol & Meaning\\
\midrule
$\kk,X,g,\gamma$ & algebraically closed ground field, curve, genus, and $p$-rank\\
$G,P,\Delta,\Omega$ & automorphism group, a wild point, wild short orbit, tame short orbit\\
$B,S,s,C_h,h$ & $G_P=S\rtimes C_h$, with $S=G_P^{(1)}$ and $s=|S|$\\
$v$ & local different excess defined by $d_P(G)=sh+v$\\
$t$ & stabilizer order on the tame short $G$-orbit\\
$a,\ell,u,\nu$ & numerical parameters in package $(\Sigma)$, with $v+1=ah$, $\ell=at-s$, $u=a\nu$\\
$T,d$ & simple socle in the final reduction and $d=|G:T|$\\
$b,c$ & tame inertia orders for the two branch points of $X\to X/T$\\
\bottomrule
\end{tabular}
\end{center}

\section{Rational quotient: Case A}\label{sec:caseA}
Assume throughout this section that
\[
  X/S\cong\PP^1.
\]
Case~A means that every point of $\Delta$ is fixed by a nontrivial element of
$S$.

\begin{proposition}[Transitive Case A]\label{prop:caseA-trans}
Assume $|\Delta|>2$ and that $S$ is transitive on
$\Delta\setminus\{P\}$.  Then
\[
  |G|\le \frac{p+1}{p-1}g^2\le2g^2.
\]
\end{proposition}

\begin{proof}
Put
\[
  m=|\Delta|-1.
\]
Choose $Q\in\Delta\setminus\{P\}$ and set
\[
  W=S\cap G_Q,
  \qquad |W|=w.
\]
By transitivity,
\[
  s=mw.
\]
Both $m$ and $w$ are nontrivial powers of $p$, hence
\begin{equation}\label{eq:m-w-lower}
  m\ge p,
  \qquad
  w\ge p.
\end{equation}
The two-point stabilizer has the form
\[
  G_{P,Q}=W\rtimes C_h,
\]
because
\[
  |G_{P,Q}|=\frac{|G_P|}{m}=wh.
\]

Let $r\ge2$ be the first index for which the ramification filtration of $W$ at
$Q$ drops.  Lemma~\ref{lem:local-div} gives
\begin{equation}\label{eq:h-local-caseA}
  h\le(r-1)(w-1).
\end{equation}
At $P$ the contribution of $S$ to the different is at least $2(s-1)$, and at
each of the $m$ points of $\Delta\setminus\{P\}$ the contribution is at least
$r(w-1)$.  Hurwitz for $X\to X/S\cong\PP^1$ therefore gives
\[
  2g-2\ge -2s+2(s-1)+mr(w-1),
\]
so that
\begin{equation}\label{eq:Ar}
  mr(w-1)\le2g.
\end{equation}
Set $A=m(w-1)$.  From \eqref{eq:h-local-caseA},
\[
  mh\le A(r-1).
\]
Moreover, by \eqref{eq:m-w-lower},
\[
  \frac{w(m+1)}{A}
  =\frac{w}{w-1}\frac{m+1}{m}
  \le\frac{p+1}{p-1}.
\]
Thus
\[
  |G|
  =(m+1)mwh
  \le \frac{p+1}{p-1}A^2(r-1).
\]
By \eqref{eq:Ar}, $A\le2g/r$, whence
\[
  |G|
  \le \frac{p+1}{p-1}\frac{4(r-1)}{r^2}g^2
  \le \frac{p+1}{p-1}g^2.
\]
\end{proof}

\begin{lemma}[Nontransitive Case A]\label{lem:caseA-nontransitive}
Assume $X/S\cong\PP^1$ and Case~A. If $S$ is not transitive on
$\Delta\setminus\{P\}$, then
\[
 |\Delta|-1\le \frac{\gamma(X)}{p-1}
 \qquad\text{and}\qquad
 |G_P|\le2\gamma(X).
\]
Consequently
\[
 |G|\le2\gamma(X)\left(\frac{\gamma(X)}{p-1}+1\right)
 \le g^2+2g\le2g^2.
\]
\end{lemma}
\begin{proof}
Let $\mathcal B$ be the set of short $S$-orbits distinct from $\{P\}$.
Every such orbit is contained in $\Delta$, because the other short $G$-orbit
is tame. Conversely Case~A implies that $\Delta\setminus\{P\}$ is the union
of the orbits in $\mathcal B$. Since $X/S$ is rational, the
Deuring--Shafarevich formula gives
\[
 \gamma(X)=\sum_{A\in\mathcal B}(s-|A|).
\]
For $A\in\mathcal B$ its stabilizer is a nontrivial $p$-group, so
$|A|\le s/p$ and therefore
\[
 s-|A|\ge(p-1)|A|.
\]
Summing proves $|\Delta|-1\le\gamma(X)/(p-1)$.

Write $G_P=S\rtimes C_h$. The induced cyclic tame group $C_h$ acts on
$X/S\cong\PP^1$, fixes the image of $P$, and preserves the branch set
corresponding to $\mathcal B$. Because $S$ is not transitive on
$\Delta\setminus\{P\}$, this branch set has at least two points. If $h>1$,
a nontrivial tame cyclic subgroup of $\PGL_2(\kk)$ has exactly two fixed
points; away from them its orbits have length $h$. Hence the branch set has
at least $h$ points. If $h=1$, it has at least two points. Each point
contributes at least $s(1-1/p)\ge s/2$ to the displayed
Deuring--Shafarevich sum. Thus in all cases
\[
 \gamma(X)\ge \frac{sh}{2}=\frac{|G_P|}{2}.
\]
The final inequalities follow from orbit--stabilizer, $p\ge3$,
$\gamma(X)\le g$, and $g\ge2$.
\end{proof}

\begin{corollary}[Exclusion of Case A]\label{cor:caseA}
Assume $p>2$, $\gamma(X)>0$, $X/S\cong\PP^1$, and
$|G|\ge900g^2$. Then Case~A cannot occur.
\end{corollary}

\begin{proof}
If $S$ is transitive on $\Delta\setminus\{P\}$,
Proposition~\ref{prop:caseA-trans} gives $|G|\le2g^2$.
Otherwise Lemma~\ref{lem:caseA-nontransitive} gives the same contradiction.
\end{proof}

\begin{remark}
Hence, when $X/S$ is rational, it remains only to consider Case~B.
\end{remark}

\section{Rational quotient: structural reductions in Case B}\label{sec:caseB}
Continue to assume $X/S\cong\PP^1$, but now suppose Case~B holds.  Thus there
exists a point $Q\in\Delta\setminus\{P\}$ such that
\[
  S\cap G_Q=1.
\]
\subsection{A singleton short \texorpdfstring{$S$}{S}-orbit}
Assume that $S$ has exactly two short orbits,
\[
  \{P\},\qquad\{R\},
\]
with $R\in\Delta$.

\begin{proposition}[Two singleton short orbits]\label{prop:singleton}
Suppose the only short $S$-orbits are $\{P\}$ and $\{R\}$,
where $R\in\Delta$. Then $\gamma=s-1$, $S\in\Syl_p(G)$,
the genus $g$ is even, and $h\mid g$.
\end{proposition}
\begin{proof}
Deuring--Shafarevich gives $\gamma-1=-s+2(s-1)=s-2$.
If $S<U$ for a $p$-group $U$, then $N_U(S)>S$.
This normalizer preserves $\{P,R\}$ and, as $p$ is odd,
fixes both points. This contradicts $|G_P|_p=s$.

Choose $x\in G$ with $xP=R$. Both $S$ and $xSx^{-1}$ are contained in
$G_R^{(1)}$, the unique Sylow $p$-subgroup of $G_R$, and both have order
$s$. Hence $S=xSx^{-1}$. Thus $x$ normalizes $S$, and the local different exponents
of the $S$-cover at $P,R$ are equal, say $d_0$.
Hurwitz gives $g=d_0-s+1$. Each summand of $d_0$ is even,
so $g$ is even. Also $d_P(S)=s+v=d_0$ gives $v=g-1$.
Lemma~\ref{lem:local-div} yields $h\mid g$.
\end{proof}

\begin{theorem}[The two-singleton case]
\label{thm:no-singleton}
Assume $\gamma(X)>0$ and $|G|\ge900g^2$. Then $S$ cannot have exactly
two short orbits, both singletons as in Proposition~\ref{prop:singleton}.
\end{theorem}

\begin{proof}
Let $d$ be the common different exponent of the $S$-cover at its two
ramified points. Every ramification group has odd order, so $d$ is even.
The equality $g=d-s+1$ in the proof of
Proposition~\ref{prop:singleton} shows that $g$ is even. If $G$ is non-soluble and $|G|>900g^2$, the case is excluded by
\cite[Theorem~1.2]{GK19}. If $G$ is soluble, \cite[Theorem~1.1]{GK19}
already gives the stronger bound
\[
 |G|\le \frac{84p}{p-2}g^2\le252g^2.
\]
Thus only the equality $|G|=900g^2$ remains.

Suppose $|G|=900g^2$. Proposition~\ref{prelim:prop:bounds}(iv) gives
$|G|\le8g^3$, hence $g\ge114$,
since $g$ is even. Write $g=ah$, and let $t$ be the tame inertia order.
Here $v=g-1$, so the global Hurwitz formula is
\[
 900g^2=\frac{2(g-1)sht}{(g-1)t-sh}.
\]
Rearrangement, using $\gcd(h,g-1)=1$, gives
\[
 g-1\mid450s.
\]
If $p\mid g$, then $\gcd(s,g-1)=1$, and $g-1$ is an odd divisor
of $450$ greater than $112$. Thus $g=226$, $p=113$, and the Sylow
condition gives $s=113^2>g+1$, contrary to $s-1=\gamma\le g$.

If $p\nmid g$, the $p$-part of $900g^2$ gives either $(p,s)=(3,9)$
or $(5,25)$. Put $\ell=at-s$. Positivity of the Hurwitz denominator
implies $\ell\ge1$. Since $g-1>s$, the function
\[
 x\longmapsto\frac{s+x}{x(g-1)-s}
\]
is positive and decreasing for $x\ge1$. Therefore
\[
 \frac{|G|}{g^2}
 =\frac{2(g-1)s(s+\ell)}{ag\bigl(\ell(g-1)-s\bigr)}
 \le\frac{2(g-1)s(s+1)}{ag(g-1-s)}
 <\frac{1300}{88}<15,
\]
again a contradiction.
\end{proof}

\subsection{A non-singleton short orbit and an intermediate rational quotient}
Suppose that $S$ has exactly one short orbit $A\ne\{P\}$, of length
\[
  m=|A|>1.
\]
Fix $R\in A$ and put
\[
  W=S_R,
  \qquad |W|=w,
\]
so $s=mw$.

\begin{lemma}[The intermediate $p$-rank]\label{lem:intermediate-prank}
Let $r_W$ be the number of $W$-orbits on $A$. Then
\[
  \gamma(X/W)=m-r_W.
\]
In particular the following are equivalent:
\[
 \gamma(X/W)=0,
 \qquad
 W\text{ fixes }A\text{ pointwise},
 \qquad
 W=\operatorname{core}_S(W).
\]
\end{lemma}

\begin{proof}
Deuring--Shafarevich for the $S$-action gives
\[
  \gamma=m(w-1).
\]
For the $W$-action the point $P$ is fixed, and every short $W$-orbit is
contained in $\{P\}\cup A$. If the $W$-orbits on $A$ have lengths
$e_1,\ldots,e_{r_W}$, their total contribution to the
Deuring--Shafarevich formula is
\[
 \sum_{i=1}^{r_W}(w-e_i)=r_Ww-m.
\]
Hence
\[
 \gamma-1=w\bigl(\gamma(X/W)-1\bigr)+(w-1)+(r_Ww-m),
\]
and substitution of $\gamma=m(w-1)$ gives
\[
 \gamma(X/W)=m-r_W.
\]
Thus $\gamma(X/W)=0$ precisely when every $W$-orbit on $A$ is a singleton.
In that case $W$ is the kernel of the transitive $S$-action on $A$, hence
$W=\operatorname{core}_S(W)$. The converse is immediate.
\end{proof}

\subsection{The core of a non-singleton orbit}
Retain the one-orbit hypothesis above, but do not assume that $X/W$ is
rational or has $p$-rank zero. Put
\[
 K=\operatorname{core}_S(W)=\bigcap_{x\in S}W^x.
\]

\begin{proposition}[The nontrivial-core block reduction]
\label{prop:core-block}
Assume $K\ne1$ and put $D=\{P\}\cup A$. Then:
\begin{enumerate}[label=\textup{(\roman*)}]
\item $D=\Fix(K)$;
\item the sets $\Fix(K^x)$, $x\in G$, form a partition of $\Delta$;
\item if $M=N_G(K)$, then $M$ is the setwise stabilizer of $D$ and
\[
 [M:G_P]=m+1;
\]
\item the permutation group $M^D$ induced by $M$ on $D$ is doubly
transitive and has soluble point stabilizer;
\item $M$ is strongly $p$-embedded in $G$, and consequently
$S\in\Syl_p(G)$.
\end{enumerate}
\end{proposition}

\begin{proof}
Since $K$ is contained in every $S$-conjugate of $W$, it fixes every point
of $A$, and of course it fixes $P$. A further fixed point would belong to a
further short $S$-orbit. Hence $\Fix(K)=D$.

Suppose that a conjugate $K^x$ fixes $P$. Then $K^x$ is a $p$-subgroup of
$G_P$, hence $K^x\le S$. Its fixed points are therefore contained in the
union of the short $S$-orbits, namely $D$. Since
$|\Fix(K^x)|=|D|$, one has $\Fix(K^x)=D$. Thus $K^x$ fixes $A$
pointwise, so $K^x\le K$; equality follows from the orders. Conjugating this
argument shows that two sets $\Fix(K^x)$ which meet are equal. Since every
point of $\Delta$ is a $G$-conjugate of $P$, these fixed sets partition
$\Delta$.

The group $G_P$ normalizes $K$, because it normalizes $S$ and preserves the
unique nontrivial short $S$-orbit $A$. Hence $G_P\le M$. If $R\in A$, choose
$x\in G$ with $x(P)=R$. The blocks $xD$ and $D$ meet at $R$, so $xD=D$ and
$x\in M$. Thus $M$ is transitive on $D$ and
$[M:G_P]=|D|=m+1$. The stabilizer $G_P$ is transitive on
$D\setminus\{P\}=A$ because $S$ is transitive on $A$. Hence $M^D$ is
doubly transitive, and its point stabilizer is a quotient of the soluble
group $G_P$.

Case~B supplies a point $Q\in\Delta$ such that $S\cap G_Q=1$. Such a point
is not in $D$, so $M<G$. Since $m+1$ is prime to $p$, $S$ is a Sylow
$p$-subgroup of $M$. Every nonidentity $p$-element of $M$ is conjugate in
$M$ to an element of $S$; its fixed-point set is nonempty and contained in
$D$. If $x\notin M$, then $D$ and $xD$ are disjoint. A nonidentity
$p$-element in $M\cap M^x$ would have a nonempty fixed-point set contained
in both blocks, a contradiction. Hence $M$ is strongly $p$-embedded.
\end{proof}

\subsection{Several short \texorpdfstring{$S$}{S}-orbits}
We next show that the configuration with several further short $S$-orbits
is not a separate numerical problem.

\begin{proposition}[Several short orbits]
\label{prop:several-short}
Assume $X/S\cong\PP^1$ and that, besides $\{P\}$, the group $S$ has at
least two short orbits. Then, with $t$ the stabilizer order on the tame
short $G$-orbit and
\[
 \nu=vt-sh,
\]
one has
\[
 sh\le\frac{p}{p-1}\gamma\le\frac32g,
 \qquad
 |\Delta|\ge600g,
 \qquad
 \frac{t}{\nu}>300.
\]
Writing $v+1=ah$, $\ell=at-s$ and $u=a\nu$, one further has
\[
 \ell=1,
 \qquad
 u<\frac{s+1}{300},
 \qquad
 h\le2u,
 \qquad
 1<|S_P^{(2)}|<s,
\]
and
\[
 s+1>45000,
 \qquad
 s<\frac{g}{200}.
\]
In particular all the inequalities in the numerical package
$(\Sigma)$ below hold, with room to spare.
\end{proposition}

\begin{proof}
Let $B$ be the branch locus of the $S$-cover $X\to X/S$ away from the
image of $P$. If the stabilizer order above $b\in B$ is $w_b$, then
Deuring--Shafarevich gives
\[
 \gamma=\sum_{b\in B}s\left(1-\frac1{w_b}\right).
\]
The cyclic group $C_h=G_P/S$ acts faithfully on $X/S\cong\PP^1$ and
preserves $B$.  If $h>1$, Proposition~\ref{prelim:prop:bounds}(vi) shows
that it fixes precisely the image of $P$ and one further point; every other
orbit has length $h$. Since $|B|\ge2$, $B$ contains a
free $C_h$-orbit. If $h=1$ the same estimate follows from any point of $B$.
As every $w_b$ is divisible by $p$, we obtain
\[
 \gamma\ge sh\left(1-\frac1p\right),
\]
which proves the first bound. Hence
\[
 |\Delta|=\frac{|G|}{sh}\ge600g.
\]
The global Hurwitz relation gives
\[
 \frac{t}{\nu}=\frac{|\Delta|}{2g-2}>300.
\]
By Lemma~\ref{lem:local-div}, write $v+1=ah$. Then
\[
 \nu=h\ell-t,
 \qquad
 u=\ell v-s,
 \qquad
 \frac{t}{\nu}=\frac{s+\ell}{u}.
\]
If $\ell\ge2$, Lemma~\ref{reduction:lem:ell} gives
$(s+\ell)/u\le7$, a contradiction. Thus $\ell=1$ and
\[
 at=s+1,
 \qquad
 u=v-s,
 \qquad
 ah=s+u+1,
 \qquad
 h=t+\nu.
\]
Consequently $u<(s+1)/300$.

Now
\[
 \sum_{i\ge2}(|S_P^{(i)}|-1)=u+2.
\]
The second ramification group is nontrivial and proper; otherwise the last
identity contradicts $u>0$ or $u<(s+1)/300$. Put
$k=|S_P^{(2)}|$. At the first drop, Lemma~\ref{lem:local-div} gives
\[
 h\mid \frac{s}{k}-1.
\]
Since $h=(s+u+1)/a>s/a$ and $h<s/k$, one obtains
$k<a\le u$. Let $j\ge3$ be the first index with
$S_P^{(j)}\subsetneq S_P^{(2)}$. Again by Lemma~\ref{lem:local-div},
\[
 h\le(j-1)(k-1),
\]
whereas the higher-different sum gives
\[
 (j-2)(k-1)\le u+2.
\]
Hence $h\le u+k+1\le2u$.
Thus $s+u+1=ah\le2u^2$, and the bound on $u$ yields $s+1>45000$.
Finally $h=t+\nu>301$ while $sh\le3g/2$, so $s<g/200$.
\end{proof}

\section{Normalizer actions in the nontrivial-core case}
\label{sec:normalizer-completion}
\subsection{A normalizer action for every core}

Assume now that $X/S\cong\PP^1$ and that the only short $S$-orbits are
\[
 \{P\}\quad\hbox{and}\quad A,\qquad |A|=m>1.
\]
Assume $A\subseteq\Delta=G(P)$. Choose $R\in A$ and put
\[
 W=S\cap G_R,\qquad w=|W|>1,\qquad s=mw.
\]
No condition on $\operatorname{core}_S(W)$ is imposed.

\begin{proposition}\label{normalizer:prop:sylow-rational}
One has
\[
 S\in\Syl_p(G),\qquad N_G(S)=G_P.
\]
Moreover $G$ contains an element interchanging $P$ and $R$, and such
an element normalizes $W$.
\end{proposition}

\begin{proof}
The only point fixed by all of $S$ is $P$. Thus every element of
$N_G(S)$ fixes $P$, and the reverse inclusion follows from
$S\triangleleft G_P$. This proves $N_G(S)=G_P$.
If $S$ were properly contained in a $p$-subgroup $U$, the normalizer
condition would give
\[
 S<N_U(S)\le N_G(S)=G_P,
\]
contrary to $|G_P|_p=s$. Hence $S$ is Sylow in $G$.

The orbit $A$ is preserved by $B=G_P$, since $B$ normalizes $S$ and
$A$ is its unique nonsingleton short orbit. It is therefore a $B$-orbit
of length $m$. Consider the orbital containing $(P,R)$. In a finite
transitive action, the paired suborbit has the same length $m$.
On that paired $B$-orbit, the normal subgroup $S$ is transitive:
the number of $S$-orbits divides both $m$ and $[B:S]=h$, and these
integers are coprime. Since $m<s$, the paired orbit must also be $A$.
The orbital is therefore self-paired, so some $x\in G$ satisfies
$x(P)=R$ and $x(R)=P$.

Let $S_R=G_R^{(1)}$. Since $W=S\cap S_R$ and $x$ interchanges
$S$ and $S_R$, it normalizes $W$.
\end{proof}

\begin{proposition}\label{normalizer:prop:normalizer-rational}
Set
\[
 q=[N_S(W):W],\qquad M=N_G(W),\qquad D_W=\Fix_X(W).
\]
Then $q\ge p$, and
\begin{equation}\label{normalizer:eq:normalizer-data}
 \begin{gathered}
 D_W=\{P\}\cup N_S(W)R,\qquad |D_W|=q+1,\\
 M_P=N_B(W),\qquad |M_P|=qwh,\qquad |M|=(q+1)qwh.
 \end{gathered}
\end{equation}
The induced permutation group $M^{D_W}$ is doubly transitive.
Its point stabilizer has a normal regular $p$-subgroup of order $q$
on the other $q$ points, and the quotient by this $p$-subgroup is cyclic.
These conclusions hold even if $\operatorname{core}_S(W)=1$.
\end{proposition}

\begin{proof}
Every fixed point of $W$ belongs to a short $S$-orbit, so
$D_W\subseteq\{P\}\cup A$. Under the identification $A\cong S/W$,
a point $s_0R$ is fixed by $W$ exactly when
$W\le s_0Ws_0^{-1}$, hence exactly when $s_0\in N_S(W)$.
Thus $\Fix_A(W)=N_S(W)R$, with $q$ points, and
$N_S(W)/W$ acts regularly on that set. Since $W<S$, the normalizer
condition gives $q\ge p$.

By Proposition~\ref{normalizer:prop:sylow-rational}, an element of $M$ sends $P$
to $R$. Together with $N_S(W)$ this makes $M$ transitive on $D_W$.
Its point stabilizer contains $N_S(W)$ and is transitive on
$D_W\setminus\{P\}$, proving double transitivity.

Since $B$ acts transitively on $A$,
\[
 |B_R|=\frac{|B|}{m}=wh.
\]
The subgroup $W=B_R\cap S$ is normal in $B_R$, so
$B_R\le N_B(W)$. Its map onto $B/S$ is surjective.
The kernel of $N_B(W)\to B/S$ is $N_S(W)$; consequently
$|N_B(W)|=qwh$. Orbit--stabilizer gives the formula for $|M|$.

Finally, the image of $N_S(W)$ in $M^{D_W}$ has kernel exactly $W$,
is normal in the point stabilizer, and is regular on the other points.
The quotient of the point stabilizer by this image is a quotient of
$N_B(W)/N_S(W)\cong C_h$, hence is cyclic.
\end{proof}

\begin{remark}\label{normalizer:rem:no-block}
This extends the availability of a doubly transitive normalizer action
to the core-free case. It does \emph{not} make $D_W$ a $G$-block.
Indeed, if $W$ is not normal in $S$, choose $s_0\in S$ with
$W^{s_0}\ne W$. Both $D_W$ and $s_0D_W$ contain $P$, but they are
different: if they were equal, $W^{s_0}$ would fix $R$, giving
$W^{s_0}\le S_R\cap S=W$, a contradiction. Thus these fixed sets
intersect without being equal. In particular, the earlier strongly
$p$-embedded argument cannot simply be applied to this normalizer.
\end{remark}

\subsection{The nontrivial-core case}
\label{na:sec:hyp}

Let $X$ be a smooth projective curve over an algebraically closed field
of characteristic $p>2$, with genus $g\ge2$ and positive $p$-rank.
We now impose the following rational Case~B configuration:
\[
 B=G_P=S\rtimes C_h,\quad s=|S|,\quad (h,p)=1,
 \quad X/S\cong\PP^1.
\]
The only short $S$-orbits on $X$ are $\{P\}$ and an orbit $A$ of
length $m>1$. Choose $R\in A$, and put
\[
 W=S\cap G_R,\quad w=|W|>1,\quad s=mw,
 \quad K=\core_S(W).
\]
We assume that $A\subseteq\Delta=G(P)$ and that some $Q\in\Delta$
has $S\cap G_Q=1$. In particular $D=\{P\}\cup A$ is a proper
subset of $\Delta$.

When $K\ne1$, the nontrivial-core block proposition above gives
\[
 M=N_G(K)<G,\quad B\le M,\quad [M:B]=m+1,
\]
with $M$ strongly $p$-embedded in $G$, $S\in\Syl_p(G)$,
and $L=M^D$ doubly transitive with soluble point stabilizer.
Here $M^D$ means the \emph{faithful induced} permutation group,
not necessarily $M/K$. The kernel $E$ of $M\to L$ satisfies
\begin{equation}\label{na:eq:kernel}
 E\le B,\qquad E\cap S=K.
\end{equation}
Thus $K$ is a Sylow $p$-subgroup of $E$.

The next step uses the classification of strongly $p$-embedded groups,
together with the action kernel in
\eqref{na:eq:kernel}.

\begin{theorem}[The nontrivial-core case]\label{na:thm:main}
Assume $K\ne1$. Then the group $L$ is affine. Moreover
\[
 m=p=2^d-1,\qquad W=K\triangleleft S,
 \qquad \gamma(X/W)=0.
\]
Thus the normalizer action in the nontrivial-core case is affine.
In particular,
\[
 W\not\triangleleft S\quad\Longrightarrow\quad\core_S(W)=1.
\]
\end{theorem}

\subsection{The group-theoretic reduction}

\begin{lemma}\label{na:lem:cyclic}
Let $H$ be a finite group with a cyclic Sylow $p$-subgroup $U$, and
let $1\ne V\triangleleft H$ be a $p$-subgroup. Then there is a
normal $p'$-subgroup $J\triangleleft H$ such that $H/J$ is soluble.
In particular every quotient $L_0$ of $H$ with $O_{p'}(L_0)=1$
is soluble.
\end{lemma}
\begin{proof}
We have $V\le U$. Set $C=C_H(V)$; then $C\triangleleft H$,
$U\le C$, and $H/C$ embeds in the abelian group $\Aut(V)$.
The group $N_C(U)/C_C(U)$ has order prime to $p$, since $U$ is
abelian and Sylow. Its conjugation action on $U$ fixes $V$ pointwise.
The subgroup of $\Aut(U)$ fixing a nontrivial subgroup of a cyclic
$p$-group pointwise is a $p$-group. Therefore $N_C(U)=C_C(U)$.
Burnside's normal $p$-complement theorem gives a normal $p$-complement
$J$ in $C$. This complement is characteristic in $C$, so $J\triangleleft H$.
Both $C/J$ and $H/C$ are soluble, hence so is $H/J$.
In a quotient with trivial $p'$-core the image of $J$ is trivial.
\end{proof}

For the next step suppose that $L$ in Subsection~\ref{na:sec:hyp}
is nonsoluble. We first prove, without a classification of doubly
transitive groups, that
\begin{equation}\label{na:eq:pcoreL}
 O_{p'}(L)=1.
\end{equation}
Let $V\triangleleft L$ be a nontrivial $p'$-subgroup, and let
$\alpha$ be the image of $P$ in $D$. Since the action is faithful
and primitive, $V$ is transitive. The image $P_0$ of $S$ is
normal in $L_\alpha$ and transitive on $D\setminus\{\alpha\}$,
which has $m$ points. Also $V_\alpha\triangleleft L_\alpha$.
Hence $[V_\alpha,P_0]=1$. The centralizer of a transitive
permutation group is semiregular: an element of the centralizer
fixing one point fixes all points. It follows that $|V_\alpha|$
divides the $p$-power $m$. Since $V_\alpha$ is a $p'$-group,
it is trivial, and $V$ is regular on $D$.

Under the identification $D=V$, the group $L_\alpha$ is
transitive by conjugation on $V\setminus\{1\}$. Thus all
nonidentity elements of $V$ have the same prime order, say
$r$, so $V$ is an $r$-group. Its nontrivial characteristic
center is $L_\alpha$-invariant, hence equals $V$.
Therefore $V$ is elementary abelian. Since
$L/V\cong L_\alpha$ is soluble, $L$ is soluble, a contradiction.
This proves \eqref{na:eq:pcoreL}.

\begin{lemma}\label{na:lem:reduction}
Suppose $L$ is nonsoluble. Then $S$ is noncyclic. Put
$N=O_{p'}(G)$ and use bars for passage to $G/N$. Then $N\le E$,
$\overline M$ is strongly $p$-embedded in $\overline G$,
$1\ne\overline K\triangleleft\overline M$, and
$T=F^*(\overline G)$ is nonabelian simple. Also $L$ is a quotient
of $\overline M$ and
\begin{equation}\label{na:eq:lp}
 |L|_p=\frac{|S|}{|K|}.
\end{equation}
\end{lemma}
\begin{proof}
If $S$ were cyclic, apply Lemma~\ref{na:lem:cyclic} to $M$ and $K$,
then use \eqref{na:eq:pcoreL}. This would make $L$ soluble, a contradiction.
For odd $p$, a noncyclic finite $p$-group has $p$-rank at least two.
Parker--Stroth's Lemma~3.2(v) now gives
\[
 N=O_{p'}(G)=\core_G(M)\le M.
\]
Its image in $L$ is a normal $p'$-subgroup, so $N\le E$ by
\eqref{na:eq:pcoreL}. Their Lemma~3.3 applies: since $N\le M<G$,
$\overline M<\overline G$ remains strongly $p$-embedded and
$F^*(\overline G)$ is nonabelian simple.
The map $S\to\overline S$ is an isomorphism, so
$\overline K\ne1$. Finally $N\le E$ makes $L$ a quotient of
$\overline M$, and \eqref{na:eq:kernel} gives \eqref{na:eq:lp}.
\end{proof}

\subsection{Solubility of the faithful action}

\begin{proof}[Proof of Theorem~\ref{na:thm:main}]
Suppose that $L$ is nonsoluble, and apply
Lemma~\ref{na:lem:reduction}. Put
\[
 P_0=\overline S,\qquad Q=P_0\cap T,\qquad
 K_0=\overline K\cap T.
\]
We claim that $K_0\ne1$. If $p\nmid|\overline G:T|$, then
$\overline K\le P_0\le T$, and this is immediate.
Since $\overline G$ is an almost simple group with socle $T$, it is a
$K$-group by the classification of finite simple groups, so the hypotheses
of Parker--Stroth's Proposition~2.5 apply. If
$p\mid|\overline G:T|$, that proposition leaves only
\[
 (\overline G,p)=(\PSL_2(8):3,3)
 \quad\text{or}\quad({}^2B_2(32):5,5).
\]
In these two cases Proposition~2.7(iii),(vi) of \cite{PS11}
gives an extraspecial group $P_0$ of order $p^3$ and a cyclic
subgroup $Q$ of order $p^2$. Every subgroup of index $p$ in
$P_0$ contains $\Phi(P_0)=Z(P_0)$, so $Z(P_0)\le Q$.
Since $1\ne\overline K\triangleleft P_0$, it meets $Z(P_0)$
nontrivially. As $|Z(P_0)|=p$, we obtain
$Z(P_0)\le\overline K\cap T=K_0$, proving the claim.

Set $J=\overline M\cap T$. The subgroup $K_0$ is normal in
$\overline M$, and strong $p$-embedding gives
$N_{\overline G}(K_0)\le\overline M$. Thus $J=N_T(K_0)$.
Moreover every normalizer $N_T(V)$ with $1\ne V\le Q$
lies in $J$, whereas $N_T(K_0)$ is one such normalizer.
Consequently
\[
 J=\Gamma_{Q,1}(T),\qquad 1\ne K_0\le O_p(J).
\]
In Proposition~2.7 of \cite{PS11}, the only nonsoluble
possibilities for $\Gamma_{Q,1}(T)$ occur when
$T\cong\operatorname{Alt}(2p)$ or $T\cong Fi_{22}$; in
both cases $O_p(\Gamma_{Q,1}(T))=1$.
All the other entries are soluble, including the
groups of orders $18$ and $100$ in the two outer-$p$ cases.
Therefore $J$ is soluble. The quotient $\overline M/J$
embeds in the soluble group $\operatorname{Out}(T)$.
It follows that $\overline M$, and hence its quotient $L$,
is soluble, a contradiction.

We have proved that $L$ is soluble. A minimal normal subgroup
$V$ of this faithful primitive group is elementary abelian
and transitive. An abelian transitive subgroup is regular:
an element fixing one point fixes all points. Thus $L$ is
affine. Its regular elementary abelian subgroup has even
order $m+1$, whence
\[
 m=p^e,\qquad p^e+1=2^d.
\]
If $e$ is even, then $p^e+1\equiv2\pmod4$, which is impossible
since $m+1\ge4$. If $e>1$ is odd, the factorization
\[
 p^e+1=(p+1)(p^{e-1}-p^{e-2}+\cdots-p+1)
\]
has an odd factor greater than one. Hence $e=1$ and
$m=p=2^d-1$. Since $[S:W]=p$, we have $W\triangleleft S$,
so $W=K$. Lemma~\ref{lem:intermediate-prank} gives
$\gamma(X/W)=0$.
\end{proof}

\begin{remark}
In general $M^D$ need not be $M/K$. The distinction is relevant when
$K=1$, because then $D$ is no longer obtained from the preceding block
argument.
\end{remark}

\section{The core-free rational reduction}
\label{sec:corefree-reduction}
\subsection{Setup}

Let $X$ be a smooth projective curve over an algebraically closed field
of characteristic $p>2$, of genus $g\ge2$. The assertion to be proved is
\begin{equation}\label{corefree:eq:target}
 \gamma(X)>0\quad\Longrightarrow\quad |G|<900g^2
 \qquad(G\le\Aut(X)).
\end{equation}
Whenever minimality is used below, $(X,G)$ is chosen with least genus
among \emph{all} counterexamples to \eqref{corefree:eq:target}, without a
restriction to any particular branch of the reduction.

We work under the following rational Case~B hypotheses:
\begin{gather}\label{corefree:eq:config}
 B=G_P=S\rtimes C_h,\qquad S\in\Syl_p(G),\qquad s=|S|,
 \qquad X/S\cong\PP^1,\\
 \text{the short $S$-orbits on $X$ are $\{P\}$ and $\mathcal A$},
 \qquad |\mathcal A|=m>1.\notag
\end{gather}
Here $\mathcal A\subseteq\Delta=G(P)$ and, for $Q\in\mathcal A$,
\[
 W=S\cap G_Q,\quad w=|W|>1,\quad s=mw,\quad
 K=\core_S(W)=1.
\]
Case~B also supplies a point of $\Delta$ with trivial $S$-stabilizer.
The Deuring--Shafarevich formula gives
\begin{equation}\label{corefree:eq:prank}
 \gamma(X)=s-m>0.
\end{equation}
We use these hypotheses throughout this section. Write
\[
 Z=Z(S),\qquad R=O_{p'}(G),\qquad r=|R|.
\]
Following Flores--Foote, $O_Z(G)$ denotes the largest normal subgroup
$N$ of $G$ for which $Z\cap N$ is a Sylow $p$-subgroup of $N$.

\begin{theorem}[The core-free reduction]\label{corefree:thm:main}
A minimum-genus counterexample satisfying \eqref{corefree:eq:config} and
$K=1$ has
\[
 O_{p'}(G)=O_Z(G)=\Rad(G)=1.
\]
Moreover $G$ is almost simple: $T\le G\le\Aut(T)$ for a nonabelian
simple group $T$, and $1\ne Z\le T$. The possible socles belong to
the list in Flores--Foote's Theorem~1.2.
\end{theorem}

The argument uses classification-dependent finite-group theorems.
In particular the coprime-automorphism theorem of Wang--Chen and
the theorem of Glauberman--Guralnick--Lynd--Navarro are additional
inputs beyond the two group classifications used previously.

\subsection{The center and tame quotients}

\begin{lemma}\label{corefree:lem:center}
The group $S$ is nonabelian. Every $z\in Z\setminus\{1\}$ has
$\Fix_X(z)=\{P\}$. In particular $Z$ is strongly closed in $S$
with respect to $G$, $N_G(Z)=N_G(S)=B$, and $O_p(G)=1$.
\end{lemma}
\begin{proof}
If $S$ were abelian, $K=W\ne1$. Since $Z\cap W$ is normal in
$S$ and contained in $W$, it is trivial. A central element fixing a
point of $\mathcal A$ would belong to $Z\cap W^a=Z\cap W$ for some
$a\in S$. Outside $\{P\}\cup\mathcal A$, stabilizers in $S$ are
trivial. This proves the assertion about fixed points.

If a conjugate of $z\in Z\setminus\{1\}$ belongs to $S$, it fixes
$P$. Uniqueness of its fixed point forces the conjugating element
to belong to $B$, which normalizes $S$ and $Z$. This proves strong
closure and the assertions about normalizers. If $1\ne U=O_p(G)$,
then $U\le S$ and $U$ fixes every point of $\Delta$ by normality.
But $U\cap Z\ne1$, contradicting $|\Delta|>1$.
\end{proof}

\begin{lemma}\label{corefree:lem:soluble}
For every $z\in Z\setminus\{1\}$,
\begin{equation}\label{corefree:eq:centralizer}
 C_R(z)=R_P.
\end{equation}
The group $R$ is soluble. If $G$ is nonsoluble, then
$\Rad(G)=R$.
\end{lemma}
\begin{proof}
An element centralizing $z$ preserves its unique fixed point. Hence
$C_R(z)\le R_P$. Both $R_P$ and $S$ are normal in $B$, so they
commute, proving the opposite inclusion. The tame point stabilizer
$R_P$ is cyclic. Conjugation by $z$ is a coprime automorphism of $R$
with nilpotent fixed subgroup. Wang--Chen's theorem \cite{WC93}, also recalled in \cite[Theorem~5.1]{RS},
makes $R$ soluble. If this automorphism is
trivial, the conclusion follows directly from $R=R_P$.

Set $Y=X/R$, use bars for passage to $\overline G=G/R$, and
identify $S$ with its image. Each geometric fiber of $X\to Y$ has
cardinality prime to $p$. If $z\in Z\setminus\{1\}$ fixes a point
of $Y$, the $p$-group $\langle z\rangle$ has a fixed point on that
fiber. Consequently
\begin{equation}\label{corefree:eq:fixeddown}
 \Fix_Y(z)=\{\overline P\},\qquad
 \overline G_{\overline P}=BR/R\cong B/R_P.
\end{equation}
The latter group is soluble. Thus if $G$ is nonsoluble, so is
$\overline G$ and $|\overline G(\overline P)|>1$.
Clearly $O_{p'}(\overline G)=1$. A nontrivial $O_p(\overline G)$
would meet $Z(S)$ and fix the whole orbit of $\overline P$, again
contradicting \eqref{corefree:eq:fixeddown}. A nontrivial soluble normal
subgroup of $\overline G$ would contain a minimal normal elementary
abelian subgroup, contrary to these two conclusions. Therefore
$\Rad(\overline G)=1$ and $\Rad(G)=R$.
\end{proof}

\begin{lemma}\label{corefree:lem:quotient}
Let $(X,G)$ be a minimum-genus counterexample in the present branch.
If $R\ne1$, then $y=g(Y)\ge2$ and $\gamma(Y)=0$. Furthermore
$Z(S)$ is cyclic and $[R,Z(S)]\ne1$.
\end{lemma}
\begin{proof}
The soluble-group bound in \cite[Theorem~1.1]{GK19} implies that
$G$ is nonsoluble. Since $S$ acts faithfully on $Y$ and fixes
$\overline P$, $Y$ cannot be rational: a finite $p$-subgroup of
$\operatorname{PGL}_2$ is elementary abelian. Nor can $y=1$,
since in odd characteristic the $p$-part of an elliptic point
stabilizer is cyclic (of order at most~3). Both possibilities
contradict Lemma~\ref{corefree:lem:center}; see also \cite[Results~2.8--2.9]{GK19}.

Hurwitz gives $g\ge r(y-1)+1$. For $r>1$ and $y\ge2$,
\[
 r(y-1)+1-\sqrt r\,y
 \ge(\sqrt r-1)^2>0.
\]
Thus $y<g$ and
\begin{equation}\label{corefree:eq:minimal}
 |\overline G|=|G|/r\ge900g^2/r>900y^2.
\end{equation}
Minimality forces $\gamma(Y)=0$.

We first show that $R=R_P$ is impossible. If it held, normality
would make $R$ fix $\Delta$ pointwise. Fibers over the images of
$\Delta$ would be singletons. Every nontrivial $p$-stabilizer on
$Y$ has a fixed point in its prime-to-$p$ fiber, so it lifts to a
short $S$-orbit on $X$. Therefore the short $S$-orbits on $Y$ would
again have lengths $1,m$, with the same stabilizers. The field
$k(Y/S)=k(X)^{RS}$ is a subfield of the rational field $k(X)^S$;
L\"uroth's theorem makes $Y/S$ rational. Deuring--Shafarevich would
give $\gamma(Y)=s-m>0$, a contradiction.

If $Z$ were noncyclic, choose $E\cong C_p^2$ in $Z$. Coprime
centralizer generation \cite[Lemma~3.15]{GMP12} and
\eqref{corefree:eq:centralizer} would give
\[
 R=\langle C_R(e):1\ne e\in E\rangle=R_P,
\]
which was just excluded. Hence $Z$ is cyclic. Finally, if $Z$
centralized $R$, \eqref{corefree:eq:centralizer} would again yield $R=R_P$.
\end{proof}

\subsection{A quotient of zero \texorpdfstring{$p$}{p}-rank}

\begin{lemma}\label{corefree:lem:quotientlist}
Under the hypotheses of Lemma~\ref{corefree:lem:quotient}, one of the
following holds:
\begin{enumerate}
\item $\PSU_3(p)\le\overline G\le\Aut(\PSU_3(p))$;
      $S$ is extraspecial of order $p^3$ and exponent $p$;
\item $p=3$ and $\overline G\cong\PSL_2(8):3$;
\item $p=5$ and $\overline G\cong{}^2B_2(32):5$.
\end{enumerate}
\end{lemma}
\begin{proof}
On a curve of $p$-rank zero every nonidentity element of a
$p$-subgroup fixing a point has that point as its unique fixed
point. Thus $S$ is TI in $\overline G$, and
$I=\overline G_{\overline P}=N_{\overline G}(S)$ is strongly
$p$-embedded. It has a cyclic $p'$-quotient, with normal Sylow
subgroup $S$. Also $O_{p'}(\overline G)=1$ and $S$ is nonabelian,
so its $p$-rank is at least two. Parker--Stroth's Lemma~3.3 gives
a simple socle $T_0=F^*(\overline G)$.

The group $\overline G$ is a $K$-group by the classification of finite
simple groups. Hence the hypotheses of
\cite[Propositions~2.5--2.7]{PS11} are satisfied. If
$p\mid|\overline G:T_0|$, precisely the two exceptional cases
in the statement remain. Otherwise $S\le T_0$ and
\[
 J=I\cap T_0=N_{T_0}(S)=\Gamma_{S,1}(T_0),
 \qquad O_p(J)=S,\quad J/S\text{ cyclic}.
\]
Here $\Gamma_{S,1}(T_0)$ is generated by the normalizers of the
nontrivial subgroups of $S$. The alternating and $Fi_{22}$ entries
have trivial $p$-core in $J$; the $\PSL_3(4)$, $M_{11}$ and
${}^2F_4(2)'$ entries have abelian Sylow subgroups; and the $McL$
and $J_4$ entries have noncyclic $J/S$. These are excluded.
The linear defining-characteristic entry also has abelian $S$.

Only the unitary and Ree entries in defining characteristic remain.
Their Sylow centers are elementary abelian of order $q$, by
\cite[Proposition~2.7]{FF}. Cyclicity of $Z(S)$ forces $q=p$.
For a simple Ree group $q=3^{2a+1}\ge27$, so this cannot occur.
The remaining unitary group has $q=p$, and its Sylow subgroup
has the asserted structure. These structural facts also follow
from the upper-unitriangular unitary matrices.
\end{proof}

\begin{remark}
The Suzuki outer-automorphism case is included here: the defining
characteristic of the abstract socle need not coincide with the
characteristic of the curve.
\end{remark}

\begin{lemma}\label{corefree:lem:representation}
Let $E$ be extraspecial of order $p^3$ and exponent $p$, with $p$
odd, acting on a soluble $p'$-group $D$. If $Z(E)$ acts
nontrivially on $D$, then $|D|\ge2^p$.
\end{lemma}
\begin{proof}
Take an $E$-invariant chief series inside $D$, whose factors are
elementary abelian of characteristics different from $p$.
By coprime action on an invariant chief series, $Z(E)$ acts nontrivially
on at least one chief factor $V$; otherwise its action is trivial on every
factor and hence on $D$.
Over an algebraic closure of the coefficient field, Maschke's
theorem supplies an irreducible constituent $V_0$ on which $Z(E)$
acts nontrivially.

For completeness, choose an abelian subgroup $H$ of $E$ of order
$p^2$ containing $Z(E)$. On $V_0$, a nonzero $H$-weight has
nontrivial restriction to $Z(E)$. Its orbit under $E/H$ has length
$p$: an element outside $H$ has $[H,x]=Z(E)$ and cannot fix this
weight. The corresponding $p$ distinct weight spaces are linearly
independent. Hence $\dim V\ge\dim V_0\ge p$ and
$|D|\ge|V|\ge2^p$.
\end{proof}

\Needspace{6\baselineskip}
\begin{proposition}\label{corefree:prop:removeR}
In a minimum-genus counterexample in the core-free rational branch,
$R=O_{p'}(G)=1$.
\end{proposition}
\begin{proof}
Suppose $r>1$ and put $y=g(X/R)$. Lemmas~\ref{corefree:lem:quotient}
and~\ref{corefree:lem:quotientlist} apply. In each case, Hurwitz gives
\begin{equation}\label{corefree:eq:ratio}
 \frac{|G|}{g^2}
 <\frac{|\overline G|}{r(y-1)^2}.
\end{equation}

\smallskip\noindent\emph{Unitary case.}
By Lemmas~\ref{corefree:lem:soluble}, \ref{corefree:lem:quotient} and
\ref{corefree:lem:representation}, $r\ge2^p$. Also
\[
 |\overline G|\le2p^3(p^3+1)(p^2-1),\qquad
 y\ge\frac{p(p-1)}2.
\]
For the second bound use Nakajima's zero-$p$-rank estimate
\[
 |S|\le\max\{y,4py^2/(p-1)^2\};
\]
see \cite{Nakajima}. If $y\ge p^3$ the conclusion is
immediate. Otherwise $|S|=p^3>y$ forces the quadratic
alternative, which gives $y\ge p(p-1)/2$.
Thus
\begin{equation}\label{corefree:eq:unitaryratio}
 \frac{|G|}{g^2}
 <F(p):=\frac{2p^3(p^3+1)(p^2-1)}
 {2^p\bigl(p(p-1)/2-1\bigr)^2}<400.
\end{equation}
Here $F(3)=378$. For $p\ge5$, $p(p-1)/2-1\ge p^2/3$ and
$p^3+1\le(126/125)p^3$. For every integer $n\ge5$,
$n^4/2^n\le81/4$: check $n=5,6$, then use
$((n+1)/n)^4/2<1$ for $n\ge6$. Therefore
\[
 F(p)\le18\frac{126}{125}\frac{p^4}{2^p}
 \le\frac{45927}{125}<400.
\]

\smallskip\noindent\emph{The $\PSL_2(8):3$ case.}
Here $|\overline G|=1512$, $r\ge2$ and $y-1\ge1$.
Equation~\eqref{corefree:eq:ratio} gives $|G|/g^2<756$.

\smallskip\noindent\emph{The ${}^2B_2(32):5$ case.}
Here $p=5$ and $|\overline G|=162688000$. Its Suzuki socle
contains a Borel subgroup of order $32^2(32-1)=31744$, prime to
$5$. Proposition~\ref{prelim:prop:bounds}(iii), applied to $Y$, gives
$31744\le84(y-1)$, hence $y-1\ge378$. Consequently
\[
 \frac{|G|}{g^2}
 <\frac{162688000}{2\cdot378^2}<570.
\]
For the Borel structure see \cite[Proposition~2.7(1)]{FF}.
All three cases contradict
$|G|\ge900g^2$. Thus $R=1$.
\end{proof}

\subsection{The subgroup \texorpdfstring{$O_Z(G)$}{OZ(G)}}

The next argument is valid before assuming minimality.

\begin{proposition}\label{corefree:prop:oz}
In the core-free case, if $G$ is nonsoluble,
$O_Z(G)=O_{p'}(G)$ and $G/O_{p'}(G)$ is almost simple.
\end{proposition}
\begin{proof}
Pass to $\overline G=G/R$ acting on $Y=X/R$ as in
Lemma~\ref{corefree:lem:soluble}. Identify $S$ and $Z$ with their images.
We have $O_{p'}(\overline G)=1$, the unique-fixed-point property
\eqref{corefree:eq:fixeddown}, and more than one point in
$\overline\Delta=\overline G(\overline P)$.
The subgroup $Z$ remains strongly closed, either by the quotient
property of strong closure or directly from fixed-point uniqueness.

Set $N=O_Z(\overline G)$ and $U=S\cap N$. By the definition of
$O_Z$, $U\le Z$ is a Sylow $p$-subgroup of $N$. Also
$O_{p'}(N)=1$, since this characteristic subgroup of $N$ is normal
in $\overline G$. If $U=1$ then $N=1$ immediately. Suppose $U\ne1$.
Conjugation by every $a\in S$ induces a $p$-power automorphism
of $N$ centralizing its Sylow subgroup $U$. By
\cite[Theorem~1.1]{GGLN}, it is inner.

Choose $n\in C_N(U)$ inducing this automorphism. The central
Sylow subgroup $U$ of $C_N(U)$ has a normal $p'$-complement,
so $C_N(U)=U\times J$ with $J$ a $p'$-group. Write $n=uj$.
The commuting inner automorphisms induced by $u,j$ have coprime
orders. Since their product has $p$-power order, $j$ induces the
identity. Therefore $a u^{-1}\in C_S(N)$, proving
\begin{equation}\label{corefree:eq:splitS}
 S=U C_S(N).
\end{equation}
If $C_S(N)\ne1$, its normality in $S$ supplies
$1\ne z\in C_S(N)\cap Z$. The group $N$ then preserves the
unique fixed point of $z$, hence fixes $\overline P$.
Normality of $N$ makes it fix all of $\overline\Delta$.
But $1\ne U\le Z$ cannot fix more than $\overline P$.
This is a contradiction. Thus $C_S(N)=1$, and
\eqref{corefree:eq:splitS} makes $S=U$ abelian, again a contradiction.
It follows that $O_Z(\overline G)=1$.
Since $R$ is a $p'$-group, the defining property of $O_Z$ passes
back and forth through this quotient. Hence $O_Z(G)=R$.

Apply \cite[Theorem~1.2]{FF} with ambient group $\overline G$
and strongly closed subgroup $Z$. The prime is odd, $Z\ne1$, and
$O_Z(\overline G)=1$. If $Z$ were Sylow in its normal closure,
that normal closure would satisfy the defining property of $O_Z$,
which is impossible. All hypotheses of the cited theorem hold.
Its conclusion concerns
\[
 H=\langle Z^{\overline G}\rangle
   =(T_1\times\cdots\times T_k)(D A_F),
 \qquad 1\ne Z\cap T_i\quad(1\le i\le k).
\]
Each $T_i$ is nonabelian simple and is normalized by $D A_F$;
thus it is a component of $H$. If $k\ge2$, an element
$1\ne z\in Z\cap T_1$ centralizes $T_2$. Uniqueness of its
fixed point puts $T_2$ in the soluble group
$\overline G_{\overline P}$, a contradiction. Hence $k=1$;
write $T=T_1$.

There is no other component $K$ of $H$: distinct components commute,
so such a $K$ would centralize $1\ne z\in Z\cap T$ and lie in
the same soluble point stabilizer. Thus $T=E(H)$, the layer of
$H$. It is characteristic in $H$, and $H\unlhd\overline G$;
therefore $T\unlhd\overline G$. Furthermore $C_{\overline G}(T)$
centralizes $z$ and is a soluble normal subgroup of $\overline G$.
Lemma~\ref{corefree:lem:soluble} gives $\Rad(\overline G)=1$, so this
centralizer is trivial. Consequently
$T\le\overline G\le\Aut(T)$.
Finally \cite[Corollary~1.2]{GGLN}, applied to $\overline G$
with $O_{p'}(\overline G)=1$, gives $Z\le F^*(\overline G)=T$.
The normal closure, its layer, and the ambient quotient have thus
been kept distinct throughout the classification step.
\end{proof}

\begin{proof}[Proof of Theorem~\ref{corefree:thm:main}]
Combine Proposition~\ref{corefree:prop:removeR},
Lemma~\ref{corefree:lem:soluble}, and Proposition~\ref{corefree:prop:oz}.
\end{proof}

\subsection{Restricting the simple socle}

\begin{proposition}\label{corefree:prop:further}
Let $T$ be the socle in Theorem~\ref{corefree:thm:main}, and put
$U=S\cap T\in\Syl_p(T)$. Then the following families are
excluded:
\begin{enumerate}
\item the unitary and Ree families in defining characteristic;
\item $G_2(q)$ in cross characteristic at $p=3$, and $J_2$ at $p=3$;
\item the cross-characteristic Lie-type family with noncyclic $U$.
\end{enumerate}
In the remaining cross-characteristic Lie-type family, $U$ is
cyclic of order at least $p^2$ and $p\mid|G:T|$.
\end{proposition}
\begin{proof}
Since $1\ne Z\le T$, we have $Z=Z\cap T\le Z(U)$.
In the unitary and Ree entries of Flores--Foote's list this forces
$Z=Z(U)$; the alternative Ree commutator subgroup is strictly
larger than $Z(U)$. Their Proposition~2.7 gives
\[
 T_P=N_T(Z)=N_T(U).
\]
Indeed $N_T(Z)$ fixes $P$, whereas $T_P$ normalizes $S\cap T=U$
and hence its center. This is the natural Borel subgroup, and
the corresponding $T$-orbit has size $|U|+1$, with $U$ transitive
on the points other than $P$.

The set of centers of Sylow $p$-subgroups of $T$ is invariant
under $G\le\Aut(T)$. Assign to each such center its unique fixed
point on $X$. This assignment is injective: the Sylow subgroup
of $T$ attached to each center also fixes its unique fixed point.
Two such Sylow subgroups fixing the same point must coincide,
since the point stabilizer in $T$ has a unique Sylow $p$-subgroup.
It follows that
$\Delta=T(P)$, so $S$ is transitive on $\Delta\setminus\{P\}$.
This is Case~A, contradicting the present Case~B hypotheses.

For $G_2(q)$, $N_T(Z)$ contains the nonsoluble subgroup
$\operatorname{SL}_3^{\pm}(q)$; for $J_2$, its normalizer is
$3\operatorname{PGL}_2(9)$. Both contradict the fact that
$N_T(Z)\le B$ is soluble; see \cite[Proposition~2.7]{FF}.

In the cross-characteristic Lie-type family, $U$ is homocyclic
abelian of exponent at least $p^2$, and $Z$ is its subgroup of
the same rank and smaller exponent. Thus $\Omega_1(U)\le Z$.
Every $1\ne u\in U$ has a nontrivial power in $Z$, and hence
$\Fix_X(u)=\{P\}$. It follows that
$T_P=N_T(U)$ is strongly $p$-embedded in $T$.
If $U$ were noncyclic, apply
\cite[Propositions~2.5--2.7]{PS11} to the simple group $T$.
Every entry there with abelian Sylow $p$-subgroup has elementary
abelian Sylow subgroup. This contradicts the exponent of $U$.
Therefore $U$ is cyclic. Since $S$ is nonabelian, $S\ne U$, so
$p\mid|G:T|$.
\end{proof}

For the next step retain the two-short-$G$-orbit conclusion and
ramification identities established above. If $t$ is the stabilizer
order on the tame orbit and $d_P(S)=s+v$, these are
\begin{equation}\label{corefree:eq:twoorbit}
 v\ge s-2,\qquad h\mid v+1,\qquad
 \epsilon=v-\frac{sh}{t}>0,\qquad
 |G|=\frac{2(g-1)sh}{\epsilon}.
\end{equation}
Here $t$ is a positive integer prime to $p$.

\begin{proposition}\label{corefree:prop:sporadic}
Under the retained two-orbit identities \eqref{corefree:eq:twoorbit},
none of the sporadic socles in the Flores--Foote list can occur.
\end{proposition}
\begin{proof}
The outer automorphism groups of the relevant sporadic groups
have order at most two. Since $p$ is odd, $S\le T$ and
$S\in\Syl_p(T)$. The $J_2$ case was excluded in
Proposition~\ref{corefree:prop:further}.

For $Co_2,Co_3,McL,J_4$ and $HS$, the normalizer data in
\cite[Proposition~2.7(4),(5)]{FF} show that $N_T(S)/S$ is noncyclic.
For $HS$ one may also read the normalizer inside the $U_3(5):2$ maximal
subgroup from the $HS$ and $U_3(5)$ ATLAS entries \cite{Atlas}. Since $N_T(S)/S$ embeds in the
cyclic group $B/S$, all these cases are impossible.

It remains only $T=J_3$. Put $d=|G:T|\in\{1,2\}$. Here
\[
 |T|=50\,232\,960,\qquad p=3,\qquad s=243,
 \qquad h_0=|N_T(S):S|=8.
\]
Since $B=N_G(S)$, the Frattini argument gives $G=TB$, and therefore
\[
 |G|=d|T|,\qquad h=dh_0\le16.
\]

First suppose $t\le h$. The positive number $t\epsilon=vt-sh$ is an
integer, so $\epsilon\ge1/t$. From \eqref{corefree:eq:twoorbit},
\[
 |G|\le2(g-1)sht\le2(g-1)sh^2<2gsh^2.
\]
Squaring and dividing by $|G|g^2$ gives
\[
 \frac{|G|}{g^2}
 <\frac{4s^2h^4}{|G|}
 =\frac{4s^2h_0^4d^3}{|T|}
 \le\frac{32s^2h_0^4}{|T|}
 =\frac{248832}{1615}<155,
\]
a contradiction.

Now suppose $t\ge h+1$. Since $v\ge s-2$,
\[
 \epsilon\ge s-2-\frac{sh}{h+1}
 =\frac{s}{h+1}-2\ge\frac{243}{17}-2>1.
\]
Also $w\ge p$ and \eqref{corefree:eq:prank} give
$g\ge(1-1/p)s=162$. Hence
\[
 |G|<2gsh\le7776g\le48g^2,
\]
again contradicting $|G|\ge900g^2$.
\end{proof}

\begin{corollary}\label{corefree:cor:final}
In the full core-free rational Case~B configuration above,
a minimum-genus counterexample has an almost simple group $G$
whose socle $T$ is of Lie type in characteristic different from
$p$. The group $U=S\cap T$ is cyclic of order at least $p^2$,
$p\mid|G:T|$, and $W\cap T=1$. In particular the nontrivial
stabilizer $W$ injects into the outer automorphism group $G/T$.
\end{corollary}
\begin{proof}
Combine Theorem~\ref{corefree:thm:main} and
Propositions~\ref{corefree:prop:further}--\ref{corefree:prop:sporadic}.
Every nontrivial element of $U$ fixes only $P$, while $W$ fixes
a point of $\mathcal A$. Hence $W\cap U=W\cap T=1$.
\end{proof}

The configuration in Corollary~\ref{corefree:cor:final} is eliminated
in the following section.

\section{The core-free cyclic-socle case}
\label{sec:corefree-elimination}
\subsection{Setup}

Let $X$ be a smooth projective curve in characteristic $p>2$ and let
$G\le\Aut(X)$ be finite. We work with the rational Case~B configuration obtained above:
\begin{gather}\label{ce:eq:config}
 B=G_P=S\rtimes C_h,\qquad S\in\Syl_p(G),\qquad (h,p)=1,\\
 \text{the short $S$-orbits on $X$ are $\{P\}$ and $\mathcal A$},
 \qquad |\mathcal A|=m>1,\qquad \mathcal A\subseteq\Delta=G(P).
 \notag
\end{gather}
Choose $Q\in\mathcal A$ and put
\[
 W=S_Q,\qquad w=|W|>1,\qquad s=|S|=mw,\qquad
 \core_S(W)=1.
\]
There is a point of $\Delta$ with trivial $S$-stabilizer.
By Corollary~\ref{corefree:cor:final}, we are left with
\begin{equation}\label{ce:eq:socle}
 T\le G\le\Aut(T),\qquad U=S\cap T\in\Syl_p(T),\qquad
 U\cong C_{p^a},\quad a\ge2,\quad W\cap T=1,
\end{equation}
where $T$ is simple of Lie type in characteristic $\ell\ne p$.
In addition,
\begin{equation}\label{ce:eq:unique}
 \Fix_X(u)=\{P\}\quad(1\ne u\in U),\qquad
 T_P=N_T(U),\qquad T_P/U\ \text{is cyclic}.
\end{equation}
The last cyclicity also follows directly from $B/S$ cyclic.

\begin{theorem}[Exclusion of the cyclic-socle case]\label{ce:thm:elimination}
The hypotheses \eqref{ce:eq:config}--\eqref{ce:eq:unique} are incompatible.
\end{theorem}

In applying this theorem, triviality of the tame normal core and
the socle description come from Section~\ref{sec:corefree-reduction}.
These hypotheses are not conclusions of the elimination alone.

We use Proposition~\ref{normalizer:prop:normalizer-rational}.
Set
\begin{equation}\label{ce:eq:local}
 M=N_G(W),\qquad D=\Fix_X(W),\qquad q=[N_S(W):W].
\end{equation}
Then $q\ge p$, $|D|=q+1$, and $M$ acts doubly transitively on $D$.
Its point stabilizer is $N_B(W)$, of order $qwh$, and
$|M|=(q+1)qwh$. If $E$ is the kernel on $D$ and $L=M/E$, then
\begin{equation}\label{ce:eq:kernel}
 E\le M_P,\qquad E\cap N_S(W)=W,
\end{equation}
and the image of $N_S(W)$ is a normal Sylow $p$-subgroup of $L_P$,
regular on $D\setminus\{P\}$. The quotient by this Sylow subgroup
is cyclic. These assertions hold for $W$ itself, without assuming
that $W$ is normal in $S$ or that $D$ is a block for $G$.

\subsection{The local action}

\begin{lemma}\label{ce:lem:cyclic}
One has $C_S(U)=U$. Consequently $S/U$ and $W$ are cyclic,
and $W$ acts faithfully on $U$.
Writing $|S/U|=p^c$ and $|W|=p^d$, we have
\begin{equation}\label{ce:eq:exponents}
 1\le d\le c\le a-1,\qquad |C_U(W)|=p^{a-d}.
\end{equation}
\end{lemma}
\begin{proof}
An element of $C_S(U)$ induces a $p$-power order automorphism of
the simple group $T$ centralizing a Sylow $p$-subgroup.
Gross's theorem, in \cite[Theorem~4.1]{GGLN}, makes that
automorphism inner. Since $G\le\Aut(T)$, the element lies in
$T$, hence in $U$. The reverse inclusion holds because $U$ is abelian.
Thus $S/U$ embeds in $\Aut(U)$. Its $p$-part is cyclic of order
$p^{a-1}$, and $W\cap U=1$.
A subgroup of order $p^d$ in the $p$-part of
$\Aut(C_{p^a})$ fixes a subgroup of order $p^{a-d}$.
\end{proof}

\begin{lemma}\label{ce:lem:centralizer}
Put $J=C_T(W)$. Then
\[
 J=N_T(W)\triangleleft M,\qquad
 C_U(W)\in\Syl_p(J).
\]
The image $\overline J=JE/E$ is a nontrivial normal subgroup of $L$.
It has cyclic Sylow $p$-subgroups of order $|C_U(W)|$.
The kernel $J\cap E$ is a soluble $p'$-group.
\end{lemma}
\begin{proof}
If $t\in N_T(W)$, then $[t,W]\le T\cap W=1$, proving the
first equality. Normality follows since $T\triangleleft G$.
The order formula for $M$ shows that $N_S(W)\in\Syl_p(M)$, so
\[
 J\cap N_S(W)=C_U(W)\in\Syl_p(J).
\]
This group is nontrivial by \eqref{ce:eq:exponents}.
The kernel $E$ is soluble because $E\le M_P\le B$.
By \eqref{ce:eq:kernel}, its intersection with $C_U(W)$ is trivial.
Since $J\cap E$ is normal in $J$, it is a $p'$-group.
This proves all assertions about the image.
\end{proof}

\begin{proposition}\label{ce:prop:local-list}
Exactly one of the following local possibilities remains:
\begin{enumerate}
\item $q=p$, with $p$ a Mersenne prime and $L$ soluble affine.
Here $J$ is soluble and $|C_U(W)|=p$.
\item $q=p\ge5$, with
$\PSL_2(p)\le L\le\PGL_2(p)$ in its natural action.
Here $J$ has precisely one nonabelian composition factor, $\PSL_2(p)$,
and $|C_U(W)|=p$.
\item $(p,q)=(3,27)$ and
$L\cong\operatorname{P\Gamma L}_2(8)$ in degree $28$.
Here $J/\Rad(J)\cong\PSL_2(8)$ and $|C_U(W)|=9$.
\end{enumerate}
\end{proposition}
\begin{proof}
Apply Hering's classification \cite[Theorem~2.4]{Hering}, in the
explicit form recorded in \cite[Result~2.24]{GK19}, to the normal
closure $V$ in $L$ of the image of $N_S(W)$.
Its hypotheses are exactly transitivity, $|D|>2$, and a normal
subgroup of $L_P$ sharply transitive on $D\setminus\{P\}$.
The possibilities are the natural linear, unitary, Suzuki and Ree
actions, the degree-$28$ semilinear action, a sharply doubly
transitive group, and exceptional actions of odd square degree.
Since $|D|=q+1$ is even, the Suzuki and odd-square alternatives
are excluded.

In the sharply doubly transitive case, the finite Frobenius kernel
is elementary abelian and regular. It is characteristic in $V$,
so $L$ is affine. The small soluble linear action has the same
property. Otherwise the natural actions have simple socle.
Central covering groups in the group-space statement do not give
extra faithful actions: their centers would be normal in the
primitive group $L$, hence transitive, but their orders are too
small for that. The simple socle is normal in $L$, and its
centralizer in $L$ is trivial. Indeed a nontrivial such centralizer
would be transitive by primitivity, forcing the socle to be
semiregular. Thus $L$ is almost simple. In degree $28$ the normal
closure is already $\operatorname{P\Gamma L}_2(8)$, whose
characteristic simple subgroup is $\PSL_2(8)$; the same centralizer
argument gives $L=\operatorname{P\Gamma L}_2(8)$.

In the affine case the regular normal subgroup has order $2^b$,
so $q=p^e=2^b-1$. If $e$ is even, reduction modulo $8$ is impossible.
If $e>1$ is odd, the factorization
$p^e+1=(p+1)(p^{e-1}-p^{e-2}+\cdots-p+1)$ has a nontrivial
odd factor. Hence $e=1$. The whole affine group is soluble, since
its regular normal subgroup and its point stabilizer are soluble.
Its Sylow $p$-subgroup has order $p$, so the nontrivial $p$-part
of $\overline J$ has that order.

In a non-affine case $L$ is almost simple, so its nontrivial normal
subgroup $\overline J$ contains its simple socle. That socle must
have cyclic Sylow $p$-subgroups by Lemma~\ref{ce:lem:centralizer}.
In the natural linear family the Sylow subgroup is elementary
abelian of order $p^e$, forcing $e=1$.
The unitary and proper Ree cases have noncyclic Sylow subgroups
and are excluded. The soluble small linear actions are already affine;
the small Ree action is the degree-$28$ exception.

In that exception the socle $\PSL_2(8)$ has cyclic Sylow
$3$-subgroups of order $9$, whereas a Sylow $3$-subgroup of $L$
is nonabelian of order $27$; see \cite[Proposition~2.7(iii)]{PS11}.
Thus $\overline J=\PSL_2(8)$. The statements about composition
factors follow since $J\cap E$ is soluble.
\end{proof}

\begin{corollary}\label{ce:cor:sizes}
In either case with $q=p$,
\begin{equation}\label{ce:eq:ordinary}
 w=p^{a-1},\qquad S=U\rtimes W,\qquad
 s=p^{2a-1},\qquad m=p^a.
\end{equation}
In the exceptional case,
\begin{equation}\label{ce:eq:exceptional}
 p=3,\quad a\ge3,\quad w=3^{a-2},\quad
 |S/U|=3^{a-1},\quad s=3^{2a-1},\quad m=3^{a+1}.
\end{equation}
\end{corollary}
\begin{proof}
When $|C_U(W)|=p$, \eqref{ce:eq:exponents} gives $d=a-1$,
then $c=a-1$ and $S=UW$, proving \eqref{ce:eq:ordinary}.
In the exception $d=a-2\ge1$. Set $N=N_S(W)$.
Since $W\cap U=1$,
\[
 N\cap UW=C_U(W)W,\qquad [N:N\cap UW]=27/9=3.
\]
The image of $N$ in $S/U$ therefore has order $3w$.
Thus $c\ge d+1=a-1$, and equality follows from
\eqref{ce:eq:exponents}. This gives \eqref{ce:eq:exceptional}.
\end{proof}

\subsection{Descent to a small subfield}

We state the finite Lie group input carefully. The distinction between
curve characteristic $p$ and defining characteristic $\ell$ matters here.

\begin{lemma}[Subfield descent]\label{ce:lem:descent}
Let $T$ be as in \eqref{ce:eq:socle} and let
$A\le\Aut(T)$ be cyclic of order $r=p^d$, with $A\cap T=1$.
Up to conjugacy in $\Aut(T)$ there are an adjoint simple algebraic
group $\mathbf H$ and a Steinberg endomorphism $\sigma$ such that
\[
 H=\mathbf H^{\sigma^r}=\operatorname{Inndiag}(T),
 \qquad A=\langle\sigma|_T\rangle.
\]
Consequently, with $H_0=\mathbf H^\sigma$,
\begin{equation}\label{ce:eq:subfield}
 H_0'\ \le C_T(A)\le H_0.
\end{equation}
If the defining field parameter of $H_0$ is $Q_0$, that of $T$ is
$Q_0^r$, with the same Lie type (up to the usual equivalent twist).
\end{lemma}
\begin{proof}
For a simple group with cyclic Sylow $p$-subgroup, $p$ divides
neither the outer diagonal nor the graph part of its automorphism
group. This is the group-theoretic observation in the proof of
\cite[Lemma~B.1]{Obus}; triality at $p=3$ is impossible because
the relevant Sylow $3$-subgroups are noncyclic. The same observation
excludes the graph-field triality case.
Thus an outer $p$-subgroup is conjugate into the cyclic field part.
In particular $A\cap H=1$, and its image has order $r$.
Here this conjugacy statement concerns the image in the outer
group; conjugacy of the actual lift is proved below. The required
root of the defining Steinberg map $F$ can be chosen explicitly:
\begin{itemize}[leftmargin=*,itemsep=2pt]
\item If $F=F_Q$ is untwisted, write $Q=\ell^f$ and set
 $Q_0=\ell^{f/r}$, $\sigma=F_{Q_0}$.
\item If $F=\tau F_Q$ has an ordinary diagram twist of order
 $t=2$ or $3$, then $(r,t)=1$. Choose $b$ with $br\equiv1\pmod t$
 and put $\sigma=\tau^b F_{Q_0}$. The two factors commute,
 so $\sigma^r=F$.
\item For a Suzuki or Ree endomorphism, write $F=\rho^f$,
 where $f$ is odd and $\rho^2=F_\ell$. Set $\sigma=\rho^{f/r}$.
 Since $r\mid f$ is odd, this retains the same twisted type.
\end{itemize}
In each case $H=\mathbf H^F=\mathbf H^{\sigma^r}$ and the
restriction of $\sigma$ generates the prescribed outer subgroup
of order $r$. In the ordinary triality case $p=3$ has already
been excluded by noncyclicity of the Sylow $3$-subgroup. Thus
the division by $r$ does not discard a twist or change type.

Choose a generator of $A$ projecting to $\sigma|_T$ modulo $H$,
and write it as $t\sigma$ with $t\in H$. Since its order is $r$,
\[
 t\,\sigma(t)\cdots\sigma^{r-1}(t)=1.
\]
Here the inner-diagonal realization is faithful.
Lang--Steinberg \cite{DM} gives $x\in\mathbf H$ with
$t=x^{-1}\sigma(x)$. Telescoping the preceding product gives
$x^{-1}\sigma^r(x)=1$, so $x\in H$.
Conjugation by $x$ therefore makes $A=\langle\sigma|_T\rangle$.
This argument uses the order of the actual lift, not a coprime-action
assumption. The prime-power statement used here is precisely the preceding
Lang--Steinberg calculation, applied to the actual lift of order $r$.

Now $C_H(A)=H_0$ and $C_T(A)=T\cap H_0$.
Since $H/T$ is abelian, $H_0'\le T$, proving
\eqref{ce:eq:subfield}.
\end{proof}

\begin{proposition}\label{ce:prop:families}
Under \eqref{ce:eq:config}--\eqref{ce:eq:unique}, the remaining possibilities
for $T$ are contained in the following list:
\[
\begin{array}{ccll}
\toprule
p&q&T&\text{range}\\
\midrule
3&3&\PSL_2(2^{\,3^{a-1}})&a\ge2\\
5&5&\PSL_2(4^{\,5^{a-1}})&a\ge2\\
7&7&\PSL_3(2^{\,7^{a-1}})&a\ge2\\
3&27&\PSL_2(8^{\,3^{a-2}})&a\ge3\\
\bottomrule
\end{array}
\]
\end{proposition}
\begin{proof}
Apply Lemma~\ref{ce:lem:descent} with $A=W$.
First suppose $J$ is soluble. Then $H_0'$ is soluble.
The small soluble Lie groups, in their projective forms, are
$\PSL_2(2)$, $\PSL_2(3)$, $\PSU_3(2)$ and $ {}^2B_2(2)$.
Passing to the inner-diagonal form does not change this list
of possibilities. The case $\PSL_2(3)$ would force
$p=\ell=3$, contrary to cross characteristic.
The group $\PSU_3(2)\cong C_3^2\rtimes Q_8$ has its normal
$C_3^2$ in its derived subgroup. Thus $H_0'$, and hence $J$,
contains $C_3^2$, incompatible with the cyclic Sylow subgroup of $J$.
The Suzuki case forces $p=5$, whereas the affine case requires
a Mersenne prime. Therefore only $\PSL_2(2)$ remains, and $p=3$.
Equation \eqref{ce:eq:ordinary} gives the first row.

Next suppose $J$ has the nonabelian composition factor
$\PSL_2(p)$, with $p\ge5$. The simple section of $H_0'$ must
be that group. In characteristic $\ell\ne p$, the exceptional
isomorphisms leave exactly
\[
 \PSL_2(4)\cong\PSL_2(5),\qquad
 \PSL_3(2)\cong\PSL_2(7);
\]
see \cite[Chapter~5, Table~2]{HHM}. This gives the second
and third rows using \eqref{ce:eq:ordinary}.
The other nonsimple small Lie groups have respective nonabelian
sections $A_6$, $\PSU_3(3)$, $\PSL_2(8)$ and the Tits group;
none is $\PSL_2(p)$ for prime $p\ge5$.

Finally, in the degree-$28$ case the section is $\PSL_2(8)$.
Its alternative small Ree realization $ {}^2G_2(3)'$ is in
characteristic $3=p$ and is excluded. Hence $Q_0=8$ and
the type is $A_1$. Equation \eqref{ce:eq:exceptional} gives
the final row.
\end{proof}

\begin{remark}
The remaining nonsimple small parameters are $B_2(2)$, $G_2(2)$,
${}^2G_2(3)$ and ${}^2F_4(2)$; their derived simple sections are
covered in the proof above. The cross-characteristic exceptional
isomorphisms used there are the six listed in
\cite[Chapter~5, Table~2]{HHM}:
\[
\begin{array}{c@{\;\cong\;}l@{\qquad}c@{\;\cong\;}l}
 \operatorname{SL}_2(4)&\PSL_2(5)&
 \operatorname{SL}_3(2)&\PSL_2(7)\\
 \PSL_2(9)&\operatorname{Sp}_4(2)'&
 {}^2G_2(3)'&\operatorname{SL}_2(8)\\
 G_2(2)'&\operatorname{SU}_3(3)&
 \operatorname{SU}_4(2)&\operatorname{PSp}_4(3).
\end{array}
\]
These identifications are used only for the corresponding simple sections.
\end{remark}

\subsection{Torus normalizers}

\begin{lemma}\label{ce:lem:transitive}
The group $T$ is transitive on $\Delta$. In particular,
\begin{equation}\label{ce:eq:congruence}
 |\Delta|=[T:N_T(U)]\equiv 1+m\pmod s.
\end{equation}
\end{lemma}
\begin{proof}
For $g\in G$, the subgroups $U^g$ and $U$ are Sylow subgroups
of $T$, since $T\triangleleft G$. Hence $U^g=U^t$ for some
$t\in T$. Their unique fixed points are equal by
\eqref{ce:eq:unique}, so $P^g=P^t$. This proves transitivity.
The short $S$-orbits in $\Delta$ have lengths $1,m$;
every other $S$-orbit has length $s$. This proves the congruence.
\end{proof}

\begin{lemma}\label{ce:lem:A1}
Suppose $T=\PSL_2(Q)$ with $Q$ even and $p$ odd dividing $Q+1$.
The hypotheses force
\begin{equation}\label{ce:eq:pureA1}
 Q+1=|U|=p^a.
\end{equation}
\end{lemma}
\begin{proof}
A nonsplit torus $C$ of order $Q+1$ contains a Sylow
$p$-subgroup, so we may take $U\le C$.
The torus normalizer contains $C\rtimes\langle j\rangle$,
where $j$ is an involution acting by inversion.
This subgroup normalizes $U$ and hence lies in $T_P$.
Since $T_P/U$ is cyclic, inversion must be trivial on $C/U$.
The group $C/U$ has odd order, so it is trivial.
\end{proof}

\begin{proof}[Proof of Theorem~\ref{ce:thm:elimination}]
We eliminate the four rows of Proposition~\ref{ce:prop:families}.

\smallskip
\noindent\emph{The first row.}
Let $Q=2^{3^{a-1}}$. Then $3\mid Q+1$ and the elementary
valuation formula gives $v_3(Q+1)=a$. Lemma~\ref{ce:lem:A1} requires
\[
 2^{3^{a-1}}+1=3^a.
\]
For $a=2$ this gives $Q=8$. For $a=3$ the left side is
$513>27$, and the strict inequality persists by induction
as $Q$ is cubed when $a$ increases by one.
Thus only $a=2$ remains.
Here \eqref{ce:eq:ordinary} gives $s=27$ and $m=9$.
But in $\PSL_2(8)$ the normalizer of a Sylow $3$-subgroup
has order $18$, so
\[
 |\Delta|=504/18=28,\qquad
 28\not\equiv 1+9\pmod{27}.
\]
This contradicts Lemma~\ref{ce:lem:transitive}.

\smallskip
\noindent\emph{The second row.}
Now $Q=4^{5^{a-1}}$, and $v_5(Q+1)=a$.
Lemma~\ref{ce:lem:A1} requires $Q+1=5^a$.
Already at $a=2$, $Q+1=1025>25$.
Taking fifth powers as $a$ increases proves the strict
inequality for every $a\ge2$.

\smallskip
\noindent\emph{The third row.}
Put $Q=2^{7^{a-1}}$. Since $Q\equiv2\pmod3$, the projective
factor $(3,Q-1)$ is $1$. A Singer torus $C$ has order
\[
 N=Q^2+Q+1,\qquad v_7(N)=a,
\]
and contains $U$. Its normalizer contains the order-three
map acting on $C$ by $c\mapsto c^Q$, obtained from
Frobenius on the three-dimensional extension field.
This map has determinant one.
If $C/U$ has order $k$, cyclicity of $T_P/U$ forces
$k\mid Q-1$. But also $k\mid Q^2+Q+1$, and
\[
 \gcd(Q-1,Q^2+Q+1)=\gcd(Q-1,3)=1.
\]
Hence $N=7^a$.
At $a=2$, $N=128^2+128+1=16513>49$.
Replacing $Q$ by $Q^7$ increases $Q^2+Q+1$ by a factor
greater than $7$, so equality is impossible for all $a\ge2$.
For completeness, the valuation follows from
$N=(Q^3-1)/(Q-1)$ and
$v_7(8^{7^{a-1}}-1)=a$, while $7\nmid Q-1$.

\smallskip
\noindent\emph{The final row.}
Here
\[
 Q=8^{3^{a-2}}=2^{3^{a-1}},\qquad a\ge3.
\]
Lemma~\ref{ce:lem:A1} again requires $Q+1=3^a$.
The strict inequality already proved for the first row
excludes every parameter in this range.
\end{proof}

\begin{corollary}\label{ce:cor:closed}
Under the preceding reductions, a minimum-genus counterexample
cannot lie in the rational one-orbit Case~B with $\core_S(W)=1$.
\end{corollary}
\begin{proof}
Corollary~\ref{corefree:cor:final}, Proposition~\ref{corefree:prop:further},
and Theorem~\ref{ce:thm:elimination} apply successively.
\end{proof}

\section{The affine--Mersenne case}
\label{sec:affine-completion}

Retain the rational one-orbit configuration, now with
\[
 m=p=2^d-1,\qquad W\triangleleft S,\qquad w=|W|>1,\qquad s=pw.
\]
Put $D=\{P\}\cup A=\Fix_X(W)$ and $M=N_G(W)$. The preceding
reduction gives $B=M_P=S\rtimes C_h$, $|D|=p+1$, and a faithful
affine induced group $L=M^D$. As before, $G$ has exactly two short
orbits, $X/G$ is rational, and the second orbit is tame with inertia
order $t$. We do not need to assume that $X/W$ has positive genus.

\begin{theorem}[The affine bound]\label{aff:thm:bound}
Under these hypotheses,
\[
 |G|<200g^2.
\]
Hence this case cannot occur in a counterexample to
Theorem~\ref{thm:main}.
\end{theorem}

The argument keeps the permutation kernel throughout. The elementary
abelian group below acts on a further quotient of $X/W$; no lift to
$\Aut(X)$ is needed.

\subsection{A tame elementary abelian genus estimate}

\begin{lemma}\label{aff:lem:rank-genus}
Let $V\cong(C_2)^d$ act faithfully on a smooth projective curve $Z$
over an algebraically closed field of characteristic different from $2$.
If $d\ge3$, then
\[
 g(Z)-1\ge 2^{d-2}(d-3).
\]
\end{lemma}
\begin{proof}
Write $z_0=g(Z/V)$ and let $r$ be the number of branch points of
$Z\to Z/V$. Every nontrivial inertia group has order $2$, because
the action is tame and $V$ has exponent $2$. Thus Hurwitz gives
\[
 g(Z)-1=2^{d-2}(4z_0-4+r).
\]
If $r>0$, Kummer theory gives
\[
 d\le 2z_0+r-1.
\]
Here is the dimension argument. The cover corresponds to a
$d$-dimensional subspace of $k(Z/V)^*/k(Z/V)^{*2}$ whose odd
valuations are supported on the $r$ branch points. Their parity
vectors have sum zero, a space of dimension $r-1$. The kernel
consists of square classes with everywhere even divisor and embeds
in $\operatorname{Pic}^0(Z/V)[2]$, of dimension $2z_0$; constants
are squares. This proves the inequality. Therefore
\[
 4z_0-4+r\ge d+2z_0-3\ge d-3.
\]
If $r=0$, the same argument gives $d\le2z_0$. Consequently
$4z_0-4\ge2d-4\ge d-3$, which gives the required bound also
in the unramified case.
\end{proof}

\subsection{The two intermediate quotients}

Let $E$ be the kernel of $M\to L$. Since $W$ fixes $D$ and
$S/W$ acts regularly on $D\setminus\{P\}$,
\[
 W\le E\le B,\qquad E\cap S=W.
\]
Set
\[
 Y=X/W,\qquad y=g(Y),\qquad C=E/W,\qquad e=|C|,
 \qquad Z=X/E=Y/C,\qquad z=g(Z).
\]
Then $C$ is cyclic of order prime to $p$, because it embeds in
$B/S$. Both $M/W$ on $Y$ and $L=M/E$ on $Z$ act faithfully.
The $p+1$ points in $D$ have distinct images on $Y$, and $C$
fixes every one of them.

\begin{lemma}\label{aff:lem:kernel}
There is an integer $f\mid d$ such that
\[
 h=ef,\qquad
 L\cong (C_2)^d\rtimes(C_p\rtimes C_f).
\]
Moreover
\begin{equation}\label{aff:eq:cyclic}
 (p-1)e\le 2y+p-1.
\end{equation}
For $d\ge5$ one also has
\begin{equation}\label{aff:eq:large-genus}
 y-1\ge e\,\frac{(p+1)(d-3)}4.
\end{equation}
\end{lemma}
\begin{proof}
The regular normal subgroup $V$ of $L$ has order $p+1=2^d$.
Identify the point stabilizer $L_P$ with a subgroup of
$\operatorname{GL}_d(2)$. Its normal subgroup $S/W=C_p$ is
regular on the nonzero vectors, hence is a Singer group.
Its centralizer is $\mathbb F_{2^d}^*$ and its normalizer is
$\mathbb F_{2^d}^*\rtimes
\operatorname{Gal}(\mathbb F_{2^d}/\mathbb F_2)$.
Indeed the algebra generated by a Singer generator is the full
field $\mathbb F_{2^d}$, and a normalizing linear map induces an
automorphism of that field. Thus $L_P/C_p$ has order $f\mid d$.
Since $|L_P|=|B/E|=ph/e$, we obtain $h=ef$ and the asserted
semidirect product.

Hurwitz for $Y\to Y/C$, using only the $p+1$ fixed points, gives
\[
 2y-2\ge -2e+(p+1)(e-1).
\]
This is \eqref{aff:eq:cyclic}. The group $V\le L$ acts faithfully
on $Z$, so Lemma~\ref{aff:lem:rank-genus} gives
\[
 z-1\ge2^{d-2}(d-3)=\frac{(p+1)(d-3)}4.
\]
Hurwitz for $Y\to Z$ gives $y-1\ge e(z-1)$ and proves
\eqref{aff:eq:large-genus}.
\end{proof}

\subsection{Genus and different estimates on the original curve}

\begin{lemma}\label{aff:lem:local}
In the notation $d_P(S)=s+v$, one has
\begin{equation}\label{aff:eq:gv}
 g\ge wy+p(w-1),\qquad v\ge s-2+2wy.
\end{equation}
Consequently, for every Mersenne prime in this configuration,
\begin{equation}\label{aff:eq:uniform}
 sh<10g,\qquad sh<5v.
\end{equation}
\end{lemma}
\begin{proof}
Every point of $D$ is totally ramified for $X\to Y$, and there
is no ramification elsewhere: outside $D$ the $S$-action is free.
Since $M$ normalizes $W$ and is transitive on $D$, the local
different exponent for $W$ has a common value $\delta_0$ on $D$.
Also $\delta_0\ge2(w-1)$. Hurwitz for $W$ and $S$ respectively gives
\[
 2g-2=2w(y-1)+(p+1)\delta_0,\qquad
 2g-2=v-s+p\delta_0.
\]
Subtracting the two identities yields
\[
 v=(p-2)w+2wy+\delta_0.
\]
The lower bound for $\delta_0$ proves \eqref{aff:eq:gv}.

If $d\ge5$, put $A_0=(p+1)(d-3)/4$. By
\eqref{aff:eq:large-genus}, $y\ge1+eA_0$. Therefore
\[
 \frac{sh}{g}<\frac{pf}{A_0}
 \le\frac{4pd}{(p+1)(d-3)}<10,\qquad
 \frac{sh}{v}<\frac{pf}{2A_0}<5.
\]
Here $d/(d-3)\le5/2$ and $p/(p+1)<1$.

Because $2^d-1$ is prime, $d$ is prime. The remaining exponents
are $d=2,3$, giving $(p,d)=(3,2),(7,3)$. Using
\eqref{aff:eq:cyclic}, $f\le d$, $w\ge p$, and
\eqref{aff:eq:gv}, we have
\[
 \frac{sh}{g}
 \le\frac{pf(2y+p-1)}{(p-1)(y+p-p/w)}
 <\frac{2pf}{p-1}\le7.
\]
The strict inequality follows from $p+1-2p/w>0$.
Similarly, since $w>2$,
\[
 \frac{sh}{v}
 \le\frac{pfw(2y+p-1)}
 {(p-1)(w(p+2y)-2)}
 <\frac{pf}{p-1}\le\frac72.
\]
These estimates imply \eqref{aff:eq:uniform} also in the two
small cases.
\end{proof}

\begin{proof}[Proof of Theorem~\ref{aff:thm:bound}]
The two-short-orbit Hurwitz identity is
\[
 |G|=\frac{2(g-1)sht}{\nu},\qquad
 \nu=vt-sh\in\mathbb Z_{>0}.
\]
If $t\le10$, then $\nu\ge1$ and \eqref{aff:eq:uniform} give
\[
 |G|\le2(g-1)sht<200g^2.
\]
If $t>10$, then $sh<5v$ gives $v-sh/t>v/2$, and hence
\[
 |G|=\frac{2(g-1)sh}{v-sh/t}
 <4(g-1)\frac{sh}{v}<20(g-1)<200g^2.
\]
Both cases give the assertion.
\end{proof}

\begin{corollary}[The one-orbit case]
\label{aff:cor:one-orbit}
Assume the preceding reductions for a minimum-genus counterexample
to Theorem~\ref{thm:main}. Rational Case~B cannot have exactly
one further short $S$-orbit.
\end{corollary}
\begin{proof}
The singleton alternative was excluded by
Theorem~\ref{thm:no-singleton}. For a nonsingleton orbit, a
nontrivial core is reduced to the affine--Mersenne case by
Theorem~\ref{na:thm:main}, and this case is excluded by
Theorem~\ref{aff:thm:bound}. A trivial core is excluded by the
core-free reduction and Theorem~\ref{ce:thm:elimination}.
\end{proof}

\begin{remark}
The case of several further short $S$-orbits will be treated through the
numerical package $(\Sigma)$.
\end{remark}

\section{The common numerical package}\label{sec:sigma}
We first treat the case $g(X/S)>0$ and derive the numerical estimates
needed later. The two-short-orbit reduction is taken from Montanucci
\cite{Mon22}.

We now assume
\[
  Y=X/S,
  \qquad y=g(Y)>0,
\]
and throughout this section we suppose
\begin{equation}\label{eq:counterexample-hyp}
  \gamma(X)>0,
  \qquad
  |G|\ge900g^2.
\end{equation}
By the two-short-orbit reduction, $G$ has wild orbit $\Delta=G(P)$ and a tame
short orbit $\Omega$.  Choose $Q_0\in\Omega$ and write
\[
  t=|G_{Q_0}|.
\]
Since $\Omega$ is tame, $t$ is prime to $p$.  Define
\[
  \nu=vt-sh.
\]
Hurwitz for the $G$-cover gives
\begin{equation}\label{eq:global-ratio}
  2g-2
  =|G|\frac{\nu}{sht}.
\end{equation}
In particular $\nu$ is a positive integer.

\begin{lemma}[The elliptic quotient is impossible]\label{lem:y1}
Under \eqref{eq:counterexample-hyp}, one has $y\ne1$.
\end{lemma}

\begin{proof}
Assume $y=1$.  Hurwitz for $X\to Y$ gives $s\le g$.  By
Proposition~\ref{prelim:prop:bounds}(v), in all cases $h\le12$.  Hence
\[
  |G_P|=sh\le12g,
\]
and therefore, with $n=|\Delta|$,
\[
  n=\frac{|G|}{|G_P|}\ge75g.
\]
From \eqref{eq:global-ratio},
\[
  \frac{t}{\nu}=\frac{n}{2g-2}>\frac{75}{2}.
\]
Write $v+1=ah$, set $\ell=at-s$ and $u=a\nu$.  Then
\[
  \nu=h\ell-t,
  \qquad
  u=\ell v-s,
  \qquad
  \frac{t}{\nu}=\frac{s+\ell}{u}.
\]
By Lemma~\ref{reduction:lem:ell}, $\ell\ge2$ would imply
$(s+\ell)/u\le7$.  Hence $\ell=1$, so $\nu=h-t$ and therefore
\[
  h=\nu+t=\nu\left(1+\frac{t}{\nu}\right)>\frac{77}{2},
\]
a contradiction to $h\le12$.
\end{proof}

\begin{proposition}[The positive-genus case]\label{prop:positive-genus}
Under \eqref{eq:counterexample-hyp}, the following hold:
\begin{enumerate}[label=\textup{(\roman*)}]
\item $y\ge28$;
\item $sy\le g$ and
\[
  h\le4y+2,
  \qquad
  sh\le\left(4+\frac2y\right)g;
\]
\item if $v+1=ah$, $\ell=at-s$ and $u=a\nu$, then
\[
  \ell=1,
\]
so
\begin{equation}\label{eq:ell-one-relations}
  at=s+1,
  \qquad
  u=v-s,
  \qquad
  ah=s+u+1,
  \qquad
  h=t+\nu;
\end{equation}
\item
\[
  \frac{s+1}{u}>\frac{2100}{19},
  \qquad
  u<\frac{19}{2100}(s+1);
\]
\item if $k=|S_P^{(2)}|$, then
\[
  1<k<s,
  \qquad
  k<a\le u,
  \qquad
  h\le2u;
\]
\item
\[
  s+1>\frac{2205000}{361}>6108.
\]
\end{enumerate}
\end{proposition}

\begin{proof}
Hurwitz for $X\to Y$ and the contribution at the fixed point $P$ give
\[
  2g-2\ge s(2y-2)+2(s-1)=2sy-2,
\]
so $sy\le g$.  By Lemma~\ref{lem:y1}, $y\ge2$.  The quotient
$C=G_P/S$ acts faithfully on $Y$ and fixes the image of $P$, hence the standard
tame point-stabilizer bound gives
\[
  h\le4y+2.
\]
Thus
\[
  sh\le\left(4+\frac2y\right)g.
\]
If $n=|\Delta|$, then
\[
  n=\frac{|G|}{sh}\ge\frac{900g}{4+2/y}.
\]
Equation \eqref{eq:global-ratio} yields
\begin{equation}\label{eq:tnu-lower}
  \frac{t}{\nu}=\frac{n}{2g-2}
  >\frac{450}{4+2/y}.
\end{equation}

By Lemma~\ref{lem:local-div}, write $v+1=ah$.  Put
\[
  \ell=at-s,
  \qquad
  u=a\nu.
\]
Then
\[
  \nu=h\ell-t>0,
  \qquad
  u=\ell v-s,
  \qquad
  \frac{t}{\nu}=\frac{s+\ell}{u}.
\]
Since $v\ge s-2$,
\[
  u\ge(\ell-1)(s-2)-2.
\]
Lemma~\ref{reduction:lem:ell} gives
\[
  \ell\ge2
  \quad\Longrightarrow\quad
  \frac{s+\ell}{u}\le7.
\]
But \eqref{eq:tnu-lower} is $>90$ for $y\ge2$.  Therefore $\ell=1$, and
\eqref{eq:ell-one-relations} follows.

Since $u=a\nu$ and $\nu\ge1$, $a\le u$.  Hence
\[
  h=\frac{s+u+1}{a}
  \ge\frac{s+u+1}{u}
  =\frac{s+1}{u}+1
  >\frac{450}{4+2/y}+1.
\]
Combining with $h\le4y+2$ gives
\[
  4y+1>\frac{450}{4+2/y},
\]
which is equivalent to
\[
  8y^2-219y+1>0.
\]
Thus $y\ge28$.  Returning to \eqref{eq:tnu-lower},
\[
  \frac{s+1}{u}=\frac{t}{\nu}
  >\frac{450}{4+2/28}
  =\frac{2100}{19},
\]
which proves (iv).

From \eqref{eq:ell-one-relations}, $v=s+u$, and hence
\[
  d_P(S)=s+v=2s+u.
\]
Therefore
\begin{equation}\label{eq:higher-sum}
  \sum_{i\ge2}(|S_P^{(i)}|-1)=u+2.
\end{equation}
The case $S_P^{(2)}=1$ is impossible. Since $u$ is a positive
integer, (iv) implies $s>109$. Thus $S_P^{(2)}=S$ would give
$u\ge s-3$, contradicting (iv).  Hence
\[
  1<k:=|S_P^{(2)}|<s.
\]
The local divisibility relation at the first drop gives
\[
  h\mid \frac{s}{k}-1.
\]
Since
\[
  h=\frac{s+u+1}{a}>\frac{s}{a}
  \qquad\text{and}\qquad
  h<\frac{s}{k},
\]
we obtain $k<a\le u$.

Let $j\ge3$ be the first index for which
$S_P^{(j)}\subsetneq S_P^{(2)}$.  Lemma~\ref{lem:local-div} gives
\[
  h\le(j-1)(k-1),
\]
whereas \eqref{eq:higher-sum} gives
\[
  (j-2)(k-1)\le u+2.
\]
Thus
\[
  h\le u+k+1\le2u.
\]
Finally
\[
  s+u+1=ah\le2u^2,
\]
so $s+1<2u^2$.  Using (iv),
\[
  s+1
  >\frac12\left(\frac{2100}{19}\right)^2
  =\frac{2205000}{361}.
\]
\end{proof}

\begin{corollary}[Large parameters]\label{cor:large-parameters}
Under \eqref{eq:counterexample-hyp},
\[
  s(s+1)>225g,
\]
and hence
\[
  g>225y^2-y.
\]
In particular
\[
  g\ge176373.
\]
Moreover, if $n=|\Delta|$, then
\[
  n\ge\frac{4200}{19}g>221g.
\]
\end{corollary}

\begin{proof}
From \eqref{eq:global-ratio} and $\ell=1$,
\[
  n=(2g-2)\frac{s+1}{u}.
\]
Therefore
\[
  |G|=nsh
  =(2g-2)\frac{s+1}{u}sh.
\]
Since $h\le2u$,
\[
  |G|\le4(g-1)s(s+1).
\]
Together with $|G|\ge900g^2$, this gives
\[
  s(s+1)\ge\frac{225g^2}{g-1}>225g.
\]
Since $s\le g/y$,
\[
  225g<s(s+1)
  \le\frac{g}{y}\left(\frac{g}{y}+1\right),
\]
which gives $g>225y^2-y$.  With $y\ge28$, this yields
$g\ge176373$.

Finally
\[
  sh\le\left(4+\frac2y\right)g
  \le\frac{57}{14}g,
\]
so
\[
  n=\frac{|G|}{sh}
  \ge\frac{900g^2}{(57/14)g}
  =\frac{4200}{19}g.
\]
\end{proof}

\subsection{The numerical package}
The arguments below use only the following consequences of the preceding
analysis, not the positivity of $g(X/S)$ itself.  We shall say that the
configuration satisfies $(\Sigma)$ if, in addition to
\eqref{eq:counterexample-hyp}, the notation above is defined and
\begin{equation*}\tag{\ensuremath{\Sigma}}
\begin{gathered}
 \ell=1,\qquad at=s+1,\qquad u=v-s,\qquad
 ah=s+u+1,\qquad h=t+\nu,\\
 0<u<\frac{19}{2100}(s+1),\qquad h\le2u,\qquad
 1<k:=|S_P^{(2)}|<a\le u,\\
 s+1>6108,\qquad |\Delta|>221g,\qquad s\le\frac g{28}.
\end{gathered}
\end{equation*}
The positive-genus quotient branch satisfies $(\Sigma)$ by
Proposition~\ref{prop:positive-genus} and
Corollary~\ref{cor:large-parameters}.  Proposition~\ref{prop:several-short}
shows that the rational Case~B configuration with at least two further short
$S$-orbits also satisfies $(\Sigma)$, in fact with stronger constants.
This observation lets us separate the long permutation-group argument from
the geometry used to obtain its numerical input.  Notice also that
$u=v-s$ gives
\begin{equation}\label{eq:sigma-higher-sum}
 \sum_{i\ge2}(|S_P^{(i)}|-1)=u+2,
\end{equation}
and the standard computation using $|G|\ge900g^2$ and $h\le2u$ gives
\begin{equation}\label{eq:sigma-size}
 s(s+1)>225g.
\end{equation}

\begin{proposition}[Reduction to $(\Sigma)$]
\label{reduction:prop:sigma}
Every counterexample of minimum genus to
Theorem~\ref{thm:main} satisfies $(\Sigma)$.
\end{proposition}
\begin{proof}
If $g(X/S)>0$, apply Lemma~\ref{lem:y1},
Proposition~\ref{prop:positive-genus}, and
Corollary~\ref{cor:large-parameters}.
Otherwise Deuring--Shafarevich and $\gamma>0$ show that
$S$ has a short orbit besides $\{P\}$.
All such orbits lie in $\Delta$, since $\Omega$ is tame.
Case A is excluded by Corollary~\ref{cor:caseA}.
In Case B, exactly one additional short orbit is excluded
by Corollary~\ref{aff:cor:one-orbit}; at least two additional
short orbits give $(\Sigma)$ by
Proposition~\ref{prop:several-short}.
\end{proof}

\subsection{The second ramification group is characteristic}

The small different excess in $(\Sigma)$ gives more than noncyclicity of $S$.

\begin{lemma}\label{normalizer:lem:local}
Let $S$ be a nontrivial finite $p$-group acting faithfully on $\kk((z))$
and fixing the residue field, where $K$ is algebraically closed of
characteristic $p$. Write $S_i$ for its lower ramification groups and
$s=|S|$, and suppose $S_0=S_1=S$. Put
\[
 D=\sum_{i\ge2}(|S_i|-1).
\]
If
\begin{equation}\label{normalizer:eq:small-D}
 D<\frac{p-1}{p}s,
\end{equation}
then
\[
 S_2=S'=\Phi(S).
\]
\end{lemma}

\begin{proof}
Let $j\ge1$ be the last lower break. If $S_2=1$, then $S=S_1/S_2$
is elementary abelian and the assertion follows. Suppose $S_2\ne1$.
The last upper break is
\[
 \beta=1+\frac1s\sum_{i=2}^{j}|S_i|.
\]
Each group in this sum is nontrivial and has order at least $p$, so
\[
 \beta-1
 \le\frac{p}{(p-1)s}\sum_{i=2}^{j}(|S_i|-1)
 =\frac{pD}{(p-1)s}<1.
\]
Thus $\beta<2$.

Consider the abelian quotient $A=S/S'$. Upper numbering is compatible
with quotients. Its upper breaks are integers by Hasse--Arf
\cite[Chapter~V, \S7]{SerreLF}, and lie
between $1$ and $\beta<2$. Hence $A^v=1$ for every $v>1$.
For sufficiently small $\epsilon>0$, $S^{1+\epsilon}=S_2$.
Compatibility with quotients gives $S_2\le S'$.
Conversely, $S/S_2$ is elementary abelian, so
\[
 S'\le\Phi(S)=S'S^p\le S_2.
\]
The three groups are equal.
\end{proof}

\begin{corollary}\label{normalizer:cor:local}
Under $(\Sigma)$ one has
\[
 \boxed{\ S_P^{(2)}=S'=\Phi(S)\ne1.\ }
\]
In particular $S$ is nonabelian, and $S_P^{(2)}$ is characteristic in $S$.
Furthermore
\[
 N_G(S)\le N_G(S_P^{(2)}).
\]
\end{corollary}

\begin{proof}
The different formula gives $D=u+2$. By $(\Sigma)$ and $p\ge3$,
\[
 D<\frac{19}{2100}(s+1)+2
 <\frac23s\le\frac{p-1}{p}s.
\]
Apply Lemma~\ref{normalizer:lem:local}. The remaining assertions are immediate.
\end{proof}

\begin{remark}
The conclusion makes $S_P^{(2)}$ intrinsic to the abstract group $S$. It
does not imply that $S$ is Sylow in $G$ or that $S'$ is normal in $G$.
\end{remark}

\section{Faithfulness on the wild orbit}\label{sec:faithfulness}
\label{sec:sigma-action}
Assume \eqref{eq:counterexample-hyp} and $(\Sigma)$, and write
$n=|\Delta|$.

\begin{lemma}[A fixed-point bound]\label{lem:fixed-point-bound}
Every nontrivial automorphism of $X$ fixes at most $2g+2$ points.
\end{lemma}

\begin{proof}
If an automorphism fixes a set $F$, then so does a subgroup of prime order in
the cyclic group it generates.  For a prime $\ell\ne p$, Hurwitz for the tame
order-$\ell$ subgroup gives $|F|\le2g+2$.  For $\ell=p$, each fixed point
contributes at least $2(p-1)$ to the different, giving an even stronger bound.
\end{proof}

\begin{proposition}[Faithfulness on the wild orbit]\label{prop:faithful}
Assume \eqref{eq:counterexample-hyp} and $(\Sigma)$. Then the action of $G$ on $\Delta$ is faithful.
\end{proposition}

\begin{proof}
By $(\Sigma)$,
\[
   |\Delta|>221g>2g+2.
\]
If the kernel of the action on $\Delta$ were nontrivial, any nontrivial element
in the kernel would fix all points of $\Delta$, contradicting
Lemma~\ref{lem:fixed-point-bound}.
\end{proof}

\section{Normal-subgroup reductions under \texorpdfstring{$(\Sigma)$}{Sigma}}
\label{sec:imprimitive-normal}
Assume \eqref{eq:counterexample-hyp} and $(\Sigma)$, and put
\[
 n=|\Delta|.
\]
Thus $n>221g$.  The following reductions do not assume that the action on
$\Delta$ is primitive.

\begin{proposition}[Normal $p$-subgroups]
\label{prop:normal-p}
If $n>2(g-1)$, every normal $p$-subgroup $U\triangleleft G$ acts freely on
$X$. If $U\ne1$, let
\[
 Z=X/U,
 \qquad
 z=g(Z),
 \qquad
 \bar\gamma=\gamma(Z).
\]
Then
\[
 g-1=|U|(z-1),
 \qquad
 \gamma-1=|U|(\bar\gamma-1),
\]
and
\[
 z\ge2,
 \qquad
 \bar\gamma>0.
\]
\end{proposition}

\begin{proof}
Normality and transitivity on $\Delta$ imply that
\[
 |U_R|=e
\]
is independent of $R\in\Delta$. Every short $U$-orbit is contained in
$\Delta$: stabilizers on the tame orbit $\Omega$ have order prime to $p$, and
outside the two short $G$-orbits the $G$-stabilizer is trivial.

Assume $e>1$. There are $ne/|U|$ short $U$-orbits, each of size $|U|/e$.
The Deuring--Shafarevich formula gives
\[
 \gamma-1
 =|U|(\bar\gamma-1)+n(e-1).
\]
By Proposition~\ref{prelim:prop:bounds}(i),
\[
 |U|\le3(g-1).
\]
Since $\gamma\le g$ and $\bar\gamma\ge0$,
\[
 n(e-1)\le \gamma-1+|U|\le4(g-1).
\]
As $e\ge p\ge3$, this implies
\[
 n\le2(g-1),
\]
a contradiction. Hence $e=1$, so $U$ has no short orbit on $X$ and acts
freely. Hurwitz and Deuring--Shafarevich now give the displayed equalities.
The first implies $z\ge2$. If $\bar\gamma=0$, the second would give
$\gamma=1-|U|<0$, which is impossible.
\end{proof}

\begin{corollary}[A minimum-genus counterexample has trivial $p$-core]
\label{cor:minimal-pcore}
Suppose that $(X,G)$ is a counterexample of minimum genus to
Theorem~\ref{thm:main}. If $n>2(g-1)$, then
\[
 O_p(G)=1.
\]
In particular this holds for every minimum-genus counterexample satisfying $(\Sigma)$.
\end{corollary}

\begin{proof}
Suppose that $1\ne U\triangleleft G$ is a normal $p$-subgroup. By
Proposition~\ref{prop:normal-p}, $Z=X/U$ has genus $z\ge2$, positive
$p$-rank, and
\[
 g-1=|U|(z-1).
\]
The induced action of $G/U$ on $Z$ is faithful. Put $M=|U|$. Since
\[
 g=M(z-1)+1,
\]
one has
\[
 g-\sqrt M\,z
 =(\sqrt M-1)\bigl(\sqrt M(z-1)-1\bigr)>0.
\]
Consequently
\[
 |G/U|
 \ge 900\frac{g^2}{M}
 >900z^2,
\]
contradicting the minimality of $g$. Notice that the minimum is taken among
all counterexamples to Theorem~\ref{thm:main}; the quotient need not remain
satisfying $(\Sigma)$.
\end{proof}

\begin{proposition}[Normal tame subgroups]
\label{prop:normal-tame}
If $n>170(g-1)$ and $R\triangleleft G$ has order prime to $p$, then $R$ acts
freely on $\Delta$. If $R\ne1$, its orbits on $\Delta$ form a nontrivial
system of blocks, every block has size $|R|$, and there are at least three
blocks.
\end{proposition}

\begin{proof}
Let
\[
 e=|R_P|
\]
for $P\in\Delta$; normality makes $e$ constant on $\Delta$. Since $R$ is
tame, the contribution of the points of $\Delta$ to the different of
$X\to X/R$ is $n(e-1)$. Hence Hurwitz gives
\[
 2g-2\ge -2|R|+n(e-1).
\]
Proposition~\ref{prelim:prop:bounds}(iii) yields
\[
 |R|\le84(g-1).
\]
If $e>1$, then
\[
 n\le2(g-1)+2|R|\le170(g-1),
\]
contrary to the hypothesis. Thus $e=1$. The $R$-orbits on $\Delta$ therefore
all have size $|R|$; because $R\triangleleft G$, they form a system of
blocks. Finally
\[
 \frac{n}{|R|}>\frac{170}{84}>2,
\]
so there are at least three blocks.
\end{proof}

\begin{corollary}[Abelian normal subgroups]
\label{cor:abelian-normal}
In a minimum-genus counterexample satisfying $(\Sigma)$, every
nontrivial abelian normal subgroup $A\triangleleft G$ has order prime to $p$
and acts freely on $\Delta$. In particular, $A$ is not transitive on
$\Delta$.
\end{corollary}

\begin{proof}
The $p$-primary subgroup of $A$ is characteristic in $A$ and hence normal in
$G$. It is trivial by Corollary~\ref{cor:minimal-pcore}. The assertion now
follows from Proposition~\ref{prop:normal-tame}, since $n>221g>170(g-1)$.
\end{proof}

\begin{proposition}[Normal tame subgroups act freely on the curve]
\label{prop:normal-tame-etale}
Assume $(\Sigma)$. Every normal $p'$-subgroup
$R\triangleleft G$ acts freely on all of $X$.
If $R\ne1$, write $M=|R|$, $Z=X/R$ and $z=g(Z)$. Then
\[
 g-1=M(z-1),\qquad 2\le z<g.
\]
The action of $\bar G=G/R$ on $Z$ has exactly two short orbits.
Their stabilizers and their local ramification filtrations are isomorphic
to those upstairs, and their lengths are $|\Delta|/M$ and $|\Omega|/M$.
\end{proposition}

\begin{proof}
Proposition~\ref{prop:normal-tame} shows that $R$ acts freely on $\Delta$.
For $Q\in\Omega$, the tame cyclic group $G_Q$ fixes $Q$, so
Proposition~\ref{prelim:prop:bounds}(ii) gives $t\le4g+2$. Thus
\[
 |\Omega|=\frac{|G|}{t}\ge\frac{900g^2}{4g+2}>170(g-1).
\]
Normality makes $e=|R_Q|$ constant on $\Omega$. If $e>1$, Hurwitz for $X\to X/R$, together with
Proposition~\ref{prelim:prop:bounds}(iii), gives
\[
 |\Omega|\le |\Omega|(e-1)\le2g-2+2|R|\le170(g-1),
\]
a contradiction. Hence $R$ also acts freely on $\Omega$, and it acts
freely elsewhere because the $G$-stabilizers there are trivial.
Hurwitz gives $g-1=M(z-1)$ and the stated genus inequalities.

If $\bar P$ is the image of $P\in X$, every element of
$\bar G_{\bar P}$ lifts to an element of $G$ preserving the $R$-orbit
of $P$. Multiplying that lift by a suitable element of $R$ makes it
fix $P$. The map $G_P\to\bar G_{\bar P}$ is consequently surjective,
and it is injective because $R_P=1$. Since $X\to Z$ is etale and
the ground field is algebraically closed, the completed local rings at
$P$ and $\bar P$ are isomorphic. This isomorphism identifies the
stabilizer actions and all lower ramification groups. The two orbit
claims follow as well.
\end{proof}

\begin{corollary}[Tame quotients in a minimum counterexample]
\label{cor:tame-zero-quotient}
Let $(X,G)$ have minimum genus among all counterexamples to
Theorem~\ref{thm:main}, and assume $(\Sigma)$.
If $1\ne R\triangleleft G$ is a $p'$-subgroup, then $Z=X/R$
has genus $z\ge2$ and zero $p$-rank. Moreover
\[
 |G/R|>900z^2,
\]
and $G/R$ has no common fixed point on $Z$.
\end{corollary}

\begin{proof}
Put $M=|R|$. By Proposition~\ref{prop:normal-tame-etale},
$g=M(z-1)+1$, with $2\le z<g$. As in the proof of
Corollary~\ref{cor:minimal-pcore}, $g>\sqrt M\,z$, so
\[
 |G/R|\ge900g^2/M>900z^2.
\]
Minimality of $g$ forces $\gamma(Z)=0$. Since $M\le g-1$,
both short orbits on $Z$ have length greater than $170$:
use $|\Delta|>221g$ and $|\Omega|>170(g-1)$.
Thus neither is a singleton, and $G/R$ has no common fixed point.
\end{proof}

\begin{corollary}[Sylow and normalizer reduction after a tame quotient]
\label{cor:tame-normalizer}
Under the hypotheses of Corollary~\ref{cor:tame-zero-quotient},
\[
 S\in\Syl_p(G),\qquad N_G(S)=C_R(S)G_P,
 \qquad |\Fix_X(S)|=|C_R(S)|.
\]
Every point of $X$ fixed by a nonidentity element of $S$ lies in
the fiber $R(P)$.
\end{corollary}

\begin{proof}
Write $\bar P$ for the image of $P$ and $\bar S$ for the image of $S$.
On a zero $p$-rank curve a nontrivial $p$-group has a unique fixed point
and acts freely elsewhere; see \cite[Result~2.4]{GK19}.
The point stabilizer $(G/R)_{\bar P}$ is isomorphic to $G_P$ by
Proposition~\ref{prop:normal-tame-etale}. The normalizer condition for
$p$-groups, applied to the unique fixed point of $\bar S$, shows that
$\bar S$ is Sylow in $G/R$. Since $R$ is a $p'$-group, $S$ is Sylow
in $G$.

Each nonidentity element of $S$ induces a nonidentity element of
$\bar S$, so its fixed points on $X$ lie over $\bar P$.
As $R$ acts freely, a point $r(P)$ is fixed by all of $S$ if and only if
\[
 \sigma r\sigma^{-1}=r\qquad\hbox{for every }\sigma\in S.
\]
Therefore $\Fix_X(S)=C_R(S)(P)$ and it has the asserted cardinality.
Every element of $N_G(S)$ preserves this set. If $x\in N_G(S)$,
choose $r\in C_R(S)$ with $x(P)=r(P)$. Then $r^{-1}x\in G_P$.
The reverse inclusion $C_R(S)G_P\le N_G(S)$ is immediate.
\end{proof}

\subsection{Normal tame subgroups}

\begin{lemma}[The large TI case]\label{reduction:lem:large-ti}
Let $p$ be odd, $Q\in\Syl_p(F)$ a noncyclic, nonnormal TI subgroup, and
suppose
\[
 |Q|>6107,\qquad I=N_F(Q)=Q\rtimes C,
 \qquad C\text{ cyclic of order prime to }p.
\]
Put $K=O_{p'}(F)$. Then $K$ is cyclic and central, $F/K$ is
almost simple with socle $T_0$ of type $\PSL_2(q)$, $\PSU_3(q)$ or
$\vphantom{}^2G_2(q)$ in characteristic $p$, and
$p\nmid |F/K:T_0|$. The conjugation action on the Sylow
$p$-subgroups is doubly transitive, of degree $|Q|+1$,
and its kernel is $K$.
\end{lemma}
\begin{proof}
The hypotheses of \cite[Theorem~3.16]{GMP12} are: $Q\in\Syl_p(F)$,
$Q$ is noncyclic (and, for $p=2$, not generalized quaternion),
$N_F(Q)=QC$ with $C$ cyclic, $Q$ is a TI subgroup, and $Q\ne F$.
They hold here verbatim: $p$ is odd, $Q$ is noncyclic and nonnormal,
and $I=N_F(Q)=Q\rtimes C$. Therefore
\[
 K=O_{p'}(F)=Z(I),
\]
$F/K$ is almost simple, its socle is one of
$\PSL_2(q)$, $\PSU_3(q)$, or ${}^2G_2(q)$ in characteristic $p$
(the Suzuki case is excluded since $p$ is odd), and $F$ is doubly
transitive on its Sylow $p$-subgroups. As recorded in
\cite[Remark~3.17]{GMP12}, Theorem~3.16 uses CFSG. Since $K\cap Q=1$, the projection
$I\to I/Q\cong C$ is injective on $K$; thus $K$ is cyclic. The kernel of
the Sylow-conjugation action is the core of $I$, which is $K$.

It remains only to exclude a $p$-part in the outer index. The almost simple
group $F/K$ is a $K$-group, and $Q$ has $p$-rank at least two. Thus
\cite[Proposition~2.5]{PS11} applies. The outer-$p$ cases are
$\PSL_2(8):3$ at $p=3$ and ${}^2B_2(32):5$ at $p=5$; their Sylow
$p$-subgroups have orders $27$ and $125$, respectively
\cite[Proposition~2.7(iii),(vi)]{PS11}. Both are excluded by
$|Q|>6107$. Therefore $p\nmid|F/K:T_0|$.
\end{proof}

The following arguments use a minimum among all counterexamples to
Theorem~\ref{thm:main}. They do not require any primitive-action classification.

\begin{lemma}[Two elementary representation bounds]
\label{lem:representation-bounds}
Let $q$ be an odd prime power of characteristic $p$.
\begin{enumerate}[label=\textup{(\roman*)}]
\item Let $U$ be a Sylow $p$-subgroup of $\operatorname{PSU}_3(q)$.
Over a field of characteristic different from $p$, any representation
on which $U'$ acts nontrivially has dimension at least $q$.
\item A faithful representation of $\operatorname{PSL}_2(q)$ over a
field of characteristic different from $p$ has dimension at least
$(q-1)/2$. Consequently the same bound holds for a faithful
representation of ${}^2G_2(q)$, where $p=3$ and $q\ge27$.
\end{enumerate}
\end{lemma}

\begin{proof}
Extend the coefficient field to an algebraic closure. For (i), the
unitary Sylow group, in odd characteristic, is the Heisenberg group
\[
 U=\{(x,y,z):x,y,z\in\FF_q\},
\]
with multiplication
\[
 (x,y,z)(x',y',z')=(x+x',y+y',z+z'+xy').
\]
This model follows from the unitary matrices in \cite[Section~2.2]{GMP12}
by choosing a symplectic basis for their commutator form. Its derived
group and center are $Z=\{(0,0,z)\}$, and
$E=\{(0,y,z)\}$ is an abelian normal subgroup of order $q^2$.
Choose an $E$-weight whose restriction $\lambda$ to $Z$ is nontrivial.
Conjugation by $(x,0,0)$ changes that weight by
$y\mapsto\lambda(xy)$. For $x\ne0$ this character is nontrivial.
Thus the weight has $q$ distinct conjugates, all present in the
representation, proving (i).

For (ii), restrict to the unipotent subgroup $E\cong(\FF_q,+)$ of
$\operatorname{PSL}_2(q)$. Faithfulness supplies a nontrivial $E$-weight.
The diagonal torus of order $(q-1)/2$ acts freely on the nontrivial
weights: its action on $E$ is multiplication by the nonzero squares.
The orbit of this weight therefore has $(q-1)/2$ members. Finally,
${}^2G_2(q)$ contains $\operatorname{PSL}_2(q)$ as the derived group
of an involution centralizer; see \cite[Proposition~3.6]{Ba14}.
\end{proof}

\begin{proposition}[Structure of a tame quotient of zero $p$-rank]
\label{prop:tame-structure}
Let $(X,G)$ be a minimum-genus counterexample to
Theorem~\ref{thm:main} satisfying $(\Sigma)$, and suppose
$1\ne A\triangleleft G$ has order prime to $p$.
Put $M=|A|$, $Z=X/A$, $\bar G=G/A$, and $V=S_P^{(2)}$.
Then
\begin{equation}\label{eq:tame-quotient-identities}
 |\Delta|/M=s+1,\qquad
 2g(Z)-2=u,\qquad
 g(Z/\bar S)=g(Z/\bar V)=0.
\end{equation}
Let $K=O_{p'}(\bar G)$. It is cyclic and central, and
$\bar G/K$ is almost simple. Its socle and the corresponding inertia
groups satisfy one of the following alternatives:
\[
\begin{array}{c|c|c|c}
\text{socle}&s&V=\Phi(S)=S'&\text{range}\\ \hline
\operatorname{PSU}_3(q)&q^3&|V|=q&q\ge19,\\
{}^2G_2(q)&q^3&|V|=q^2&q=3^{2e+1}\ge27.
\end{array}
\]
In the Ree alternative, $K=1$, $\bar G={}^2G_2(q)$, and $h=q-1$.
Moreover,
\begin{equation}\label{eq:tame-noncentral}
 [V,A]\ne1.
\end{equation}
\end{proposition}

\begin{proof}
Corollaries~\ref{cor:tame-zero-quotient} and~\ref{cor:tame-normalizer}
give $\gamma(Z)=0$, no common fixed point for $\bar G$, and
$\bar S\in\Syl_p(\bar G)$. Corollary~\ref{normalizer:cor:local} makes $S$
noncyclic. On $Z$, $\bar S$ is a nonnormal TI subgroup, and its
normalizer is the point stabilizer $\bar S\rtimes C_h$.
Lemma~\ref{reduction:lem:large-ti} applies, since $s+1>6108$.
It gives the cyclic central kernel $K$, the rank-one socle,
and the doubly transitive action of degree $s+1$.
Sylow groups correspond to their unique fixed points, so this
is the wild orbit.

From $|\Delta|=(2g-2)(s+1)/u$ and the etale genus relation, we get
$2g(Z)-2=u$. At $\bar P$ the different exponents for $\bar S$ and
$\bar V$ are respectively $2s+u$ and $2|V|+u$; the latter follows
by intersecting the ramification groups with $V=S_P^{(2)}$.
Both groups ramify at only $\bar P$. Hurwitz now proves both
rationality assertions in \eqref{eq:tame-quotient-identities}.

In the linear case $q=s$ and $h\ge(s-1)/2$, contrary to
$h\le2u<19(s+1)/1050$. Thus only unitary and Ree socles remain.
In both cases the Borel torus acts irreducibly on $S/\Phi(S)$;
see \cite[Proposition~2.7(i)]{PS11}.
As $S/V$ is elementary abelian, $\Phi(S)\le V<S$.
Invariance under that torus gives $V=\Phi(S)$.
For the unitary group, $\Phi(S)=S'=Z(S)$ has order $q$, as follows
from the Heisenberg model above. For the Ree group,
$\Phi(S)=S'$ has order $q^2$; see \cite[Proposition~3.1(3)--(5)]{Ba14}.
In both cases $s=q^3$; hence the previously proved bound on $s$ gives
the stated ranges.

In the Ree case local divisibility gives $h\mid s/|V|-1=q-1$.
On the other hand the inverse image of a socle Borel has order
$|K|q^3(q-1)$ and lies in the point stabilizer. Consequently
$|K|(q-1)\le h$, forcing $K=1$ and $h=q-1$. Now
$|\bar G|=(s+1)sh=|{}^2G_2(q)|$, so $\bar G={}^2G_2(q)$.

Finally suppose $[V,A]=1$. Every point of the fiber $A(P)$ is then
fixed by $V$, with different exponent $2|V|+u$. There is no other
ramification for $V$, because a nonidentity element of $\bar V$ fixes
only $\bar P$ on $Z$. As $2g-2=Mu$, Hurwitz for $X/V$ yields
\[
 Mu=2|V|\bigl(g(X/V)-1\bigr)+M(2|V|+u),
\]
and hence $g(X/V)=1-M<0$, a contradiction.
\end{proof}

\begin{theorem}[No nontrivial normal $p'$-subgroup]
\label{thm:no-tame-core}
If $(X,G)$ is a minimum-genus counterexample to
Theorem~\ref{thm:main} satisfying $(\Sigma)$, then
\[
 O_{p'}(G)=1.
\]
\end{theorem}

\begin{proof}
Suppose otherwise, and choose a minimal nontrivial normal subgroup
$A\triangleleft G$ contained in $O_{p'}(G)$. Retain the notation of
Proposition~\ref{prop:tame-structure}. Since
$g=Mu/2+1$ and $|G|=M(s+1)sh$, the large-group inequality gives
\begin{equation}\label{eq:M-polynomial}
 M<\frac{s(s+1)h}{225u^2}
 \le\frac{2s(s+1)}{225u}.
\end{equation}
Also $u+2=\sum_{i\ge2}(|S_P^{(i)}|-1)\ge |V|-1$.
Consequently the two alternatives give
\begin{equation}\label{eq:M-family-bounds}
 \begin{aligned}
 \operatorname{PSU}_3(q):\quad
 M&<\frac{2q^3(q^3+1)}{225(q-3)}<\frac{q^5}{50},\qquad q\ge19,\\
 {}^2G_2(q):\quad
 M&<\frac{2q^3(q^3+1)}{225(q^2-3)}<\frac{q^4}{100},\qquad q\ge27.
 \end{aligned}
\end{equation}

First suppose $A$ is abelian. Minimal normality gives
$A\cong(C_\ell)^d$ for some prime $\ell\ne p$.
Conjugation on $A$ is a representation of $\bar G$ over $\FF_\ell$.
In the unitary case, $V=S'$ acts nontrivially by
\eqref{eq:tame-noncentral}, so Lemma~\ref{lem:representation-bounds}(i)
gives $d\ge q$. In the Ree case $\bar G={}^2G_2(q)$ is simple;
the representation is nontrivial by the same noncentralization, hence
faithful. Part (ii) gives $d\ge(q-1)/2$. Thus, respectively,
\[
 M\ge2^q>q^5/50,\qquad
 M\ge2^{(q-1)/2}>q^4/100,
\]
contradicting \eqref{eq:M-family-bounds}. The elementary exponential
inequalities hold at $q=19$ and $q=27$, respectively, and their ratios
increase thereafter (differentiate their logarithms).

It remains to consider $A=L^d$, with $L$ nonabelian simple of order
prime to $p$. Conjugation induces maps
\[
 \bar G\longrightarrow\operatorname{Out}(A)
 =\operatorname{Out}(L)^d\rtimes\operatorname{Sym}(d)
 \longrightarrow\operatorname{Sym}(d).
\]
Here $\operatorname{Out}(L)$ is soluble by the Schreier theorem.
Let $D$ be the last derived subgroup of the inverse image in $\bar G$
of the socle of $\bar G/K$. Then $D$ is perfect, $D/(D\cap K)$ is
that simple socle, $D\cap K$ is central of order prime to $p$, and
$\bar S\le D$. The last assertion holds because the quotient of
this inverse image by $D$ is a quotient of the $p'$-group $K$.

If $D$ acts trivially on the $d$ factors, its image in
$\operatorname{Out}(A)$ lies in the soluble group
$\operatorname{Out}(L)^d$, and is trivial by perfectness.
Then the original $p$-group $S$ acts on $A$ by inner automorphisms.
But $\operatorname{Inn}(A)$ has order prime to $p$, so $[S,A]=1$,
contrary to \eqref{eq:tame-noncentral}.
Thus the permutation image of $D$ is nontrivial.
Every proper normal subgroup of $D$ lies in $D\cap K$: if its image
in the simple quotient were the whole quotient, the corresponding
quotient of $D$ would be both abelian and perfect. Hence the kernel
of this permutation action has order prime to $p$.

In the unitary case, the complex permutation representation of degree
$d$ therefore detects $V=S'$, and
Lemma~\ref{lem:representation-bounds}(i) gives $d\ge q$.
In the Ree case $D={}^2G_2(q)$, the permutation representation is
faithful, and part (ii) gives $d\ge(q-1)/2$.
Since $M=|L|^d\ge2^d$, the same exponential inequalities contradict
\eqref{eq:M-family-bounds}. Both possibilities for $A$ are excluded.
\end{proof}

\begin{corollary}[Trivial soluble radical]
\label{cor:radical-free}
A minimum-genus counterexample satisfying $(\Sigma)$ has no nontrivial soluble normal subgroup. Every minimal nontrivial
normal subgroup is a direct power of a nonabelian simple group whose
order is divisible by $p$.
\end{corollary}

\begin{proof}
A nontrivial soluble normal subgroup would contain a minimal nontrivial
normal subgroup of $G$ which is elementary abelian of some prime
characteristic. The characteristic $p$ case is excluded by
Corollary~\ref{cor:minimal-pcore}, and every other characteristic by
Theorem~\ref{thm:no-tame-core}. The assertion about minimal normal
subgroups follows from their standard structure and $O_{p'}(G)=1$.
\end{proof}

\subsection{Normal subgroups meeting the wild inertia}

The following elementary lemma will be used twice.

\begin{lemma}[Normal subgroups on a long wild orbit]\label{lem:long-orbit-normal}
Let $F\le\Aut(X)$ be transitive on an orbit $\Delta$ of length
$n>2g-2$. Suppose $F_P$ is soluble, and write $V=F_P^{(1)}$.
If $L\triangleleft F$ and $L\cap V\ne1$, then $L$ is transitive
on $\Delta$, $X/L$ is rational, and $F/L$ is soluble.
\end{lemma}
\begin{proof}
Put $a=|L\cap V|$, $e=|L_P|$, and $z=g(X/L)$.
Normality makes these stabilizer orders constant along $\Delta$.
The first two ramification groups give
\[
 d_R(L)\ge e+a-2\qquad(R\in\Delta).
\]
If $z\ge1$, Hurwitz yields
$2g-2\ge n(e+a-2)\ge n$, a contradiction. Thus $z=0$.
Let $r$ be the number of $L$-orbits on $\Delta$. Then
$|L|=ne/r$, and Hurwitz gives
\[
 2g-2\ge n\left(e+a-2-\frac{2e}{r}\right).
\]
If $r\ge2$, the expression in parentheses is at least $a-2\ge1$,
again a contradiction. Hence $r=1$. Consequently $F=LF_P$ and
$F/L\cong F_P/(L\cap F_P)$ is soluble.
\end{proof}

The next proposition applies the lemma to the counterexample group.

\begin{proposition}[Normal subgroups meeting $S$]
\label{prop:normal-wild-transitive}
Assume \eqref{eq:counterexample-hyp} and $(\Sigma)$. Let
$1\ne L\triangleleft G$ and suppose
\[
 a:=|L\cap S|>1.
\]
Then
\[
 X/L\cong\PP^1,
\]
$L$ is transitive on $\Delta$, and
\[
 G/L\cong G_P/L_P
\]
is soluble.
\end{proposition}

\begin{proof}
By $(\Sigma)$, $|\Delta|>221g>2g-2$. Apply
Lemma~\ref{lem:long-orbit-normal} with $F=G$ and $V=S$.
\end{proof}

\subsection{A semisimple normal subgroup disjoint from \texorpdfstring{$S$}{S}}

We next eliminate the other possible interaction between a minimal normal
subgroup and the wild inertia.  No primitivity hypothesis is used.

\begin{lemma}[The quotient attached to a disjoint minimal normal subgroup]
\label{lem:disjoint-minimal-quotient}
Let $(X,G)$ be a minimum-genus counterexample to
Theorem~\ref{thm:main} satisfying $(\Sigma)$, and let
\[
 N=T^k\triangleleft G
\]
be a minimal normal subgroup with $T$ nonabelian simple and
$N\cap S=1$. Put
\[
 M=|N|,\qquad Z=X/N,\qquad z=g(Z),\qquad H=G/N.
\]
Then the $N$-action on $X$ is tame,
\[
 M\le84(g-1),
\]
and
\[
 z\ge2,\qquad \gamma(Z)=0.
\]
Moreover $H$ has no common fixed point on $Z$.
\end{lemma}

\begin{proof}
Let $x\in N$ have order $p$. If $x$ fixes a point of $\Delta$, conjugating
that point to $P$ conjugates $x$ into the unique Sylow $p$-subgroup $S$ of
$G_P$, contradicting $N\cap S=1$.  A point of $\Omega$ has $p'$-stabilizer,
and a point outside $\Delta\cup\Omega$ has trivial $G$-stabilizer.  Thus every
$p$-element of $N$ is fixed-point-free.  All inertia groups for the cover
$X\to Z$ are therefore $p'$-groups, so the cover is tame and
Proposition~\ref{prelim:prop:bounds}(iii) gives $M\le84(g-1)$.

We first exclude $z=1$.  The image of $S$ in $H$ is isomorphic to $S$ and
fixes the image $\bar P$ of $P$.  Proposition~\ref{prelim:prop:bounds}(v) gives a point-stabilizer bound of
$12$ on an elliptic curve in odd characteristic, whereas $s>6108$ by
$(\Sigma)$.  Hence $z\ne1$.

Suppose next that $z=0$. Put $J=N_P$. Then $J\triangleleft G_P$ and
$p\nmid|J|$. The local quotient group
\[
 G_P/J=H_{\bar P}
\]
acts on $\PP^1$.  By Herbrand's theorem for quotients in upper numbering; see \cite[Chap.~IV]{SerreLF},
for every $u>0$,
\[
 (G_P/J)^u=G_P^uJ/J.
\]
Since $J$ is a $p'$-group, the positive upper ramification groups upstairs
embed unchanged in the quotient.  Upper and lower numbering have the same
sequence of ramification subgroups, only with different indices; hence the
proper nontrivial group
\[
 1<S_P^{(2)}<S
\]
from $(\Sigma)$ produces a proper nontrivial positive ramification
group in $H_{\bar P}$.  This is impossible for a finite subgroup of
$\operatorname{PGL}_2(K)$ at a wild fixed point: after conjugation the wild
part acts by translations, and its lower filtration is $U_1=U$, $U_2=1$;
see also \cite[Chap.~11]{HKT}.
Thus $z\ne0$.

We have $z\ge2$. Hurwitz gives
\[
 g-1\ge M(z-1).
\]
Consequently
\[
 g\ge M(z-1)+1>\sqrt M\,z,
\]
because
\[
 M(z-1)+1-\sqrt M\,z
 = (\sqrt M-1)\bigl(\sqrt M(z-1)-1\bigr)>0.
\]
Therefore
\[
 |H|=\frac{|G|}{M}>900z^2.
\]
If $\gamma(Z)>0$, this contradicts the minimality of $g$. Hence
$\gamma(Z)=0$.

Finally suppose that $H$ fixes a point of $Z$. Its fiber is a $G$-stable
$N$-orbit of size at most $M\le84(g-1)$, hence a short $G$-orbit. But
\[
 |\Delta|>221g>84(g-1)
\]
and, since $t\le4g+2$,
\[
 |\Omega|=\frac{|G|}{t}
 \ge\frac{900g^2}{4g+2}>170(g-1)>84(g-1).
\]
This is impossible. Thus $H$ has no common fixed point.
\end{proof}

\begin{proposition}[The key Hurwitz identity]
\label{prop:key-semisimple-hurwitz}
Under the hypotheses of Lemma~\ref{lem:disjoint-minimal-quotient}, let
\[
 e=|N_P|,
 \qquad
 b=|N_Q|\quad(Q\in\Omega).
\]
Then
\begin{equation}\label{eq:semisimple-M}
 M=\frac{(2g-2)e}{u}
\end{equation}
and
\begin{equation}\label{eq:key-semisimple-hurwitz}
 \boxed{
 u=2e(z-1)+(s+1)(e-1)
   +s(s+u+1)(b-1).}
\end{equation}
Consequently
\[
 e=b=1,
\]
so $N$ acts freely on $X$, and
\begin{equation}\label{eq:semisimple-genus}
 2z-2=u.
\end{equation}
\end{proposition}

\begin{proof}
Since $\gamma(Z)=0$, every nontrivial $p$-subgroup of $H$ fixes a unique
point and acts freely elsewhere.  The image of $S$ fixes $\bar P$.  If it
were properly contained in a $p$-subgroup of $H$, that larger subgroup would
also fix $\bar P$, contrary to the fact that the $p$-part of
$H_{\bar P}=G_P/N_P$ has order $s$. Hence $S\in\Syl_p(H)$. Distinct Sylow
$p$-subgroups meet trivially, and
\[
 N_H(S)=H_{\bar P}=S\rtimes C_{h/e}.
\]
The group $S$ is noncyclic by Corollary~\ref{normalizer:cor:local}, and it is not
normal because $H$ has no common fixed point. Together with $s+1>6108$, these are precisely the hypotheses
of Lemma~\ref{reduction:lem:large-ti}.

The first part of the proof of Proposition~\ref{prop:tame-structure}, which
uses only these zero-$p$-rank hypotheses, now applies to $(Z,H)$. In
particular the wild $H$-orbit containing $\bar P$ has exactly $s+1$ points.
On the other hand the image of $\Delta$ in $Z$ has length
\[
 \frac{|\Delta|e}{M}.
\]
Since $(\Sigma)$ gives
\[
 |\Delta|=(2g-2)\frac{s+1}{u},
\]
we obtain \eqref{eq:semisimple-M}.

The only ramification of the tame $N$-cover $X\to Z$ occurs over the images
of $\Delta$ and $\Omega$.  Moreover
\[
 |\Omega|
 =\frac{|G|}{t}
 =(2g-2)\frac{s(s+u+1)}{u}.
\]
Hurwitz therefore gives
\[
 2g-2
 =M(2z-2)+|\Delta|(e-1)+|\Omega|(b-1).
\]
Substituting the preceding formulas and dividing by $(2g-2)/u$ yields
\eqref{eq:key-semisimple-hurwitz}.

By $(\Sigma)$, $u<s+1$. If $b>1$, the final term on the right of
\eqref{eq:key-semisimple-hurwitz} is already greater than $u$. Hence $b=1$.
If $e>1$, the middle term is at least $s+1>u$. Hence $e=1$ as well. Thus
$N$ has trivial stabilizer on both short $G$-orbits, and elsewhere the
$G$-stabilizer is trivial; $N$ acts freely on $X$. The identity then reduces
to \eqref{eq:semisimple-genus}.
\end{proof}

\begin{proposition}[Structure after the free semisimple quotient]
\label{prop:free-semisimple-structure}
Retain the hypotheses of Proposition~\ref{prop:key-semisimple-hurwitz} and
put $V=S_P^{(2)}$. Let $K=O_{p'}(H)$. Then $K$ is cyclic and central and
$H/K$ is almost simple. The linear possibility is excluded, and the simple
socle $A_0=\operatorname{soc}(H/K)$ is one of
\[
\begin{array}{c|c|c|c}
A_0&s&V=\Phi(S)=S'&\text{range}\\ \hline
\operatorname{PSU}_3(q)&q^3&|V|=q&q\ge19,\\
{}^2G_2(q)&q^3&|V|=q^2&q=3^{2e_0+1}\ge27.
\end{array}
\]
In the Ree case $K=1$, $H={}^2G_2(q)$ and $h=q-1$. Moreover
\begin{equation}\label{eq:semisimple-noncentral}
 [V,N]\ne1.
\end{equation}
Finally
\begin{equation}\label{eq:semisimple-M-bound}
 M<\frac{2s(s+1)}{225u},
 \qquad
 g-1<\frac{s(s+1)}{225},
\end{equation}
and hence
\[
 \operatorname{PSU}_3(q):\quad M<\frac{q^5}{50},
 \qquad
 {}^2G_2(q):\quad M<\frac{q^4}{100}.
\]
\end{proposition}

\begin{proof}
By Proposition~\ref{prop:key-semisimple-hurwitz}, the map $X\to Z$ is
etale.  Hence the completed local rings at corresponding points, the point
stabilizers, and all ramification filtrations are identified.  The proof of Proposition~\ref{prop:tame-structure}, beginning
with Lemma~\ref{reduction:lem:large-ti}, therefore applies with $A=N$.
This proves the assertions about $K$, the unitary and Ree alternatives,
$V$, and \eqref{eq:semisimple-noncentral}. Notice that the hypothesis
$p\nmid|A|$ in Proposition~\ref{prop:tame-structure} was used before that
point only to obtain an etale zero-$p$-rank quotient; those properties are
already available here.

Since $N$ is free, \eqref{eq:semisimple-genus} gives
$g-1=Mu/2$, while $|\Delta|/M=s+1$ gives
\[
 |G|=M(s+1)sh.
\]
Using $|G|\ge900g^2$, $g>Mu/2$, and $h\le2u$, we obtain
\[
 M<\frac{s(s+1)h}{225u^2}
 \le\frac{2s(s+1)}{225u}.
\]
Multiplication by $u/2$ gives the second inequality in
\eqref{eq:semisimple-M-bound}. The two displayed family bounds follow from
$u+2\ge|V|-1$ exactly as in \eqref{eq:M-family-bounds}.
\end{proof}

\begin{proposition}[The centralizer is trivial]
\label{prop:semisimple-centralizer}
Under the hypotheses of Proposition~\ref{prop:free-semisimple-structure},
\[
 C_G(N)=1.
\]
\end{proposition}

\begin{proof}
Suppose $C_G(N)\ne1$ and choose a minimal normal subgroup
$N_2\triangleleft G$ contained in $C_G(N)$. Since $N$ is nonabelian,
$N_2\ne N$; hence $N_2\cap N=1$ and $[N_2,N]=1$. By
Corollary~\ref{cor:radical-free}, $N_2$ is a direct power of a nonabelian
simple group.

If $N_2\cap S\ne1$, Proposition~\ref{prop:normal-wild-transitive} makes
$G/N_2$ soluble. But $N$ embeds in $G/N_2$, a contradiction. Thus
$N_2\cap S=1$. Lemma~\ref{lem:disjoint-minimal-quotient} and
Proposition~\ref{prop:key-semisimple-hurwitz}, applied to $N_2$, show that
$N_2$ also acts freely on $X$. In particular
\[
 |N_2|\le g-1.
\]

The image $N_2N/N\cong N_2$ is a nontrivial perfect normal subgroup of
$H=G/N$. Its image in $H/K$ cannot be trivial, since $K$ is cyclic. Hence
it contains the simple socle $A_0$. Thus
\[
 |A_0|\le|N_2|\le g-1.
\]
This contradicts \eqref{eq:semisimple-M-bound}. Indeed, in the unitary case
\[
 |A_0|
 =\frac{q^3(q^3+1)(q^2-1)}{\gcd(3,q+1)}
 >\frac{q^3(q^3+1)}{225}>g-1,
\]
and in the Ree case
\[
 |A_0|=q^3(q^3+1)(q-1)
 >\frac{q^3(q^3+1)}{225}>g-1.
\]
Therefore $C_G(N)=1$.
\end{proof}

\begin{theorem}[No minimal normal subgroup is disjoint from $S$]
\label{thm:no-disjoint-minimal}
Let $(X,G)$ be a minimum-genus counterexample to
Theorem~\ref{thm:main} satisfying $(\Sigma)$. Then every minimal
nontrivial normal subgroup $N\triangleleft G$ satisfies
\[
 N\cap S\ne1.
\]
\end{theorem}

\begin{proof}
Suppose that $N=T^k$ is minimal normal and $N\cap S=1$.
Propositions~\ref{prop:free-semisimple-structure} and
\ref{prop:semisimple-centralizer} apply. Since $C_G(N)=1$, conjugation gives
an embedding
\[
 H=G/N\hookrightarrow\operatorname{Out}(N)
 =\operatorname{Out}(T)^k\rtimes\operatorname{Sym}(k).
\]
Let $A_0=\operatorname{soc}(H/K)$, let $E$ be its inverse image in $H$, and
let $D$ be the last term of the derived series of $E$. Then $D$ is perfect and
\[
 D/(D\cap K)\cong A_0.
\]
Moreover $D\cap K$ is a central $p'$-group. Since the $p$-part of $H$ is
$s$, equal to the $p$-part of $A_0$, the quotient $H/E$ has order prime to
$p$, so $S\le E$. The quotient $E/D$ is a quotient of $K/(D\cap K)$ and is
therefore also a $p'$-group. Hence $S\le D$.

Consider the action of $D$ on the $k$ simple direct factors of $N$. If this
action were trivial, the image of $D$ in $\operatorname{Out}(N)$ would lie
in the soluble group $\operatorname{Out}(T)^k$ by the Schreier theorem.
Perfectness would make that image trivial, contradicting the faithful
embedding of $H$ in $\operatorname{Out}(N)$. Thus the factor action is
nontrivial.

Let $J$ be its kernel. The image of $J$ in the simple group $A_0$ is either
trivial or all of $A_0$. In the latter case
$D=J(D\cap K)$, so $D/J$ is abelian; but it is also a quotient of the
perfect group $D$, hence trivial, contrary to the nontrivial factor action.
Therefore
\[
 J\le D\cap K,
\]
and in particular $p\nmid|J|$. Thus the permutation action on the $k$
factors detects $V=S'$. In the unitary case the argument of
Lemma~\ref{lem:representation-bounds}(i), unchanged under the central
$p'$-extension $D\to A_0$, gives
\[
 k\ge q.
\]
In the Ree case $K=1$ and $D={}^2G_2(q)$, so the factor action is faithful;
Lemma~\ref{lem:representation-bounds}(ii) gives
\[
 k\ge\frac{q-1}{2}.
\]
Since every nonabelian simple group has order at least $60$,
\[
 M=|T|^k\ge60^k.
\]
Hence in the unitary case
\[
 M\ge60^q>q^5/50,
\]
and in the Ree case
\[
 M\ge60^{(q-1)/2}>q^4/100.
\]
These inequalities already hold at $q=19$ and $q=27$, respectively, and
the ratios increase thereafter. Both contradict
\eqref{eq:semisimple-M-bound}. Thus $N\cap S\ne1$.
\end{proof}

\begin{corollary}[Quasiprimitive reduction]
\label{cor:residual-quasiprimitive}
Let $(X,G)$ be a minimum-genus counterexample to
Theorem~\ref{thm:main} satisfying $(\Sigma)$. Then $G$ has a unique
minimal normal subgroup
\[
 N=\operatorname{soc}(G)=T^k,
\]
where $T$ is nonabelian simple and $p\mid|T|$. Moreover
\[
 N\cap S\ne1,
 \qquad
 X/N\cong\PP^1,
 \qquad
 N\text{ is transitive on }\Delta,
 \qquad
 G/N\text{ is soluble}.
\]
Every nontrivial normal subgroup of $G$ is transitive on $\Delta$.
Consequently the faithful action of $G$ on $\Delta$ is quasiprimitive.
\end{corollary}

\begin{proof}
By Corollary~\ref{cor:radical-free}, every minimal normal subgroup is a
direct power of a nonabelian simple group, and its simple factors have order
divisible by $p$. Theorem~\ref{thm:no-disjoint-minimal} shows that every
minimal normal subgroup meets $S$ nontrivially. Proposition~\ref{prop:normal-wild-transitive}
therefore makes each such subgroup transitive on $\Delta$ and gives a
soluble quotient.

There cannot be two distinct minimal normal subgroups $N_1,N_2$: they would
commute and meet trivially, so $N_2$ would embed in the soluble group
$G/N_1$, contrary to its nonsolubility. Hence $N$ is unique. Every
nontrivial normal subgroup contains a minimal normal subgroup and therefore
contains $N$; it is consequently transitive on $\Delta$. This is exactly
quasiprimitivity. The remaining assertions follow from
Proposition~\ref{prop:normal-wild-transitive}.
\end{proof}

\section{The almost-simple two-point reduction}
\label{sec:further-residual}

We continue with the configuration obtained in
Corollary~\ref{cor:residual-quasiprimitive}.  Thus $G$ has a unique minimal
normal subgroup
\[
 N=T^k=\operatorname{soc}(G),
\]
where $T$ is nonabelian simple, $N\cap S\ne1$, the group $N$ is transitive
on $\Delta$, $X/N\cong\PP^1$, and $G/N$ is soluble.  The next propositions
sharpen this description.

\begin{proposition}[The socle contains the full wild inertia]
\label{prop:socle-contains-S}
One has
\[
 S\le N.
\]
More precisely, put
\[
 r=|S\cap N|,
 \qquad
 |N_P|=rc,
 \qquad
 |N_Q|=b\quad(Q\in\Omega),
\]
where $c$ and $b$ are prime to $p$. Then
\begin{equation}\label{eq:ct-bh}
 ct=bh.
\end{equation}
If $r<s$, then
\begin{equation}\label{eq:c-r-bound}
 (c-1)(s-r)\le u+2,
\end{equation}
and this contradicts $(\Sigma)$. Hence $r=s$.
\end{proposition}

\begin{proof}
Since $N$ is transitive on $\Delta$, the image $\bar P$ of $P$ in
$X/N\cong\PP^1$ is fixed by $G/N$, and
\[
 G/N\cong G_P/N_P.
\]
The $p$-part of $N_P$ is $S\cap N$, of order $r$, while the quotient of
$N_P$ by this $p$-part embeds in $G_P/S\cong C_h$; hence $|N_P|=rc$ with
$p\nmid c$.

The group $G_P/N_P$ is a finite subgroup of the affine group of the line.
Its wild subgroup has order $s/r$, and its tame complement has order $h/c$.
Thus the local different exponent of the quotient action at $\bar P$ is
\[
 d_{\bar P}(G_P/N_P)
   =\frac{sh}{rc}+\frac{s}{r}-2.
\]
On the other hand
\[
 d_P(G_P)=sh+s+u.
\]
Transitivity of the different in the local tower gives
\[
 d_P(N_P)
 =d_P(G_P)-|N_P|d_{\bar P}(G_P/N_P)
 =u+s(1-c)+2rc.
\label{eq:np-different}
\]

The cover $X\to X/N\cong\PP^1$ can ramify only on $\Delta\cup\Omega$.
Since $N$ is transitive on $\Delta$,
\[
 |N|=rc\,|\Delta|.
\]
Hurwitz, together with
\[
 |\Delta|=(2g-2)\frac{s+1}{u},
 \qquad
 |\Omega|=(2g-2)\frac{s(s+u+1)}{u},
\]
therefore gives
\[
 0=u+(s+1)(1-c)+(s+u+1)(b-1).
\]
Using
\[
 s+1=at,\qquad u=a\nu,\qquad s+u+1=ah,\qquad h=t+\nu,
\]
this is exactly \eqref{eq:ct-bh}.

The first ramification group of $N_P$ has order $r$. Hence
\[
 d_P(N_P)\ge (rc-1)+(r-1).
\]
Comparison with \eqref{eq:np-different} yields
\[
 (c-1)(s-r)\le u+2,
\]
which is \eqref{eq:c-r-bound}. If $r<s$, then $r$ and $s$ are powers of
$p$, so
\[
 s-r\ge s\left(1-\frac1p\right)\ge\frac{2s}{3}.
\]
But $(\Sigma)$ gives
\[
 u+2<\frac{19}{2100}(s+1)+2<\frac{2s}{3},
\]
using $s+1>6108$. Thus \eqref{eq:c-r-bound} forces $c=1$. Then
\eqref{eq:ct-bh} gives $t=bh$, impossible because $b\ge1$ and
$h=t+\nu>t$. Therefore $r=s$, as claimed.
\end{proof}

\begin{corollary}[A two-point socle cover]
\label{cor:two-point-socle-cover}
Write
\[
 |N_P|=sc,
 \qquad
 |N_Q|=b\quad(Q\in\Omega),
 \qquad
 d=|G/N|.
\]
Then
\[
 G/N\cong C_d,
\]
where $p\nmid d$, and
\begin{equation}\label{eq:h-t-cbd}
 h=cd,
 \qquad
 t=bd,
 \qquad
 \nu=(c-b)d.
\end{equation}
Moreover $N$ is transitive on both $\Delta$ and $\Omega$, and
\[
 X\longrightarrow X/N\cong\PP^1
\]
is branched at exactly two points: a wild point with inertia
$S\rtimes C_c$ and a tame point with inertia $C_b$.
Finally
\begin{equation}\label{eq:b-over-cb}
 \frac{b}{c-b}=\frac{s+1}{u}>\frac{2100}{19},
\end{equation}
so in particular
\[
 b\ge111,
 \qquad
 c\ge112.
\]
If $N=T^k$, then
\[
 k\mid d\mid\nu.
\]
\end{corollary}

\begin{proof}
By Proposition~\ref{prop:socle-contains-S}, the quotient $G/N$ has order
prime to $p$.  It fixes the image of $\Delta$ on $X/N\cong\PP^1$.
By Proposition~\ref{prelim:prop:bounds}(vi), a finite $p'$-subgroup of
$\operatorname{PGL}_2(\kk)$ fixing a point is cyclic, so $G/N\cong C_d$.

Since
\[
 d=|G_P/N_P|=\frac{h}{c},
\]
we have $h=cd$.  Equation~\eqref{eq:ct-bh} gives $t=bd$, and then
$h-t=\nu$ yields the last identity in \eqref{eq:h-t-cbd}.
The number of $N$-orbits on $\Omega$ is
\[
 \frac{|\Omega|}{|N|/b}
 =\frac{db}{t}=1,
\]
so $N$ is transitive on $\Omega$.  There is no ramification outside
$\Delta\cup\Omega$, hence the socle cover has exactly the two stated branch
points.

Using $s+1=at$ and $u=a\nu$ together with \eqref{eq:h-t-cbd}, we obtain
\[
 \frac{s+1}{u}=\frac{t}{\nu}=\frac{b}{c-b},
\]
which proves \eqref{eq:b-over-cb} and the numerical conclusions.
Finally, $G$ acts transitively by conjugation on the $k$ simple direct
factors of the minimal normal subgroup $N=T^k$. Since $N$ itself fixes each
factor, this action factors through the cyclic group $G/N=C_d$. Thus
$k\mid d$, and $d\mid\nu$ follows from \eqref{eq:h-t-cbd}.
\end{proof}

\subsection{A geometric simplicity argument}

\begin{proposition}\label{normalizer:prop:direct-product}
Suppose
\[
 N=T_1\times\cdots\times T_k\le\Aut(X),
\]
where each $T_i$ is nonabelian simple. Assume $p>2$ and $\gamma(X)>0$, that $N$ is transitive
on an orbit $\Delta$ with $|\Delta|>2g-2$, and that $N_P$ is soluble.
Put $V=N_P^{(1)}$ and $s=|V|$. If
\begin{equation}\label{normalizer:eq:square}
 s^2>\frac{p}{p-2}(g-1),
\end{equation}
then $k=1$.
\end{proposition}

\begin{proof}
Suppose $k\ge2$, and for each $i$ set
\[
 L_i=\prod_{j\ne i}T_j\triangleleft N.
\]
If $L_i\cap V\ne1$, Lemma~\ref{lem:long-orbit-normal} implies that
$N/L_i\cong T_i$ is soluble, which is impossible. Thus
$L_i\cap V=1$ for every $i$. The projection $N\to T_i$ therefore
restricts to an injection on $V$, and hence
\[
 |T_i|_p\ge s.
\]
A Sylow $p$-subgroup $U$ of $N$ consequently satisfies
\[
 |U|=\prod_i|T_i|_p\ge s^k\ge s^2.
\]
Since $\gamma(X)>0$, Proposition~\ref{prelim:prop:bounds}(i) gives
\[
 |U|\le\frac{p}{p-2}(g-1),
\]
contradicting \eqref{normalizer:eq:square}.
\end{proof}

\begin{proposition}[The socle is simple]
\label{prop:residual-socle-simple}
In the quasiprimitive configuration one has
\[
 N=T;
\]
in particular the unique socle is nonabelian simple.
\end{proposition}

\begin{proof}
Here $N=T^k$ is transitive on $\Delta$, $N_P\le G_P$ is
soluble, and $N_P^{(1)}=S$. By $(\Sigma)$,
\[
 \frac{p}{p-2}(g-1)\le3(g-1)
 <\frac{s(s+1)}{75}<s^2.
\]
Proposition~\ref{normalizer:prop:direct-product} therefore gives $k=1$.
\end{proof}

\subsection{The almost-simple case}

From now on put
\[
 T=N=\operatorname{soc}(G).
\]
Uniqueness of the minimal normal subgroup gives $C_G(T)=1$:
a nontrivial centralizer would contain $T$, contradicting $Z(T)=1$.
Thus conjugation embeds $G$ in $\Aut(T)$ and
\[
 T\triangleleft G\le\operatorname{Aut}(T),
 \qquad
 G/T\cong C_d.
\]
The group $T$ is transitive on both short $G$-orbits, and
\[
 T_P=S\rtimes C_c,
 \qquad
 T_Q=C_b\quad(Q\in\Omega),
\]
with \eqref{eq:h-t-cbd} and \eqref{eq:b-over-cb}.
Choose $C_c=C_h\cap T$ throughout. In particular the generator
of $C_h$ centralizes this particular tame inertia complement.
No maximality hypothesis on $T_P$ is imposed. The subsequent
normalizer and parabolic arguments apply to this actual inertia subgroup.

\begin{lemma}[The normalizer block]
\label{lem:normalizer-block}
Let
\[
 M=N_T(S),
 \qquad
 f=[M:T_P].
\]
Then
\[
 \Fix_\Delta(S)=M\cdot P
\]
and hence
\[
 |\Fix_\Delta(S)|=f.
\]
In particular $M\cdot P$ is a block for the transitive action of $T$ on
$\Delta$.
\end{lemma}

\begin{proof}
Certainly $M\cdot P\subseteq\Fix_\Delta(S)$. Conversely, if $R=tP\in\Delta$
is fixed by $S$, then
\[
 S\le T_R=tT_Pt^{-1}.
\]
The $p$-part of $T_R$ has order $s$ and is the conjugate $tSt^{-1}$.
Therefore $S=tSt^{-1}$, so $t\in N_T(S)=M$. This proves the equality.
The last statement is the standard correspondence between blocks containing
$P$ and subgroups between $T_P$ and $T$.
\end{proof}

\begin{proposition}[Numerical constraints on the normalizer block]
\label{prop:normalizer-block-numerics}
Define
\[
 \mu=\frac{[T:M]}{s+1}.
\]
Then
\begin{equation}\label{eq:muf}
 2g-2=\mu f u,
\end{equation}
and
\begin{equation}\label{eq:mu-hurwitz}
 (\mu-1)u\ge 2s\left(1-\frac1f\right).
\end{equation}
Moreover
\begin{equation}\label{eq:muf-6188}
 \mu f>6188,
\end{equation}
and
\begin{equation}\label{eq:f-upper-current}
 f<1+\frac{s+1}{225}.
\end{equation}
Consequently
\[
 f=1\Longrightarrow \mu>6188,
\]
while
\[
 f\ge2\Longrightarrow \mu>111.
\]
If $f=1$, then
\[
 S\in\Syl_p(T).
\]
\end{proposition}

\begin{proof}
Since $T$ is transitive on $\Delta$,
\[
 |\Delta|=[T:T_P]=[T:M]f=\mu(s+1)f.
\]
Comparison with
\[
 |\Delta|=(2g-2)\frac{s+1}{u}
\]
gives \eqref{eq:muf}.

At each point of $\Fix_\Delta(S)$ the different exponent for the $S$-action
is the same as at $P$, namely
\[
 d_P(S)=2s+u.
\]
Hurwitz for $X\to X/S$, using only these $f$ ramified points and the lower
bound $g(X/S)\ge0$, gives
\[
 2g-2\ge -2s+f(2s+u).
\]
Together with \eqref{eq:muf}, this is exactly
\eqref{eq:mu-hurwitz}. It also gives
\[
 f<\frac{g-1+s}{s}.
\]
Since \eqref{eq:sigma-size} says $g<s(s+1)/225$, we obtain
\eqref{eq:f-upper-current}.

Next $s\le g/28$ and $u<19(s+1)/2100$ give
\[
 \mu f=\frac{2g-2}{u}
 >\frac{(56s-2)2100}{19(s+1)}>6188,
\]
using $s+1>6108$. If $f\ge2$, then
\eqref{eq:mu-hurwitz} yields
\[
 (\mu-1)u\ge s,
\]
so
\[
 \mu-1>\frac{2100s}{19(s+1)}>110,
\]
and hence $\mu>111$.

Finally suppose $f=1$. Then $N_T(S)=T_P$. If $S$ were properly contained
in a $p$-subgroup $U\le T$, the normalizer condition for finite $p$-groups
would give
\[
 S<N_U(S)\le N_T(S)=T_P,
\]
contrary to the fact that $S$ is a Sylow $p$-subgroup of $T_P$. Thus
$S\in\Syl_p(T)$.
\end{proof}

\section{Screening the simple socle}
\label{scr:sec:screening}

Throughout this section we work with the minimum-genus counterexample
obtained in Section~\ref{sec:further-residual}.
We use only the already established inclusions $S\le T\le G$, the
nonabelian nature of $S$, the inequalities $s+1>6108$ and $c>b\ge111$,
and the numerical package $(\Sigma)$. In particular, $S$ is not assumed
to be a Sylow subgroup. For a positive integer $m$, write $m_p$ for
its $p$-part, and put
\[
 r=|T|_p\ge s.
\]
Here $r$ is an order, not a prime. The field size of a group of Lie
type will be denoted by $q$; initially its characteristic need not be $p$.

\begin{lemma}[An order bound]\label{scr:lem:order}
Under the present hypotheses,
\begin{equation}\label{scr:eq:order}
 |T|\le |G|<\frac4{225}s^2(s+1)^2
 \le\frac4{225}r^2(r+1)^2<\frac{r^4}{50}.
\end{equation}
\end{lemma}
\begin{proof}
The exact order formula and $h\le2u$ give
\[
 |G|=(2g-2)\frac{s(s+1)h}{u}
 \le4(g-1)s(s+1).
\]
Combining this with $|G|\ge900g^2$ yields
$g<s(s+1)/225$, and substitution gives the first strict inequality
in \eqref{scr:eq:order}. The function $x^2(x+1)^2$ is increasing.
Finally $r\ge s>6107$ implies $(1+1/r)^2<9/8$, so that
$\frac4{225}(1+1/r)^2<1/50$.
\end{proof}

\begin{proposition}[Alternating groups]
\label{scr:prop:alternating}
The simple group $T$ is not an alternating group.
\end{proposition}
\begin{proof}
Suppose $T=A_m$, $m\ge5$. Since $p$ is odd, Legendre's formula gives
\[
 r=p^{v_p(m!)}\le p^{(m-1)/(p-1)}\le3^{(m-1)/2}.
\]
The last inequality follows because $\log x/(x-1)$ decreases for
$x>1$. For $m\le17$, the largest odd prime-power divisor of $m!$
is at most $3^6=729$, contrary to $r\ge s>6107$.
For $m\ge18$ one has
\[
 \frac{m!}{2}>\frac{9^{m-1}}{50}\ge\frac{r^4}{50}.
\]
Indeed this holds at $m=18$, and the ratio $m!/9^{m-1}$ increases
thereafter, with successive multiplier $(m+1)/9>1$.
This contradicts Lemma~\ref{scr:lem:order}.
\end{proof}

\begin{proposition}[Sporadic groups]
\label{scr:prop:sporadic}
The simple group $T$ is not sporadic, regardless of $f$.
\end{proposition}
\begin{proof}
For every sporadic simple group $T$ other than the Monster, the ATLAS class
tables \cite{AtlasBook} give
\[
 \meo(T)\le70.
\]
But $T_P$ contains the cyclic complement $C_c$ with $c\ge112$, a
contradiction. For the Monster $\mathbb M$, the largest odd prime-power
divisor of its order is $R=3^{20}$, and the ATLAS order \cite{AtlasBook} gives
\[
 \frac{|\mathbb M|}{R^4}>10^{15}>\frac1{50}.
\]
Since $r\le R$, this contradicts Lemma~\ref{scr:lem:order}.
\end{proof}

\subsection{Classical groups in different characteristic}

\begin{lemma}[Prime-part estimates for classical orders]
\label{scr:lem:classical}
Assume $p\nmid q$ and $p$ is odd.
\begin{enumerate}[label=\textup{(\roman*)},leftmargin=*]
\item If $T=\PSL_m(q)$ or $\PSU_m(q)$, $m\ge3$, then
\[
 r\le(q+1)^{m-1}3^{(m-1)/2},\qquad
 |T|>\tfrac12q^{m^2-2}.
\]
\item If $T$ has type $B_n$ or $C_n$, $n\ge2$, then
\[
 r\le(q+1)^n3^{(n-1)/2},\qquad
 |T|>\tfrac13q^{2n^2+n}.
\]
\item If $T=P\Omega_{2n}^{\pm}(q)$, $n\ge4$, then
\[
 r\le(q+1)^n3^{(n-1)/2},\qquad
 |T|>\tfrac18q^{2n^2-n}.
\]
\end{enumerate}
\end{lemma}
\begin{proof}
For (i), put $\epsilon=1$ in the linear case and $\epsilon=-1$
in the unitary case, and let $e$ be the order of $\epsilon q$
modulo $p$. Set $A=(q^e-\epsilon^e)_p$ and $j=\lfloor m/e\rfloor$.
The lifting-the-exponent formula applied to the group order gives,
when $e>1$,
\[
 r=A^j(j!)_p.
\]
Here $p\nmid q-\epsilon$, and hence
\[
 A\le\frac{q^e-\epsilon^e}{q-\epsilon}
 \le(q+1)^{e-1}.
\]
Since $j(e-1)\le m-1$ and $(j!)_p\le(m!)_p$, this yields
$r\le(q+1)^{m-1}(m!)_p$. When $e=1$, the determinant condition
instead gives
\[
 r=\frac{(q-\epsilon)_p^{m-1}(m!)_p}
 {\gcd(m,q-\epsilon)_p},
\]
so the same upper bound holds. Legendre's formula gives
$(m!)_p\le3^{(m-1)/2}$.

For (ii), let $e$ be the order of $q^2$ modulo $p$ and
$j=\lfloor n/e\rfloor$. Exactly one of $q^e-1,q^e+1$ is divisible
by $p$, so $(q^{2e}-1)_p\le(q+1)^e$. The same valuation calculation
in $\prod_{i=1}^n(q^{2i}-1)$ gives
\[
 r\le(q+1)^{ej}(j!)_p\le(q+1)^n3^{(n-1)/2}.
\]
The odd part of the order in (iii) divides the corresponding odd
part for the symplectic group of dimension $2n$, which proves its
prime-part bound.

For completeness, the order lower bounds also follow directly from
the order formulae. The product inequality
$\prod(1-a_i)>1-\sum a_i$ gives
\[
 \prod_{i=2}^{m}(1-q^{-i})>1/2,\qquad
 \prod_{i=1}^{n}(1-q^{-2i})>2/3.
\]
Dividing the linear order by its center of order at most $q-1$
proves (i) for $\PSL_m$. For the unitary order retain the factor
$1+q^{-3}$, discard the other odd-index factors, and bound the
even-index product below by $1-1/(q^2-1)$. Then
\[
 \frac{q}{q+1}(1+q^{-3})\left(1-\frac1{q^2-1}\right)\ge\frac12
 \quad(q\ge2),
\]
with a strict lower bound for the actual finite product. This proves
the unitary assertion after division by a center of order at most $q+1$.
For (ii) the center has order at most $2$. For (iii) its order is
at most $4$, and the additional factor $1\mp q^{-n}$ is at least
$15/16$; thus the lower constant can be taken to be $1/8$.
\end{proof}

\begin{proposition}[Cross-characteristic linear and unitary groups]
\label{scr:prop:linear}
The simple socle is neither $\PSL_m(q)$ nor $\PSU_m(q)$ when
$p\nmid q$.
\end{proposition}
\begin{proof}
For $m\ge3$ put
\[
 F_m(q)=\frac{q^{m^2-2}}
 {2(q+1)^{4m-4}3^{2m-2}}.
\]
Lemma~\ref{scr:lem:classical} gives $|T|/r^4>F_m(q)$.
For $m\ge4$ this function increases with $q$, since
\[
 q\frac{d}{dq}\log F_m(q)
 =m^2-2-(4m-4)\frac{q}{q+1}
 >m^2-4m+2>0.
\]
Moreover,
\[
 \frac{F_{m+1}(q)}{F_m(q)}
 =\frac{q^{2m+1}}{9(q+1)^4}>1
 \quad(m\ge6,\ q\ge2).
\]
The two base values
\[
 F_6(3)=\frac{3^{24}}{2^{41}}>\frac1{50},\qquad
 F_8(2)=\frac{2^{61}}{3^{42}}>\frac1{50}
\]
exclude all $m\ge6,q\ge3$, and all $m\ge8,q\ge2$.
For $q=2,m=6,7$ the largest odd prime-parts are
\[
\begin{array}{c|rrrr}
 T&\PSL_6(2)&\PSU_6(2)&\PSL_7(2)&\PSU_7(2)\\ \hline
 R(T)&81&729&127&6561.
\end{array}
\]
The first three are too small. For the fourth,
$|T|=227787103272960>6561^4/50$.

Now suppose $3\le m\le5$. A torus with partition
$e^{\lfloor m/e\rfloor}1^{m-e\lfloor m/e\rfloor}$ contains a
Sylow $p$-subgroup whenever $\lfloor m/e\rfloor<p$: its order
already has the full $p$-part. Thus a nonabelian Sylow subgroup,
which is necessary because it contains $S$, forces $e=1$ and
$p\le m$. This is also the elementary torus argument underlying
\cite[proof of Theorem~3.8, Cases~1--2]{scr:SZ}.
Writing $A=(q-\epsilon)_p$, the exact order formula now gives
\[
\begin{array}{c|c|c}
 m&p&r\\ \hline
3&3&A^2\\
4&3&3A^3\\
5&3&3A^4\\
5&5&A^4.
\end{array}
\]
For $m=3,4$ and $q<8$, these numbers are at most $64$ and $1536$,
respectively, contrary to $r>6107$. For $m=3,q\ge8$, the exact
order gives $|T|>q^8/6$, and therefore
\[
 \frac{|T|}{r^4}>\frac16\left(\frac q{q+1}\right)^8
 \ge\frac16\left(\frac89\right)^8>\frac1{50}.
\]
For $m=4,q\ge8$, and for $m=5,q\ge3$, respectively, one obtains
\[
 \frac{|T|}{r^4}>
 \frac{q^{14}}{162(q+1)^{12}}>\frac1{50},\qquad
 \frac{|T|}{r^4}>
 \frac{q^{23}}{162(q+1)^{16}}>\frac1{50}.
\]
Both functions increase in $q$: their logarithmic derivatives, after
multiplication by $q$, are respectively
$14-12q/(q+1)>2$ and $23-16q/(q+1)>7$.
Their stated base values exceed $1/50$. For $m=5,q=2$, the largest odd prime-parts are $31$ and
$243$. Finally the odd Sylow subgroups of $\PSL_2(q)$ in different
characteristic are cyclic. This proves the proposition.
\end{proof}

\begin{proposition}[Other cross-characteristic classical groups]
\label{scr:prop:other-classical}
If $p\nmid q$, the simple socle is not of type
$B_n$, $C_n$, $D_n$, or ${}^2D_n$.
\end{proposition}
\begin{proof}
For $B_n,C_n$, Lemma~\ref{scr:lem:classical} gives
\[
 \frac{|T|}{r^4}>
 J_n(q):=\frac{q^{2n^2+n}}{3(q+1)^{4n}3^{2n-2}}.
\]
For $n\ge3$ this increases with $q$, since
\[
 q\frac{d}{dq}\log J_n(q)
 =2n^2+n-4n\frac{q}{q+1}>2n^2-3n>0.
\]
Its successive ratio in $n$ is $q^{4n+3}/(9(q+1)^4)>1$.
The base values $J_3(3)=3^{16}/2^{24}$ and
$J_4(2)=2^{36}/3^{23}$ exceed $1/50$.
The omitted group at $n=3,q=2$ has largest odd prime-part $81$.
At $n=2$, every odd cross-characteristic Sylow subgroup is abelian:
the tori of orders $(q-1)^2,(q+1)^2,q^2+1$ account for its full
odd prime-part. Central quotients of order at most two do not affect
this argument. This contradicts the nonabelian subgroup $S$.

For the orthogonal groups of even dimension put
\[
 K_n(q)=\frac{q^{2n^2-n}}{8(q+1)^{4n}3^{2n-2}}.
\]
Then $|T|/r^4>K_n(q)$. For $n\ge4$ this increases with $q$, since
\[
 q\frac{d}{dq}\log K_n(q)
 =2n^2-n-4n\frac{q}{q+1}>2n^2-5n>0.
\]
The successive ratio is $q^{4n+1}/(9(q+1)^4)>1$.
The two base values
\[
 K_4(3)=\frac{3^{22}}{2^{35}}>\frac1{50},\qquad
 K_5(2)=\frac{2^{42}}{3^{28}}>\frac1{50}
\]
leave only $P\Omega_8^{\pm}(2)$, whose largest odd prime-parts
are $243$ and $81$. These are too small.
\end{proof}

\subsection{Exceptional groups in different characteristic}

\begin{proposition}[Cross-characteristic exceptional groups]\label{scr:prop:exceptional}
The simple socle is not an exceptional group of Lie type with $p\nmid q$.
\end{proposition}
\begin{proof}
First consider $E_6,{}^2E_6,E_7,E_8,F_4$. The bound of
Glasby--Niemeyer--Popiel \cite[Theorem~1.1 and Table~1]{scr:GNP}
is
\[
 r\le |T|^{(\lfloor\log_p M_0\rfloor+1)/K_0},
\]
with $(K_0,M_0)=(12,12),(12,18),(18,18),(29,30),(12,12)$,
respectively. For every odd $p$ the exponent is at most $1/4$,
contrary to \eqref{scr:eq:order}. The listed exceptions relevant here
are $(E_6(3),13)$, $(F_4(3),13)$, and $(E_8(2),31)$.
Their prime-parts are $13^3,13^2,31^2$, all less than $6108$.

For three further types one can use the order formula directly.
In the next table the cyclotomic factors are the part of $|T|$
coprime to $q$, with $\Phi_i=\Phi_i(q)$.
\[
\begin{array}{c|c|c}
 T&\text{cyclotomic factors}&r\ \text{is at most}\\ \hline
 G_2(q)&\Phi_1^2\Phi_2^2\Phi_3\Phi_6&3(q+1)^2\\
 {}^3D_4(q)&\Phi_1^2\Phi_2^2\Phi_3^2\Phi_6^2\Phi_{12}&(q+1)^4\\
 {}^2F_4(q),\ q\ge8&\Phi_1^2\Phi_2^2\Phi_4^2\Phi_6\Phi_{12}&(q+1)^4
\end{array}
\]
To check the last column, an odd prime can divide two different
$\Phi_i(q)$ only when the indices differ by a power of that prime.
For these index sets, the only possible odd prime is $3$.
For $p\ge5$, the largest individual factor with its displayed
multiplicity gives the claimed bound. For $p=3$, lifting the
exponent gives $3(q\pm1)_3^2$, $9(q\pm1)_3^2$, and
$3(q+1)_3^2$, respectively. The second and third are at most
$(q+1)^4$ for $q\ge2$.

For $G_2(q)$ with $q\ge4$, one has $|T|>q^{14}/2$; hence
\[
 |T|/r^4>q^{14}/\bigl(162(q+1)^8\bigr)>1.
\]
At $q=3$ every odd cross-characteristic prime-part is at most $13$.
The nonsimple group at $q=2$ has simple derived group $\PSU_3(3)$,
whose largest odd prime-part is $27$.
For ${}^3D_4(q)$, one has $|T|>q^{28}/2$ for $q\ge2$, so
\[
 |T|/r^4>q^{28}/\bigl(2(q+1)^{16}\bigr)>1.
\]
For ${}^2F_4(q)$, $q\ge8$, one has $|T|>q^{26}/3$ and
$q^{26}/(3(q+1)^{16})>1$. The three displayed comparison functions
are increasing in $q$: after multiplying their logarithmic derivatives
by $q$, the relevant expressions are
$14-8q/(q+1)>6$, $28-16q/(q+1)>12$, and
$26-16q/(q+1)>10$. Hence it is enough to check the smallest stated
value in each case. The Tits group at $q=2$
has largest odd prime-part $27$.

Finally, odd Sylow subgroups of the Suzuki groups, and odd
cross-characteristic Sylow subgroups of the small Ree groups,
are cyclic; the torus descriptions are recalled in
\cite[proof of Theorem~3.8, Cases~5 and~8]{scr:SZ}.
They cannot contain the nonabelian subgroup $S$.
\end{proof}

\begin{theorem}[Defining characteristic]
\label{scr:thm:defining}
In a minimum-genus counterexample, the simple socle $T$
is a finite simple group of Lie type in characteristic $p$.
Moreover $T\not\cong\PSL_2(q)$, and $q$ is odd.
\end{theorem}
\begin{proof}
Apply the classification of finite simple groups and
Propositions~\ref{scr:prop:alternating}--\ref{scr:prop:exceptional}.
The defining-characteristic Sylow subgroups of $\PSL_2(q)$ are
elementary abelian, so they too cannot contain $S$.
All exceptional small derived groups have either been included
under their simple isomorphisms or excluded by their prime-parts.
\end{proof}

\section{Further reductions in defining characteristic}
\label{def:sec:new}

Throughout this section we assume $(\Sigma)$
and the conclusions of Section~\ref{scr:sec:screening}. In particular,
$T_P=S\rtimes C_c$, $T_Q=C_b$, $s=|S|$, $d=[G:T]$, and
\[
 c>b\ge111,\qquad S'=\Phi(S)=S_2\ne1,\qquad f\ge1,\quad\mu>111.
\]
The bound $\mu>111$ holds for $f=1$ as well, by
Proposition~\ref{prop:normalizer-block-numerics}.
The field exponent below is denoted by $a_0$, to distinguish it from
the integer $a$ already occurring in $(\Sigma)$.

\subsection{A stronger order bound}

\begin{lemma}[The two inertia orders constrain the socle]
\label{def:lem:order}
Under the present hypotheses,
\begin{equation}\label{def:eq:order}
 |T|<\frac{d\,s^2c^2(c-1)^2}{225}
      <\frac{d\,s^2c^4}{225}.
\end{equation}
\end{lemma}
\begin{proof}
The two-point Hurwitz formula, or the identities of $(\Sigma)$, gives
\[
 2g-2=\frac{|T|(c-b)}{scb}.
\]
Since $c-b\ge1$ and $b\le c-1$, it follows that
$g-1\ge |T|/(2sc(c-1))$. Now use $900g^2\le d|T|$ and $g>g-1$.
\end{proof}

\begin{lemma}[The tame action on the Frattini quotient]
\label{def:lem:free}
Every nonidentity element of $C_c$ acts without nonzero fixed vectors
on $S/\Phi(S)$. The same holds on every nonzero quotient of this
$\FF_p C_c$-module.
\end{lemma}
\begin{proof}
The first ramification coefficient embeds $S/S_2$ into the additive
group of the algebraically closed constant field. Tame conjugation
acts there by the first power, or its inverse according to conventions,
of the faithful tangent character of $C_c$. Since $S_2=\Phi(S)$,
this proves the first assertion. Maschke's theorem gives the assertion
for quotients.
\end{proof}

\subsection{Sylow wild inertia}

We record the elementary uniform estimates used below. Let $q=p^{a_0}$
be odd, let $T$ be simple of Lie type over $q$, and initially exclude
the small Ree groups. Let $\rho$ be its \emph{absolute} rank,
$R=|T|_p$, and $\delta$ the order of the rational center of the
simply connected cover. Then
\begin{equation}\label{def:eq:uniform}
 \frac{q^\rho}{2\delta}<\frac{|T|}{R^2}
 <\frac{2q^\rho}{\delta},\qquad
 \delta\le\rho+1,\qquad |\operatorname{Out}(T)|\le6\delta a_0.
\end{equation}
The central factors and order formulas are those in
\cite[Tables~2--3]{scr:GNP}; the automorphism decomposition is
\cite[Definition~3.3 and Theorem~3.4]{def:BMO}.
Here is also a direct check of the constants. In the order formula
$|T|=R^2q^\rho\Pi/\delta$, the normalized factors of $\Pi$ are
$1\pm q^{-j}$ with $j\ge2$, with at most one repeated degree.
Consequently the product lies between $1/2$ and $2$:
\[
 (1-3^{-2})\Bigl(1-\sum_{j\ge2}3^{-j}\Bigr)>\tfrac12,
 \qquad (1+3^{-2})\exp\Bigl(\sum_{j\ge2}3^{-j}\Bigr)<2.
\]
For ${}^3D_4(q)$ the factor is instead
$(1+q^{-4}+q^{-8})(1-q^{-6})(1-q^{-2})$, with the same bounds.
The diagonal, field and graph parts of the outer group give the
last inequality in \eqref{def:eq:uniform}. Twisted field orders
$2a_0$ or $3a_0$ and the exceptional graph automorphism of $G_2$
in characteristic three are included in the factor six.

\begin{lemma}[Borel torus data]\label{scr:lem:borel}
In a simply connected group of odd defining characteristic, the
rational Borel torus has the following structure. The torus in
the simple quotient is obtained by dividing by the rational
center of order $\delta$.
\[
\begin{array}{c|c}
\text{type}&\text{simply connected Borel torus}\\ \hline
\text{split, rank }\rho&(C_{q-1})^\rho\\
{}^2A_{2m}&(C_{q^2-1})^m\\
{}^2A_{2m-1}&(C_{q^2-1})^{m-1}\times C_{q-1}\\
{}^2D_n&(C_{q-1})^{n-2}\times C_{q^2-1}\\
{}^3D_4&C_{q-1}\times C_{q^3-1}\\
{}^2E_6&(C_{q-1})^2\times(C_{q^2-1})^2.
\end{array}
\]
\end{lemma}
\begin{proof}
The simple coroots form a basis of the cocharacter lattice
in the simply connected form. Frobenius is $q$ times the
diagram permutation. An orbit of length $j$ gives a cyclic
cokernel of order $q^j-1$. This proves the table; passage to
the simple group divides out the rational center.
See the torus construction in \cite{DM}.
\end{proof}

If $U\in\Syl_p(T)$, write $\mathcal B=N_T(U)=U\rtimes H$.
The torus table in Lemma~\ref{scr:lem:borel} gives
\begin{equation}\label{def:eq:torus}
 \exp H\le q^3-1,\qquad
 \delta|H|\ge q^\rho(1-q^{-1})^\rho.
\end{equation}
For the first inequality, all Frobenius orbits on the simple
coroots have length one, or lengths one and two, or lengths one
and three; the corresponding cyclic orders divide $q^j-1$ for
one fixed $j\le3$. The second follows by multiplying
$q^j-1\ge q^j(1-q^{-1})^j$ over these orbits.

\begin{lemma}[A Borel containment bounds the rank]
\label{def:lem:rank}
Suppose $T$ is not a small Ree group. If $T_P$ is contained in
a Borel subgroup of $T$, then $q\ge5$ and $\rho\le15$.
\end{lemma}
\begin{proof}
The image of $C_c$ in the Borel quotient has order $c$, whence
$112\le c\le q^3-1$. Thus $q\ne3$. Using $s\le R$,
\eqref{def:eq:order}, and \eqref{def:eq:uniform}, we obtain
\[
 q^{\rho-12}<\frac{4}{75}\delta^2a_0
 \le\frac{4}{75}(\rho+1)^2a_0.
\]
Since $a_0\le q$, this implies
$q^{\rho-13}<4(\rho+1)^2/75$. For $\rho=16$ this fails already
at $q=5$, since $125>4\cdot17^2/75$. Increasing $\rho$ multiplies
the left side by at least five and the right side by less than two.
It therefore fails for every $\rho\ge16$.
\end{proof}

\begin{theorem}[The wild inertia is not Sylow]
\label{def:thm:non-sylow}
$S$ is a proper subgroup of a
Sylow $p$-subgroup of $T$.
\end{theorem}
\begin{proof}
Suppose $S\in\Syl_p(T)$. If $T={}^2G_2(q)$, then
$s=q^3$ and $[T:N_T(S)]=q^3+1$, so $\mu=1$, a contradiction.
Otherwise $T_P\le N_T(S)=\mathcal B$, so
Lemma~\ref{def:lem:rank} applies. Equations
\eqref{def:eq:uniform}--\eqref{def:eq:torus} give
\[
 \mu=\frac{|T|}{s|H|(s+1)}
 <2(1-q^{-1})^{-\rho}
 \le2(5/4)^{15}<57<111.
\]
This contradicts the established lower bound for $\mu$.
\end{proof}

\begin{corollary}[Divisibility of the normalizer block]
\label{def:cor:block}
One has $p\mid f$, $f\ge p$, and $\mu>148$.
\end{corollary}
\begin{proof}
Choose a Sylow group $U>S$. The normalizer condition for finite
$p$-groups gives $N_U(S)>S$. Thus $p$ divides $[N_T(S):S]$,
and hence divides $f=[N_T(S):S]/c$. The normalizer-block inequality
and $(\Sigma)$ now give
\[
 \mu>1+\frac{4200}{19}\frac{s}{s+1}\left(1-\frac1p\right)
 >1+\frac{2800}{19}\frac{6107}{6108}>148.
\]
\end{proof}

\subsection{The unitary groups of dimension three}

\begin{theorem}\label{def:thm:unitary}
The simple socle $T$ is not isomorphic to $\PSU_3(q)$.
\end{theorem}
\begin{proof}
Let $q=p^{a_0}$ and $\delta=\gcd(3,q+1)$. A Sylow group $U$ has
order $q^3$, with $Z=U'=\Phi(U)$ of order $q$ and
$U/Z\cong(\FF_{q^2},+)$. Distinct Sylow groups intersect trivially:
in the natural action on isotropic points, $U$ acts regularly on
the other $q^3$ points. Hence a nontrivial $S\le U$ lies in a
unique Sylow group and $N_T(S)\le N_T(U)$.
These assertions can also be read directly from the standard
unitary matrices; the coordinates and torus action are recorded
in \cite[proof of Proposition~4.8]{def:BMO}.

By Theorem~\ref{def:thm:non-sylow}, $S<U$. Conjugate the tame
complement into the torus of $N_T(U)$; this is possible by
Schur--Zassenhaus in the preimage of its image modulo $U$.
This torus has order $(q^2-1)/\delta$ and acts faithfully by
scalar multiplication on $U/Z$. Thus a generator of $C_c$ acts
as multiplication by some $\lambda\in\FF_{q^2}^{\times}$ of
order exactly $c$.

Put $e=\operatorname{ord}_c(p)$ and $Q=p^e$. Then $e\mid2a_0$,
$c\mid Q-1$, and the image $A=SZ/Z$ is a vector space over
$\FF_Q$. It is proper in $\FF_{q^2}$, since $SZ=U$ would imply
$S=U$ by the Frattini property. It is nonzero, since otherwise
$S\le Z$ would be abelian. Therefore
\[
 |A|\le q^2/Q,\qquad s\le q^3/Q.
\]
In fact $e<a_0$. If $e>a_0$, divisibility forces $e=2a_0$,
inconsistent with $A$ nonzero and proper. If $e=a_0$, the proper
nonzero space $A$ is a line over $\FF_q$. The commutator of two
unitary unipotent elements with first coordinates $x,y$ is,
up to sign, $xy^q-yx^q$. It vanishes on an $\FF_q$-line, forcing
$S$ to be abelian. Thus $e<a_0$, and divisibility by $2a_0$
gives $e\le2a_0/3$.

It follows from \eqref{def:eq:order} that
\[
 |T|<\frac{d q^6(Q-1)^4}{225Q^2}
 <\frac{d q^{22/3}}{225}.
\]
On the other hand,
\[
 |T|=\frac{q^3(q^3+1)(q^2-1)}{\delta}
 >\frac{8q^8}{9\delta},\qquad d\le2\delta a_0.
\]
Combining these inequalities gives
$100q^{2/3}<\delta^2a_0\le9a_0$, which is impossible because
$q^{2/3}\ge3^{2a_0/3}>a_0$ for every positive integer $a_0$.
\end{proof}

\subsection{The small Ree groups}

\begin{theorem}\label{def:thm:ree}
The simple socle $T$ is not isomorphic to ${}^2G_2(q)$.
\end{theorem}
\begin{proof}
Write $q=3^{a_0}$, where $a_0\ge3$ is odd. Here
\[
 |T|=q^3(q^3+1)(q-1),\quad d\le a_0,\quad |U|=q^3.
\]
The Sylow groups have trivial pairwise intersections,
$\Phi(U)=U'$ is abelian of order $q^2$, and $U/U'\cong(\FF_q,+)$.
The normalizer quotient is $\FF_q^\times$ and acts faithfully
by scalars on $U/U'$. For the Sylow structure see
\cite[Proposition~3.1]{Ba14}; the scalar coordinates are also given
in \cite[proof of Proposition~4.8]{def:BMO}.

As in the preceding proof, $S<U$, $N_T(S)\le N_T(U)$, and we
may put $C_c$ in its torus. For $e=\operatorname{ord}_c(3)$,
$Q=3^e$, the image $A=SU'/U'$ is a nonzero proper $\FF_Q$-subspace
of $\FF_q$. Nonzero follows from nonabelianness of $S$, and
proper follows from the Frattini property. Thus $e$ is a proper
divisor of the odd integer $a_0$, so
\[
 e\le a_0/3,\qquad c\le Q-1,\qquad s\le q^3/Q.
\]
Consequently \eqref{def:eq:order} gives
\[
 |T|<\frac{a_0q^{20/3}}{225},
\]
whereas $|T|>2q^7/3$. This would imply $150q^{1/3}<a_0$.
But $3^{a_0/3}\ge a_0$ for odd $a_0\ge3$: it holds at three,
and increasing $a_0$ by two multiplies the left side by
$3^{2/3}>2$ and the right side by at most $5/3$.
\end{proof}

\subsection{The linear groups of dimension three}

\begin{theorem}\label{def:thm:linear-three}
The simple socle $T$ is not isomorphic to $\PSL_3(q)$.
\end{theorem}
\begin{proof}
Let $q=p^{a_0}$, $\delta=\gcd(3,q-1)$, and let $U$ be the image
of the upper unitriangular group. Its first superdiagonal gives
\[
 U/\Phi(U)=U/Z(U)\cong\FF_q\oplus\FF_q,\qquad |U|=q^3.
\]
The image of the nonabelian group $S$ has nonzero projections on
both coordinates: otherwise the commutator $xy'-x'y$ vanishes.
It contains a pair with both coordinates nonzero, since an
additive subgroup cannot be the union of two proper subgroups.
The corresponding element is regular unipotent, with a single
Jordan block of size three. Its unique invariant complete flag
shows that $U$ is the unique Sylow group containing $S$.
In particular, every automorphism in $G_P$ normalizes $U$.

Put $C_c$ in the diagonal torus of $N_T(U)$. On the two root
coordinates a generator acts with eigenvalues $\lambda_1,\lambda_2$.
Both have order $c$: each nonzero coordinate projection of $S$
is a quotient of $S/\Phi(S)$, so Lemma~\ref{def:lem:free} applies.
Thus $c\mid q-1$. Put $e=\operatorname{ord}_c(p)$ and $Q=p^e$,
so $e\mid a_0$. Every irreducible constituent of these two
coordinate spaces has dimension $e$ over $\FF_p$.
The proper submodule $A=S\Phi(U)/\Phi(U)$ therefore has
codimension at least $e$, giving $s\le q^3/Q$.

If $e<a_0$, then $Q^2\le q$. By \eqref{def:eq:order},
\[
 |T|<\frac{d q^6Q^2}{225}
 \le\frac{2\delta a_0q^7}{225}.
\]
But $|T|=q^3(q^3-1)(q^2-1)/\delta>q^8/(2\delta)$, so
$q<4\delta^2a_0/225\le36a_0/225$, contradicting $q>a_0$.

It remains to consider $e=a_0$, when $s\le q^2$.
Let $z$ generate the tame complement $C_h$ in $G_P$. This
complement contains $C_c$, centralizes it, and maps onto $G/T$.
Use the diagonal--field--graph decomposition of the automorphisms
normalizing $U$, as in \cite[Theorem~4.2]{def:BMO}.
On $U/\Phi(U)$ an automorphism has a diagonal factor, a field
factor $x\mapsto x^{p^j}$, and possibly the interchange of the
two root coordinates. Its unipotent factor acts trivially there.
Since $z$ centralizes $C_c$, the eigenvalues $\lambda_i$ are
preserved by this operation. Squaring removes the interchange
and shows that $\lambda_i^{p^{2j}}=\lambda_i$. Hence
$a_0=e\mid2j$. The image of $\langle z\rangle$ in the field
group has order at most two; its graph image has order at most
two, and its remaining diagonal outer image has order at most
$\delta$. Therefore $d\le4\delta\le12$.

Now \eqref{def:eq:order} gives $|T|<12q^8/225$, whereas
$|T|>q^8/(2\delta)\ge q^8/6$. This is the final contradiction.
\end{proof}

\subsection{Borel containment in split groups of higher rank}

\begin{lemma}[A codimension forced by tame inertia]
\label{def:lem:codimension}
Let $T$ be a split simple group of Lie type of absolute rank
$\rho\ge3$ over $\FF_q$, $q=p^{a_0}$ odd. Suppose
$T_P\le\mathcal B=U\rtimes H$, a Borel subgroup, and $S<U$.
For $e=\operatorname{ord}_c(p)$ and $Q=p^e$, one has
\[
 e\mid a_0,\qquad c<Q\le q,\qquad [U:S]\ge Q.
\]
\end{lemma}
\begin{proof}
For these split groups in odd characteristic,
$U/\Phi(U)$ is the sum of the simple-root additive groups,
each isomorphic to $\FF_q$. This follows from the Chevalley
commutator relations: the nonsimple positive-root groups generate
$U'$, and $U/U'$ is elementary abelian. In the nonsimply laced
types of rank at least three the needed coefficients are one or
two, hence nonzero here. The rank-two type $G_2$, which requires
separate care at three, is not part of this assertion.

Conjugate $C_c$ into $H$. Since $H$ has exponent dividing $q-1$,
the first two assertions hold. Write $V=U/\Phi(U)$ and
$A=S\Phi(U)/\Phi(U)<V$. By Lemma~\ref{def:lem:free}, all
irreducible constituents of $A$ are faithful $C_c$-modules.
On a simple-root coordinate on which $C_c$ is not faithful,
the projection of $A$ must vanish; such a missing coordinate
has dimension $a_0\ge e$. If all coordinate actions are
faithful, their irreducible constituents have dimension $e$,
so the proper submodule $A$ has codimension at least $e$.
These observations also cover a mixture of the two kinds of
coordinates. Thus $[V:A]\ge p^e$, and $[U:S]\ge[V:A]$.
\end{proof}

\section{Intermediate quotients and parabolic inertia}
\label{par:sec:new}

We continue to assume $(\Sigma)$ and use the notation
$I=T_P=S\rtimes C_c$, $T_Q=C_b$, $s=|S|$, $d=[G:T]$.
In particular,
\[
 c>b\ge111,\qquad \frac cb<\frac{2119}{2100},\qquad
 |T|<\frac{d s^2c^4}{225},\qquad S'=\Phi(S)\ne1.
\]
Put
\[
 \kappa=\frac1{225}\left(\frac{2119}{2100}\right)^2
 =\frac{4490161}{992250000}<\frac1{220}.
\]
The argument below uses intermediate curves, not merely the Hurwitz
formula for the Galois cover $X\longrightarrow X/T$.

\subsection{Intermediate quotients}

\begin{lemma}[Coset orbits and intermediate genus]\label{par:lem:genus}
Let $I\le H<T$, set $M=[T:H]$, and let $\omega(J)$ be the number
of orbits of a subgroup $J\le T$ on $T/H$. If $\rho=\omega(I)$,
then
\begin{equation}\label{par:eq:genus}
 2g(X/H)-2\le-\rho+\frac{M-s-1}{s b}.
\end{equation}
Consequently
\[
 M\ge s+1+(\rho-2)s b.
\]
In particular, if the action of $T$ on $T/H$ has rank at least
three, then $M\ge s(b+1)+1>s b$.
\end{lemma}
\begin{proof}
Apply the conductor--discriminant formula to the permutation
representation on $T/H$, and then Riemann--Hurwitz for
$X/H\longrightarrow X/T=\PP^1$; see \cite[Chapter~VI]{SerreLF}.
Writing $S_1=S$ and $S_i$ for the lower ramification groups at
the wild branch point, the result is
\[
 2g(X/H)-2=-\omega(I)-\omega(C_b)
       +\frac1{sc}\sum_{i\ge1}|S_i|\bigl(M-\omega(S_i)\bigr).
\]
Indeed, the terms at level zero are $M-\omega(I)$ and
$M-\omega(C_b)$, and cancel the $-2M$ in Riemann--Hurwitz.
Every $S_i$ fixes the coset $H$. Every other $S_i$-orbit has
size at most $|S_i|$, so
\[
 |S_i|\bigl(M-\omega(S_i)\bigr)
 \le(M-1)(|S_i|-1).
\]
Also $\omega(C_b)\ge M/b$. The identities in $(\Sigma)$ give
\[
 \sum_{i\ge1}(|S_i|-1)=s+u+1,\qquad
 \frac{s+u+1}{c}=\frac{s+1}{b}.
\]
Substitution gives \eqref{par:eq:genus}. Finally, $g(X/H)\ge0$
and $\omega(I)\ge\omega(H)$ because $I\le H$.
\end{proof}

\begin{corollary}[A parabolic numerical test]\label{par:cor:test}
Suppose $I\le P=V\rtimes L<T$, put $v=|V|$,
$M=[T:P]$ and $z=M/v$, and suppose the coset action has rank
at least three. Then
\begin{equation}\label{par:eq:test}
 \frac{|L|}{d c^2 z}<\kappa.
\end{equation}
\end{corollary}
\begin{proof}
Lemma~\ref{par:lem:genus} gives $s<M/b$. Combining this with
Lemma~\ref{def:lem:order} yields
$|T|/M^2<d c^4/(225b^2)$.
But $|T|=v^2|L|z$ and $M^2=v^2z^2$.
\end{proof}

\begin{lemma}[A canonical parabolic for a $p$-subgroup]\label{par:lem:canonical}
Let $T$ be a finite simple group of Lie type in defining characteristic $p$
and let $1\ne S\le T$ be a $p$-subgroup. There is a proper parabolic
subgroup $P(S)$ such that
\[
 S\le O_p(P(S)),\qquad N_T(S)\le P(S),
\]
and every automorphism of $T$ normalizing $S$ also normalizes $P(S)$.
\end{lemma}
\begin{proof}
Define inductively
\[
 U_0=S,\qquad U_{i+1}=O_p\bigl(N_T(U_i)\bigr).
\]
Since $U_i\triangleleft N_T(U_i)$, one has $U_i\le U_{i+1}$. Moreover
$U_{i+1}$ is characteristic in $N_T(U_i)$, so
$N_T(U_i)\le N_T(U_{i+1})$. The group $T$ is finite, hence the chain
stabilizes at a nontrivial $p$-subgroup $R$ satisfying
\[
 R=O_p(N_T(R)).
\]
Thus $R$ is $p$-radical. The finite Borel--Tits theorem in defining
characteristic identifies $N_T(R)$ with a proper parabolic subgroup and
$R$ with its unipotent radical; see \cite[Chapter~3, \S3.1]{GLS3}. Put
\[
 P(S)=N_T(R).
\]
Then $S\le R=O_p(P(S))$ and $N_T(S)\le P(S)$.

Finally, if $\alpha\in\Aut(T)$ normalizes $S$, induction gives
$\alpha(U_i)=U_i$ for every $i$. Hence $\alpha(R)=R$ and
$\alpha(P(S))=P(S)$. This includes field, diagonal, diagram, triality, and
the exceptional graph automorphism in type $G_2$ in characteristic three.
\end{proof}

The parabolic coset actions used below have rank at least three. By the
Bruhat decomposition, the rank of the action on $T/P_J$ is the number of
double cosets $W_J\backslash W/W_J$ in the corresponding Weyl group. In the
families used below these double cosets give at least three orbits.
Concretely, for the symplectic, unitary and orthogonal groups of relative
rank at least two, intersection dimensions (or orthogonality in the line
case) already distinguish three orbits for maximal parabolics; passing to a
finer flag cannot decrease the number of double cosets. For linear groups
we use flags with at least three blocks, whose identity double coset and the
two adjacent block interchanges give three double cosets. The exceptional
cases are checked in the same Weyl-group form of the Bruhat decomposition.

\subsection{Linear groups}

Write
\[
 \pi_j(q)=\prod_{i=1}^j(1-q^{-i}),\quad \pi_0(q)=1,\quad
 A=\frac{(q-1)^2}{q},\quad
 F_j=q^{j(j-2)}\pi_{j-1}(q)^2.
\]
We repeatedly use
\begin{equation}\label{par:eq:increment}
 \frac{F_{j+1}}{F_j}=q^{2j-1}(1-q^{-j})^2\ge A.
\end{equation}
Adding a block of size one to an expression containing
$(q-1)^{2r}\prod F_{n_i}$ also multiplies it by $A$.
For $q=p^{a_0}$ odd,
\begin{equation}\label{par:eq:field}
 A\ge\frac43 a_0,\qquad q\ge3a_0.
\end{equation}
Both follow from $q\ge3^{a_0}$; at $a_0=1$ the least case is
$q=3$, and the successive exponential growth proves the rest.

\begin{theorem}\label{par:thm:linear}
The simple socle $T$ is not isomorphic to $\PSL_n(q)$.
\end{theorem}
\begin{proof}
Dimensions two and three were excluded earlier. Let $n\ge4$,
$q=p^{a_0}$ and $\delta=\gcd(n,q-1)$, so $d\le2\delta a_0$.
Lift $I$ to $\operatorname{SL}_n(q)$ and take a composition flag
of its natural module. The normal $p$-subgroup lifting $S$ acts
trivially on every simple factor. Consequently $I$ is contained
in the corresponding parabolic and $S$ lies in its radical.
Let its block dimensions be $n_1,\ldots,n_r$. If $r=1$ then
$S=1$, and if $r=2$ its radical is abelian. Hence $r\ge3$.

The Levi order and normalized index are
\[
 |L|=\frac{q^{\sum n_i^2}\prod_i\pi_{n_i}(q)}{(q-1)\delta},
 \qquad z=\frac{\pi_n(q)}{\prod_i\pi_{n_i}(q)}.
\]
A semisimple lift of a generator of $C_c$ has order at least
$c$. Decompose its eigenvalues into irreducible blocks. Each
torus factor $q^j-1$ is divisible by $q-1$, and combining the
factors within each block gives
\[
 c\le\frac{\prod_i(q^{n_i}-1)}{(q-1)^{r-1}}.
\]
It follows that
\[
 \frac{|L|}{z c^2}\ge
 \frac{(q-1)^{2r-3}\prod_i F_{n_i}}{\delta\pi_n(q)}.
\]
The numerator at $(n_1,n_2,n_3)=(1,1,1)$ is
$(1-q^{-1})^3$. Increasing a block, or adding a block of size
one, multiplies it by at least $A$. Therefore
\[
 \frac{|L|}{d z c^2}>
 \frac{(1-q^{-1})^3 A^{n-3}}{2n^2a_0}
 \ge\frac4{27n^2}\left(\frac43\right)^{n-3}.
\]
For integers $n\ge4$ the last expression is minimized at $n=7$:
the ratio of consecutive terms is $\frac43(n/(n+1))^2$.
Its minimum is $1024/107163>1/105>\kappa$, contradicting
Corollary~\ref{par:cor:test}.
\end{proof}

\subsection{Symplectic and odd orthogonal groups}

Put $\epsilon_j(q)=\prod_{i=1}^j(1-q^{-2i})$, with
$\epsilon_0(q)=1$. We use
\[
 \epsilon_j(q)>\frac78\quad(j\ge1),
 \qquad \pi_j(q)>\frac12\quad(j\ge1).
\]
The bounds follow by comparing a product with one minus the
sum of the omitted geometric series (strictness holds for the
finite products).

\begin{theorem}\label{par:thm:bc}
The simple socle $T$ is not isomorphic to $\operatorname{PSp}_{2n}(q)$
with $n\ge3$, or $\operatorname{P\Omega}_{2n+1}(q)$ with $n\ge3$.
\end{theorem}
\begin{proof}
Choose the canonical parabolic described above. Write its Levi
parameters as $n_1,\ldots,n_r,m$, where
$k=\sum_i n_i\ge1$ and $n=k+m$. Up to the central or spinor-norm
factor two its Levi is
\[
 \prod_i\operatorname{GL}_{n_i}(q)\times
 \operatorname{Sp}_{2m}(q),\quad\hbox{or}\quad
 \prod_i\operatorname{GL}_{n_i}(q)\times
 \operatorname{SO}_{2m+1}(q).
\]
Thus, in both cases,
\[
 |L|=\tfrac12 q^{\sum n_i^2+2m^2+m}
                 \Bigl(\prod_i\pi_{n_i}(q)\Bigr)\epsilon_m(q),
 \qquad z=\frac{\epsilon_n(q)}{
                  (\prod_i\pi_{n_i}(q))\epsilon_m(q)},
 \qquad d\le2a_0.
\]

First suppose $m\ge1$. A signed torus partition in the
symplectic or orthogonal factor has exponent at most
$q^{m+1}/(q-1)$; see \cite[Lemma~2.9]{par:GMPS}.
All its torus factors, and all the linear torus factors, are
even. Taking their least common multiple gives
\[
 c\le\frac{q^{m+1}\prod_i(q^{n_i}-1)}{2(q-1)^r}.
\]
Equations \eqref{par:eq:increment}--\eqref{par:eq:field} now give
\[
 \frac{|L|}{d z c^2}\ge
 \frac{49}{64a_0}q^{2m^2-m-2}A^k.
\]
For $m=1$, we have $k\ge2$ and
$A^2/(a_0q)\ge16/27$, so this is at least $49/108$.
For $m\ge2$ the exponent is at least four and $k\ge1$,
which gives a still larger lower bound. Both contradict
\eqref{par:eq:test}.

If $m=0$ and $r\ge2$, use
$c\le\prod_i(q^{n_i}-1)/(q-1)^{r-1}$ instead. Starting with
the two blocks $(1,1)$ gives
\[
 \frac{|L|}{d z c^2}\ge
 \frac{(1-q^{-1})^2 A^{n-2}}{4a_0}\ge\frac4{27}.
\]
If $m=0,r=1$, the symplectic radical is abelian, contrary to
nonabelianness of $S$. In the odd orthogonal case this parameter
is allowed, but then $c\le q^n-1$ and
\[
 \frac{|L|}{d z c^2}\ge
 \frac{q^{n(n-2)}\pi_{n-1}(q)^2}{4a_0}
 >\frac{q^{n(n-2)}}{16a_0}\ge\frac{27}{16}.
\]
This excludes the last case.
\end{proof}

\begin{theorem}\label{par:thm:sp4}
The simple socle $T$ is not isomorphic to $\operatorname{PSp}_4(q)$.
\end{theorem}
\begin{proof}
The canonical parabolic is a Borel or one of the two maximal
parabolics. The Siegel radical is abelian. For the line parabolic,
\[
 |L|=\tfrac12 q(q-1)^2(q+1),\qquad
 z=\frac{(q+1)(q^2+1)}{q^3},\qquad c\le\frac{q^2-1}{2}.
\]
The element generating $C_h$, of order $h=dc$, belongs to
$\Aut(T)$. The bound in \cite[Theorem~2.16]{par:GMPS} gives
$h\le q^3/(q-1)$. Hence $dc^2\le q^3(q+1)/2$ and
\[
 \frac{|L|}{d z c^2}\ge
 \frac{q(q-1)^2}{(q^2+1)(q+1)}\ge\frac3{10}>\kappa.
\]

For a Borel, $c\mid q-1$. The proof of
Lemma~\ref{def:lem:codimension} also applies to type $C_2$ in
odd characteristic: the simple-root quotient is $\FF_q^2$ and
the relevant commutator coefficient is two. With
$e=\operatorname{ord}_c(p)$, $Q=p^e$, it gives $s\le q^4/Q$.
If $e<a_0$, then $Q^2\le q$, and the strong order inequality
and $|T|/q^8>q^2/4$ give $q<8a_0/225$, impossible.

If $e=a_0$, let $x$ generate $C_c$ and let $z_0$ generate $C_h$.
Write its outer field component as $p^j$. There is no graph
automorphism in odd characteristic. Its inner--diagonal part
is inner in the adjoint algebraic group, so centralization of
$x$ implies that $x^{p^j}$ and $x$ are conjugate there.
Semisimple conjugacy in a torus is controlled by the Weyl group
of type $C_2$, each of whose elements has order dividing four.
It follows that $x^{p^{4j}}=x$, so $a_0\mid4j$.
The outer field image of $\langle z_0\rangle$ has order at most
four, and its diagonal kernel has order at most two: $d\le8$.
Now $|T|/q^8>q^2/4$, whereas
\eqref{def:eq:order} gives $|T|/q^8<8q^2/225$, a contradiction.
\end{proof}

\subsection{Even orthogonal groups}

We first record a useful consequence of the faithful local action.
\begin{lemma}[Removing a centralized Levi projection]\label{par:lem:projection}
Suppose $I\le V\rtimes L$, $S\le V\rtimes U_L$ for a
$p$-subgroup $U_L\le L$, and a nonidentity element of $C_c$
centralizes $U_L$. Then the image of $S$ in $U_L$ is trivial.
\end{lemma}
\begin{proof}
If the image $K$ were nontrivial, then $K/\Phi(K)\ne0$ would
be a quotient of $S/\Phi(S)$. The indicated nonidentity element
would act trivially on it, contrary to Lemma~\ref{def:lem:free}.
\end{proof}

\begin{theorem}\label{par:thm:d}
The simple socle $T$ is not isomorphic to
$\operatorname{P\Omega}_{2n}^{\varepsilon}(q)$ for $n\ge4$.
\end{theorem}
\begin{proof}
Let $\delta=\gcd(4,q^n-\varepsilon)\le4$, so
$d\le6\delta a_0$. We keep weak inequalities at the steps where this
crude outer-automorphism bound can be attained; the final numerical
constants are still strictly larger than $\kappa$.
A parabolic with $S$ in its radical has Levi parameters $n_1,\ldots,n_r,m$, with $k=\sum_i n_i$,
$n=k+m$. Put
\[
 \zeta_m^{\varepsilon}=(1-\varepsilon q^{-m})\epsilon_{m-1}(q)
 \quad(m\ge1),\qquad \zeta_0^+=1.
\]
The Levi and normalized index formulas are
\[
 |L|=\frac{q^{\sum n_i^2+2m^2-m}}{\delta}
               \Bigl(\prod_i\pi_{n_i}(q)\Bigr)\zeta_m^{\varepsilon},
 \qquad z=\frac{\zeta_n^{\varepsilon}}
                    {(\prod_i\pi_{n_i}(q))\zeta_m^{\varepsilon}}.
\]
Here $m=0$ occurs only in the plus type; $m=0,r=1$ has
abelian radical and is excluded. For $m\ge2$,
$\zeta_m^{\varepsilon}>3/4$, and always
$\zeta_n^{\varepsilon}<9/8$.

For $m\ge2$ the same signed-partition estimate as in the previous
proof gives
\[
 \frac{|L|}{d z c^2}\ge
 \frac{q^{2m^2-3m-2} A^k}{3\delta^2a_0}
 \ge\frac1{36}>\kappa.
\]
For $m=0,r\ge2$, the linear torus estimate gives
\[
 \frac{|L|}{d z c^2}\ge
 \frac{(1-q^{-1})^2 A^{n-2}}{6\delta^2a_0\zeta_n^+}
 \ge\frac{16}{2187}>\kappa.
\]
The last comparison uses $n\ge4$, $A/a_0\ge4/3$ and $A\ge4/3$.

If $m=1,\varepsilon=+$, the remaining torus factor has order $q-1$
and contributes no new factor to the linear torus exponent.
Consequently, since $k\ge3$,
\[
 \frac{|L|}{d z c^2}\ge
 \frac{A^k}{6\delta^2a_0q\zeta_n^+}
 \ge\frac{16}{2187}>\kappa.
\]
For $m=1,\varepsilon=-$ use instead
\[
 c\le\frac{(q+1)\prod_i(q^{n_i}-1)}{2(q-1)^{r-1}}.
\]
When $k\ge4$ this gives
\[
 \frac{|L|}{d z c^2}\ge
 \frac{4A^{k-1}}{6\delta^2a_0q^2\zeta_n^-}
 \ge\frac{64}{6561}>\kappa.
\]
These estimates follow directly from
$A/q\ge4/9$ and \eqref{par:eq:field}.

It remains to consider $n=4,m=1,\varepsilon=-$, with the
composition of $k=3$. For the composition $(3)$, the same
formula before replacing its block product by $A^{k-1}$ gives
\[
 \frac{|L|}{d z c^2}\ge
 \frac{4q^2\pi_2(q)^2}{6\delta^2a_0\zeta_4^-}
 \ge\frac1{12}.
\]
For $(2,1)$ or $(1,2)$ the sharper bound $c\le q^2-1$ gives
\[
 \frac{|L|}{d z c^2}\ge
 \frac{q^2(1-q^{-1})^4(1+q^{-1})^2}
      {6\delta^2a_0\zeta_4^-}
 \ge\frac4{243}>\kappa.
\]

The composition $(1,1,1)$ is the Borel of ${}^2D_4(q)$.
Its simple-root quotient has two coordinates over $\FF_q$
and one over $\FF_{q^2}$. If $c\mid q-1$, the codimension
proof of Lemma~\ref{def:lem:codimension} gives
$s\le |T|_p/p^e$, where $e=\operatorname{ord}_c(p)\mid a_0$.
Here $\delta=2$ and $d\le4a_0$. The strong order inequality
would give $q^2<16a_0/225$, impossible.

If $c\nmid q-1$, a generator $x$ of $C_c$ has
$x^{q-1}\ne1$. The two fixed simple nodes in the twisted
$D_4$ diagram form an $A_2$ subsystem. If their root groups are
$X_\alpha$ and $X_\beta$, then for $u,v\in\FF_q$ the Chevalley
relation is
\[
 [x_\alpha(u),x_\beta(v)]
   =x_{\alpha+\beta}(\pm uv).
\]
Thus the positive unipotent subgroup of this $A_2$ is generated by
$X_\alpha(\FF_q)$ and $X_\beta(\FF_q)$ together with their displayed
commutators. The root characters of $x^{q-1}$ are trivial on both
simple root groups, hence also on $X_{\alpha+\beta}$; consequently
$x^{q-1}$ centralizes the whole positive unipotent subgroup.
Lemma~\ref{par:lem:projection} puts
$S$ in the radical of the larger parabolic with Levi parameters
$(3),m=1,\varepsilon=-$. This is the case already excluded.
\end{proof}

\subsection{Unitary groups in all dimensions}

For this subsection put
\[
 Q_0=q^2,\quad B=\frac{(Q_0-1)^2}{Q_0},\quad
 \eta_j=\prod_{i=1}^j(1-(-q)^{-i}),\quad \eta_0=1.
\]
Pairing consecutive factors, and then bounding the positive
factors by an exponential, gives $1\le\eta_j<3/2$.
For a composition $k=n_1+\cdots+n_r$ define
\[
 \mathcal F=(Q_0-1)^{2r-2}
       \prod_i Q_0^{n_i(n_i-2)}\pi_{n_i-1}(Q_0)^2.
\]
The increment argument gives
\begin{equation}\label{par:eq:unitary-f}
 \mathcal F\ge q^{-2}B^{k-1},\qquad
 \mathcal F\ge(1-q^{-2})^2B^{k-2}\quad(r\ge2).
\end{equation}

\begin{lemma}[The first layer in dimension five]\label{par:lem:su5-layer}
For odd $q=p^{a_0}$ let $U$ be the standard upper-unitriangular
Sylow subgroup of $\operatorname{SU}_5(q)$, for the Hermitian form
with antidiagonal Gram matrix. Then
\[
 |U|=q^{10},\qquad U/\Phi(U)\cong\FF_{q^2}\oplus\FF_{q^2},
 \qquad |\Phi(U)|=q^6.
\]
The quotient coordinates are $(u_{12},u_{23})$. An automorphism
normalizing the corresponding Borel acts on each coordinate by
$t\mapsto\alpha_i t^{p^j}$, with the same field exponent $j$;
it does not interchange the two coordinates. Its outer field
component has order dividing $2a_0$.
\end{lemma}
\begin{proof}
Write $\bar t=t^q$. Two simple-root subgroups are represented by
\[
 x(a)=1+aE_{12}-\bar a E_{45},\qquad
 y(b,c)=1+bE_{23}-\bar b E_{34}+cE_{24},
 \quad c+\bar c=-b\bar b.
\]
The remaining relative roots have representatives
\[
 w(b,c)=1+bE_{13}-\bar b E_{35}+cE_{15},\qquad
 z(a)=1+aE_{14}-\bar a E_{25},
\]
with the same trace condition for $w$; the two doubled-root
subgroups are $1+vE_{24}$ and $1+vE_{15}$ for $v+\bar v=0$.
Back substitution in the Hermitian equations gives a unique
ordered expression $x(a)y(b,c)w(d,e)z(f)$, with the two
trace constraints above. Thus $|U|=q^2q^3q^3q^2=q^{10}$.
The map
$u\mapsto(u_{12},u_{23})$ is a surjective homomorphism with
kernel of order $q^6$.

Here are the commutator checks for that kernel. Commutators within
the $y$-subgroup give all trace-zero $E_{24}$ entries, and those
within the $w$-subgroup give all trace-zero $E_{15}$ entries:
their parameters, up to sign, are
$b\bar b'-\bar b b'$, which range over the trace-zero line.
Commuting $x(a)$ with $1+vE_{24}$ gives the $z$-coordinate $av$
modulo the $E_{15}$ subgroup. Fixing $v\ne0$ gives every such
coordinate. Finally $[x(a),y(b,c)]$ has $w$-coordinate $ab$
up to sign, modulo the already obtained higher root groups.
This gives the entire kernel inside $U'$. The reverse containment
holds because the quotient is abelian. It is elementary abelian,
so $\Phi(U)=U'$, proving the orders.

Inner conjugation by $U$ is trivial on this quotient, and torus
conjugation scales each coordinate. The diagram reversal of
$A_4$ preserves the pairs $\{1,4\}$ and $\{2,3\}$ individually;
it does not interchange them. In the unitary fixed-point group
it is the $q$-power field operation, modulo these scalings.
The standard automorphism description
\cite[Theorem~4.2]{def:BMO} therefore gives the stated action.
The central quotient from $\operatorname{SU}_5$ to $\PSU_5$
has order prime to $p$, so it does not change $U$ or this layer.
\end{proof}

\Needspace{5\baselineskip}
\begin{theorem}\label{par:thm:unitary}
The simple socle $T$ is not isomorphic to $\PSU_n(q)$.
\end{theorem}
\begin{proof}
The case $n=3$ was excluded earlier. Suppose $n\ge4$ and let
$\delta=\gcd(n,q+1)$, so $d\le2\delta a_0$.
Choose the canonical parabolic with Levi parameters
$n_1,\ldots,n_r,m$, where $n=2k+m$ and $k=\sum n_i\ge1$.
Its Levi is obtained from
$\prod_i\operatorname{GL}_{n_i}(q^2)\times\operatorname{GU}_m(q)$
by the unitary determinant constraint and the central quotient.
Thus
\[
 |L|=\frac{q^{2\sum n_i^2+m^2}
                    (\prod_i\pi_{n_i}(q^2))\eta_m}{(q+1)\delta},
 \qquad z=\frac{\eta_n}{(\prod_i\pi_{n_i}(q^2))\eta_m}.
\]
The determinant constraint is surjective also when $m=0$,
since the block determinant maps by $a\mapsto a^{1-q}$.
When $m=0,r=1$, the radical is an additive group of Hermitian
matrices and is abelian, so this case is excluded.

For $m\ge3$, a semisimple unitary factor has order at most
$q^m+1$ (the torus partitions, or \cite[Lemma~2.13]{par:GMPS}).
All its torus factors are divisible by $q+1$. Hence
\[
 c\le\frac{(q^m+1)\prod_i(q^{2n_i}-1)}
                {(q+1)(q^2-1)^{r-1}}.
\]
Using $\delta\le n$, $\eta_m\ge1$, $\eta_n<3/2$, and
$q^m+1\le(28/27)q^m$, we obtain
\[
 \frac{|L|}{d z c^2}\ge
 \frac{243}{784n^2a_0}
       q^{(m-1)^2-2}B^{k-1}
 \ge\frac{2187}{19600}>\kappa.
\]
For the last comparison the minimum is at $m=3,k=1$ and
$q=3,a_0=1$. Increasing $k$ multiplies the bound by at least
$(64/9)(n/(n+2))^2>1$; increasing $m$ multiplies it by at
least $3^{2m-1}(n/(n+1))^2>1$. Finally $q^2/a_0\ge9$.

For $m=2$, the unitary torus exponents already divide $q^2-1$,
so $c\le\prod_i(q^{2n_i}-1)/(q^2-1)^{r-1}$. It follows that
\[
 \frac{|L|}{d z c^2}\ge
 \frac{qB^{k-1}}{4n^2a_0}\ge\frac3{64}>\kappa.
\]
The minimum has $k=1,n=4$; the same increment comparison
handles larger $k$.

For $m=1$ the same bound for $c$ holds. If $k\ge3$, it gives
\[
 \frac{|L|}{d z c^2}\ge
 \frac{q^{-2}B^{k-1}}{3n^2a_0}
 \ge\frac{4096}{107163}>\kappa,
\]
using $B^2/(a_0q^2)\ge4096/729$ and $n=7$ at $k=3$.
For $m=0,r\ge2,k\ge3$, instead
\[
 \frac{|L|}{d z c^2}\ge
 \frac{16B^{k-2}}{81n^2a_0q}
 \ge\frac{256}{19683}>\kappa.
\]
Here the least $k$ is three, $n=6$, and
$B/(a_0q)\ge64/27$. Again the increments in $k$ are increasing.

Only the Borels in dimensions four and five, and the
$\operatorname{GL}_2(q^2)$ Levi in dimension five, remain.
For the last of these,
\[
 |L|=|\operatorname{GL}_2(q^2)|/\delta,\quad
 c\le q^4-1,\quad h=dc\le q^4+q.
\]
The bound on $h$ is \cite[Theorem~2.16]{par:GMPS}.
The Levi-index formula gives $z<2$. Since
$|\operatorname{GL}_2(q^2)|> (64/81)q^8$,
$\delta\le5$, and $(q^4+q)(q^4-1)<(28/27)q^8$,
we obtain $|L|/(d z c^2)\ge8/105>\kappa$.

For the dimension-four Borel, if $c\mid q-1$, its simple-root
quotient $\FF_{q^2}\oplus\FF_q$ gives
$s\le |T|_p/p^e$, $e=\operatorname{ord}_c(p)\mid a_0$.
The strong order inequality forces
$q<4\delta^2a_0/225\le64a_0/225$, impossible.
If $c\nmid q-1$, the nonidentity element $x^{q-1}$ centralizes
the positive root group at the fixed middle node $\alpha_2$ of $A_3$.
Indeed, for $u$ in that root group,
\[
 h x_{\alpha_2}(u)h^{-1}
   =x_{\alpha_2}(\alpha_2(h)u),
 \qquad \alpha_2(h)\in\FF_q^\times,
\]
so the $(q-1)$st power acts trivially. The adjacent Chevalley relations
\[
 [x_{\alpha_1}(a),x_{\alpha_2}(u)]
   =x_{\alpha_1+\alpha_2}(\pm au),\qquad
 [x_{\alpha_2}(u),x_{\alpha_3}(b)]
   =x_{\alpha_2+\alpha_3}(\pm ub)
\]
show explicitly that this is the fixed-node unipotent subgroup used in
the corresponding line-parabolic Levi. Lemma~\ref{par:lem:projection}
puts $S$ in the radical of the
line parabolic, whose Levi has $m=2$ and was excluded above.

Finally consider the canonical Borel in dimension five.
By Lemma~\ref{par:lem:su5-layer}, $|U|=q^{10}$ and
$U/\Phi(U)=\FF_{q^2}\oplus\FF_{q^2}$; also $c\mid q^2-1$. Put $e=\operatorname{ord}_c(p)\mid2a_0$,
$Q=p^e$. The same faithful-constituent argument as in
Lemma~\ref{def:lem:codimension} gives $s\le q^{10}/Q$.
If $e<2a_0$, then $e\le a_0$; the order inequality gives
$q^2<4\delta^2a_0/225$, impossible.

Suppose $e=2a_0$. If $S\le\Phi(U)$ then $s\le q^6$, and the
order inequality would give $q^4<4\delta^2a_0/225$, again
impossible. Otherwise a nonzero simple-root projection of
$S$ has a faithful $C_c$-character of order $c$.
The generator of $C_h$ normalizes the canonical Borel.
By Lemma~\ref{par:lem:su5-layer} it acts on each coordinate
by $t\mapsto\alpha_i t^{p^j}$. If $\lambda$ is the eigenvalue
of $C_c$ on a nonzero projection of $S$, fixed-point freeness
gives $|\lambda|=c$. Centralization on that projection yields
$\lambda^{p^j}=\lambda$, hence $e\mid j$. Since $e=2a_0$,
the outer field image is trivial and $d\le\delta$.
Now the order inequality gives
$|T|/|T|_p^2<\delta q^4/225$, contradicting
$|T|/|T|_p^2>q^4/(2\delta)$ and $\delta\le5$.
\end{proof}

\subsection{Exceptional groups of absolute rank at least four}

\begin{lemma}[An exponent bound for a parabolic torus]\label{par:lem:torus}
Let a rational parabolic in a split group have a Levi root
subsystem of rank $r$. The exponent of any rational maximal
torus of the Levi is at most $(q-1)(q+1)^r$.
For ${}^2E_6(q)$ it is at most $(q^2-1)(q+1)^r$.
The bounds continue to hold after taking central quotients.
\end{lemma}
\begin{proof}
On the cocharacter lattice, write the torus Frobenius as
$q\theta$, where $\theta$ is in the Weyl group of the Levi,
possibly composed with the diagram involution. On the rational
span of its roots let $f(t)$ be the characteristic polynomial
of $\theta^{-1}$. In the split case
$(t-1)f(t)$ annihilates $\theta^{-1}$; in the twisted case
$(t^2-1)f(t)$ does. This holds on the rational lattice, hence
on the integral lattice. Evaluation at $q$ therefore
annihilates the cokernel of $q-\theta^{-1}$, which describes
the rational torus points. All $r$ roots of $f$ have absolute
value one, so $|f(q)|\le(q+1)^r$.
\end{proof}

\Needspace{6\baselineskip}
\begin{theorem}\label{par:thm:exceptional-large}
The simple socle $T$ is not isomorphic to any of
$F_4(q)$, $E_6(q)$, ${}^2E_6(q)$, $E_7(q)$, or $E_8(q)$.
\end{theorem}
\begin{proof}
Choose a proper parabolic with $S$ in its radical. Its coset
action has rank at least three: by Bruhat decomposition the
rank is the number of double cosets of the corresponding
parabolic Weyl subgroup, and none of these exceptional types
has a parabolic action of rank two.

First take a split type of absolute rank $\rho$. Let $r$ and
$R_L$ be the rank and number of positive roots of its Levi
subsystem. Then $0\le r\le\rho-1$, $R_L\ge r$.
Let $\delta$ be the rational central factor and write
$d\le o\delta a_0$, where
\[
 (\rho,\delta\hbox{ upper bound},o)=
 (4,1,1),\ (6,3,2),\ (7,2,1),\ (8,1,1)
\]
in the order $F_4,E_6,E_7,E_8$.
The Levi order is at least
$q^{\rho+2R_L}(1-q^{-1})^\rho/\delta$: its Borel torus
has order $(q-1)^\rho/\delta$, and the highest-degree term
of the flag polynomial has degree $R_L$.
The normalized total group-order factor is less than two,
so $z<2/(1-q^{-1})^\rho$. Lemma~\ref{par:lem:torus} gives
\begin{equation}\label{par:eq:exceptional-bound}
 \frac{|L|}{d z c^2}>
 \frac{q^{\rho-2+2(R_L-r)}(1-q^{-1})^{2\rho-2}}
      {2o\delta^2a_0(1+q^{-1})^{2r}}.
\end{equation}
For any integer $t\ge1$, $q^t/a_0\ge3^t$. Moreover,
$1-q^{-1}$ and $(1+q^{-1})^{-1}$ are increasing functions of $q$,
so every remaining $q$-dependent factor on the right improves as $q$
increases.
For $F_4,E_7,E_8$, replace $R_L-r$ by zero and $r$ by
$\rho-1$. This gives respectively
\[
 \frac9{128},\qquad \frac{243}{32768},\qquad
 \frac{729}{32768},
\]
all larger than $\kappa$.

For $E_6$, a simple-root subdiagram has at most three connected
components, since the independence number of the diagram is
three. A connected component of rank $t$ has at least $2t-1$
positive roots. Thus $R_L-r\ge\max(0,r-3)$.
In \eqref{par:eq:exceptional-bound} the minimum over $r$ is
therefore at $r=3$, and is $1/144>\kappa$.

For ${}^2E_6(q)$, still $\delta\le3$ and $d\le2\delta a_0$.
The normalized Borel torus factor is now
$H_0=(1-q^{-1})^2(1-q^{-2})^2$.
Every Levi has normalized order at least $H_0$, and
$z<2/H_0$. The twisted torus bound gives
\[
 \frac{|L|}{d z c^2}>
 \frac{q^{2+2(R_L-r)}(1-q^{-1})^4(1-q^{-2})^2}
      {4\delta^2a_0(1+q^{-1})^{2r}}.
\]
The same root-count argument applies. Its lower bound at
$q=3,a_0=1,r=3,\delta=3$ is again $1/144$.
Every case contradicts Corollary~\ref{par:cor:test}.
\end{proof}

\subsection{The split exceptional group of rank two}

\begin{lemma}[The bad-prime first layer of $G_2$]\label{par:lem:g2-layer}
For $q=3^{a_0}>3$, let $U$ be a positive Sylow subgroup of
$G_2(q)$, with simple roots $\alpha$ short and $\beta$ long.
Then $\Phi(U)=U'$ is the product of the four nonsimple positive
root groups, and $U/\Phi(U)\cong\FF_q^2$ on the simple roots.
\end{lemma}
\begin{proof}
Use the root parametrizations of \cite[Section~6]{par:vB}.
The characteristic-three relations give
\begin{align*}
 [x_\alpha(t),x_{\alpha+\beta}(u)]&=x_{2\alpha+\beta}(tu),\\
 [x_\beta(t),x_{3\alpha+\beta}(u)]&=x_{3\alpha+2\beta}(tu).
\end{align*}
Both groups on the right are central in $U$; their product
$Z_0$ thus lies in $U'$. Modulo $Z_0$, the coordinates of
$[x_\alpha(t),x_\beta(u)]$ on $X_{\alpha+\beta}$ and
$X_{3\alpha+\beta}$ are $(-tu,t^3u)$. Choose $a\in\FF_q$
with $a^2\ne1$ and $a\ne0$. The two choices $t=1,a$ give
the lines generated by $(-1,1)$ and $(-a,a^3)$. Their
determinant is $a-a^3\ne0$, so these commutators span both
coordinates. Hence all nonsimple root groups lie in $U'$.
Conversely the simple-root coordinates give an elementary
abelian quotient of order $q^2$. This proves $U'=\Phi(U)$.
The restriction $q>3$ is essential for this argument and is
checked in its application below.
\end{proof}

\begin{theorem}\label{par:thm:g2}
The simple socle $T$ is not isomorphic to $G_2(q)$.
\end{theorem}
\begin{proof}
Write $q=p^{a_0}$, $R=q^6$. Here $|T|/R^2>q^2/2$ and
$d\le2a_0$. Let $z_0$ generate $C_h$, so $z_0$ centralizes
$C_c$ and its outer image generates $G/T$.
By Lemma~\ref{par:lem:canonical}, the canonical parabolic for $S$ is
normalized by $z_0$, including in characteristic three.

If the parabolic is maximal, its Levi is $\operatorname{GL}_2(q)$,
\[
 |L|=q(q-1)^2(q+1),\qquad
 z=\frac{q^6-1}{q^5(q-1)}<\frac q{q-1}.
\]
The image $x$ of a generator of $C_c$ is semisimple of order
$c$, whence $c\mid q^2-1$. Put $e=\operatorname{ord}_c(p)$,
so $e\mid2a_0$ and $c<p^e$.
Because the two maximal parabolic classes of $G_2(q)$ are exchanged by
the exceptional graph automorphism in characteristic three, invariance of
this maximal parabolic rules out a graph component in $z_0$. If $p^j$ is
its field component, the two eigenvalues of $x$ are permuted by $p^j$.
Their orders have least common multiple $c$, so $c\mid p^{2j}-1$ and
$e\mid2j$. Hence the field outer order, and therefore $d$, is at most
$2a_0/e$. Putting $t=2a_0/e\in\mathbb N$, we obtain
\[
 dc^2<t p^{2e}\le p^{2et}=q^4<2q^4.
\]
The elementary inequality used here is $t\le p^{2e(t-1)}$.
Consequently
\[
 \frac{|L|}{d z c^2}>
 \tfrac12(1-q^{-1})^3(1+q^{-1})\ge\frac4{27}>\kappa.
\]

Suppose instead that the canonical parabolic is a Borel
$U\rtimes H$. Then $c\mid q-1$, and $c\ge112$ implies $q>3$.
For these $q$ in odd characteristic,
$U/\Phi(U)=\FF_q\oplus\FF_q$ on the simple roots.
For $p\ge5$ this is the usual simple-root calculation;
the only bad odd prime, $p=3$, is covered explicitly by
Lemma~\ref{par:lem:g2-layer}.

Set $e=\operatorname{ord}_c(p)\mid a_0$ and $Q=p^e$.
The faithful-constituent argument gives $s\le q^6/Q$ because
$S<U$. If $S\le\Phi(U)$, then $s\le q^4$; the strong order
inequality would imply $q^2<4a_0/225$, impossible.
Otherwise at least one nonzero simple-root projection of $S$
has a $C_c$-character of order $c$.
On the simple-root quotient, $z_0$ is diagonal--field, possibly followed
by the interchange induced by the exceptional graph automorphism. Since
$z_0$ centralizes $C_c$, the character of order $c$ is fixed up to this
interchange; hence $e\mid2j$. The field image has order at most $2a_0/e$
and the graph image has order at most two, so $d\le4a_0/e$. Writing
$a_0=et$, the elementary inequality $t\le p^{2e(t-1)}$ gives
\[
 dQ^2\le4t p^{2e}\le4q^2.
\]
Therefore
\[
 \frac{|T|}{R^2}<\frac{d Q^2}{225}
 \le\frac{4q^2}{225},
\]
contradicting $|T|/R^2>q^2/2$.
\end{proof}

\subsection{The triality groups}

\begin{lemma}[The small Levi in triality]\label{par:lem:triality-radical}
In $T={}^3D_4(q)$, $q$ odd, let $P_0=V\rtimes L$ be the
maximal parabolic whose Levi has derived group
$\operatorname{SL}_2(q)$. Then
\[
 |V|=q^{11},\quad |L|=q(q^2-1)(q^3-1),\quad
 |\Phi(V)|=q^5,\quad \Phi(V)=V'\text{ is abelian},
 \quad V/\Phi(V)\cong\FF_{q^3}^2.
\]
The Levi acts faithfully and $\FF_{q^3}$-linearly on this
last space; its projective image is contained in
$\operatorname{PGL}_2(q)$. Field automorphisms act by the
corresponding powers of $p$ on the coordinates.
\end{lemma}
\begin{proof}
Use the relative simple roots $\alpha$ (short) and $\beta$
(long), with short-root parameters in $\FF_{q^3}$ and
long-root parameters in $\FF_q$. The radical omits $X_\beta$.
Its first layer consists of $X_\alpha,X_{\alpha+\beta}$,
the next of $X_{2\alpha+\beta}$, and the last of
$X_{3\alpha+\beta},X_{3\alpha+2\beta}$.
The commutator onto the middle layer includes
$uv^q+u^qv$. For $v=1$ the map is $u\mapsto u+u^q$.
A kernel element satisfies $u^{q^3}=-u$, hence $u=0$ in odd
characteristic; the map is therefore bijective.
The commutators onto the last two layers are trace pairings,
which are nondegenerate, and hence surjective,
because $\FF_{q^3}/\FF_q$ is separable. This remains true when
$p=3$; $\operatorname{Tr}(1)=0$ does not make the trace map zero.
These relations show that $V'$ is
exactly the product of the last three root groups. They
commute with each other and the quotient is elementary
abelian, proving the statements about $\Phi(V)$.

The long-root $\operatorname{SL}_2(q)$ acts naturally on
the first two coordinates. For torus parameters
$t_1\in\FF_{q^3}^{\times}$, $t_2\in\FF_q^{\times}$,
the eigenvalues are
\[
 t_1^{-2}t_2,\qquad t_1^{q^2+q-1}t_2^{-1}.
\]
Their ratio is $t_2^2/N_{\FF_{q^3}/\FF_q}(t_1)$,
which lies in $\FF_q^{\times}$. If both equal one,
then $t_1^2=t_2\in\FF_q$, hence $t_1\in\FF_q$; substitution
gives $t_1=t_2=1$. Together with the natural root-group
action this identifies the Levi image with
$\FF_{q^3}^{\times}\operatorname{GL}_2(q)$ in
$\operatorname{GL}_2(q^3)$, a group of the stated Levi order.
In particular its action is faithful.
All assertions follow from the root and torus relations in
\cite[Tables~2.2--2.3 and Sections~3.4, 4.7]{par:Himstedt};
see also the six-root commutator relations (3)--(8) immediately preceding
Proposition~4.15 in \cite{def:BMO}.
\end{proof}

\begin{theorem}\label{par:thm:triality}
The simple socle $T$ is not isomorphic to ${}^3D_4(q)$.
\end{theorem}
\begin{proof}
Here $|T|>q^{28}/2$, $|T|_p=q^{12}$, and $d\le3a_0$.
Choose the canonical parabolic for $S$, which is normalized
by the generator $z_0$ of $C_h$.
There are two maximal parabolic types. For the type with
derived Levi group $\operatorname{SL}_2(q^3)$, one has
\[
 |L|=q^3(q^6-1)(q-1)>q^{10}/2,\quad z<2,\quad
 c\le(q^3+1)(q-1)<2q^4.
\]
Thus $|L|/(d z c^2)>q^2/(48a_0)\ge3/16>\kappa$.

Consider next the parabolic $P_0$ of
Lemma~\ref{par:lem:triality-radical}. Its normalized index is
\[
 z=(1+q^{-4}+q^{-8})(1+q^{-3})<2.
\]
If $S=V$, then $N_T(S)\ge P_0$ and
$\mu\le[T:P_0]/(q^{11}+1)<2$, contrary to $\mu>148$.
Thus $S<V$. Its image $A$ in $V/\Phi(V)$ is proper by the
Frattini property and nonzero because $\Phi(V)$ is abelian.

Place $C_c$ in the Levi. Its two eigenvalues lie in
$\FF_{q^6}$. A nonzero irreducible constituent occurring in
$A$ is faithful by Lemma~\ref{def:lem:free}; its eigenvalue
therefore has order $c$. Put $e=\operatorname{ord}_c(p)$ and
$Q=p^e$, so $e\mid6a_0$ and $c<Q$.
If the eigenvalues are not in $\FF_{q^3}$, they are conjugate
over that field, and both have order $c$. If they are split,
each coordinate has dimension $3a_0$; a nonfaithful coordinate
has zero projection from $A$. In either case a proper $A$
has codimension at least $e$. Hence $s\le q^{11}/Q$.
The value $e=6a_0$ would make the entire first layer
irreducible, contrary to $0<A<V/\Phi(V)$. Therefore
$e\le3a_0$.

If $e\le2a_0$, the strong order inequality gives
$|T|<3a_0q^{26}/225$, contradicting $|T|>q^{28}/2$.
Otherwise divisibility by $6a_0$ forces $e=3a_0$.
Let $p^j$ be the outer field component of $z_0$.
On $V/\Phi(V)$ it is an $\FF_{q^3}$-linear map followed by
$p^j$. Since it centralizes $C_c$, it permutes its two
eigenvalues. Squaring this permutation fixes the faithful
eigenvalue, so $3a_0=e\mid2j$.
The outer group is cyclic of order $3a_0$, with no diagonal
kernel. Therefore $d\le2$, and
\[
 |T|<\frac{2q^{28}}{225}<\frac{q^{28}}2,
\]
again a contradiction. This argument applies whenever
$S\le V$, $I\le P_0$, and $z_0$ normalizes $P_0$.

Finally suppose the canonical parabolic is a Borel $U\rtimes H$.
Its torus exponent divides $q^3-1$. Its simple-root quotient
is $\FF_{q^3}\oplus\FF_q$: the commutator of $X_\alpha(t)$
and $X_\beta(u)$ has $X_{\alpha+\beta}$ coordinate $\pm tu$;
the remaining three nonsimple root groups are obtained by the
surjective maps in Lemma~\ref{par:lem:triality-radical}.
Thus $U'=\Phi(U)$ has order $q^8$, also when $p=3$.
If $c\mid q-1$, the faithful-constituent argument gives
$s\le q^{12}/Q$, with $Q=p^e\le q$.
The order inequality would imply $q^2<6a_0/225$, impossible.
If $c\nmid q-1$, the nonidentity element $x^{q-1}$ of $C_c$
centralizes the long simple-root group. By
Lemma~\ref{par:lem:projection}, $S$ lies in the radical of
$P_0$. The normalizer of the canonical Borel in $\Aut(T)$
preserves each of the two relative parabolic types: field
operations preserve both, and the absolute triality operation
permutes the three outer $D_4$ nodes while fixing the central
node. Thus it fixes both relative nodes. Inner Borel and torus
operations also preserve both standard maximal parabolics.
Consequently $z_0$ normalizes $P_0$. The preceding argument applies and finishes
the exclusion.
\end{proof}

\begin{theorem}[Completion of the socle exclusion]
\label{par:thm:closure}
Under $(\Sigma)$ and the conclusions of Sections~\ref{scr:sec:screening}
and~\ref{def:sec:new}, no simple socle $T$ can occur.
\end{theorem}
\begin{proof}
The classical families are excluded by
Theorems~\ref{par:thm:linear}, \ref{par:thm:bc},
\ref{par:thm:sp4}, \ref{par:thm:d}, and \ref{par:thm:unitary}.
The odd-characteristic exceptional families of relative rank
at least two are excluded by
Theorems~\ref{par:thm:exceptional-large}, \ref{par:thm:g2},
and \ref{par:thm:triality}. The remaining rank-one families
were excluded in the preceding sections. These exhaust
the defining-characteristic possibilities.
\end{proof}

\section{Completion of the proof}\label{sec:current-status}

\begin{proof}[Proof of Theorem~\ref{thm:main}]
Assume that the theorem is false, and choose a counterexample $(X,G)$ of
minimum genus. Proposition~\ref{prop:two-orbits} gives a rational quotient
$X/G$, one wild short orbit $\Delta=G(P)$ and one tame short orbit
$\Omega$.

Let $S=G_P^{(1)}$. If $g(X/S)>0$, Lemma~\ref{lem:y1},
Proposition~\ref{prop:positive-genus}, and
Corollary~\ref{cor:large-parameters} give the numerical package
$(\Sigma)$. Suppose that $X/S$ is rational. Case~A is excluded by
Corollary~\ref{cor:caseA}. In Case~B, the configuration with exactly one
additional short $S$-orbit is excluded by
Corollary~\ref{aff:cor:one-orbit}; if there are at least two additional
short $S$-orbits, Proposition~\ref{prop:several-short} gives $(\Sigma)$.
Thus every minimum-genus counterexample satisfies $(\Sigma)$, as stated in
Proposition~\ref{reduction:prop:sigma}. In particular
\[
 S_P^{(2)}=S'=\Phi(S)\ne1.
\]

Corollary~\ref{cor:minimal-pcore} and
Theorem~\ref{thm:no-tame-core} give
\[
 O_p(G)=O_{p'}(G)=1,
\]
and Corollary~\ref{cor:radical-free} excludes every nontrivial soluble
normal subgroup. By Theorem~\ref{thm:no-disjoint-minimal}, every minimal
normal subgroup meets $S$. Proposition~\ref{prop:normal-wild-transitive}
then makes it transitive on $\Delta$ and gives a soluble quotient.
Corollary~\ref{cor:residual-quasiprimitive} therefore yields a unique
minimal normal subgroup
\[
 N=T^k=\operatorname{soc}(G),
\]
with $T$ nonabelian simple.

Proposition~\ref{prop:socle-contains-S} gives $S\le N$, and
Corollary~\ref{cor:two-point-socle-cover} gives a two-point cover of
$\PP^1$. The geometric direct-product argument in
Proposition~\ref{prop:residual-socle-simple} gives $k=1$. Hence
\[
 T\le G\le\Aut(T),
\]
$G/T$ is cyclic of order prime to $p$, and the two inertia groups in
$X\to X/T$ are
\[
 S\rtimes C_c,\qquad C_b,
\]
with $c>b\ge111$.

Section~\ref{scr:sec:screening} excludes alternating and sporadic socles
and all groups of Lie type in characteristic different from $p$.
Consequently $T$ is a simple group of Lie type in characteristic $p$.
Section~\ref{def:sec:new} shows that $S$ is not Sylow in $T$ and excludes
$\PSU_3(q)$, ${}^2G_2(q)$, and $\PSL_3(q)$. Finally,
Section~\ref{par:sec:new} excludes all remaining defining-characteristic
families: the linear, unitary, symplectic and orthogonal groups, and the
exceptional groups $G_2$, ${}^3D_4$, $F_4$, $E_6$, ${}^2E_6$, $E_7$ and
$E_8$. This exhausts the finite simple groups of Lie type in odd defining
characteristic.

No simple socle remains, a contradiction. Therefore no counterexample
exists, and the theorem follows.
\end{proof}

\section*{Acknowledgments}
The author was partially supported by CNPq (grant no.~302774/2025-4), FAEPEX (grant no.~3485/25), and FAPESP (grant no.~2024/00923-6).

\section*{Declaration on the use of generative AI}
During the preparation of this manuscript, the author used ChatGPT (OpenAI) as an auxiliary tool for language editing, improving exposition, and checking intermediate mathematical arguments. All statements, proofs, computations, citations, and conclusions were subsequently reviewed and verified by the author, who takes full responsibility for the content of the paper.

\end{document}